\documentclass[11pt]{amsart}
\usepackage[english]{babel}
\usepackage{amsthm,amsmath}
\usepackage{amssymb,bbm,esint}
\numberwithin{equation}{section}
\usepackage{appendix}
\usepackage{geometry}
\usepackage[dvipsnames]{xcolor}
\usepackage[colorlinks=true,allcolors=ForestGreen]{hyperref}
\usepackage{amsthm}
\usepackage[capitalise]{cleveref}
\DeclareMathOperator*{\esssup}{ess\,sup}
\newtheorem{theorem}{Theorem}[section]
\newtheorem{definition}[theorem]{Definition}
\newtheorem{lemma}[theorem]{Lemma}
\newtheorem{remark}[theorem]{Remark}
\newcommand{\hn}{{\mathbb{H}^{\text{N}}}}
\newcommand{\gaglifrac}[1][u]{{\frac{|#1(\xi)-#1(\eta)|^p}{|\eta^{-1}\circ\xi|^{Q+sp}}}}
\newcommand{\intbr}[1][r]{\int_{B_{#1}}}
\newcommand{\gaglifractwo}[4]{{\frac{|#1-#2|^p}{|#3^{-1}\circ#4|^{Q+sp}}}}

\newcommand{\intgaglifrac}[2][u]{{\int_{#2}\int_{#2}\frac{|#1(\xi)-#1(\eta)|^p}{|\eta^{-1}\circ\xi|^{Q+sp}}\,d\xi\,d\eta}}
\newcommand{\intgaglifractwo}[3][u]{{\int_{#3}\int_{#2}\frac{|#1(\xi)-#1(\eta)|^p}{|\eta^{-1}\circ\xi|^{Q+sp}}\,d\xi\,d\eta}}
\newcommand{\intgaglifracflip}[3][u]{{\int_{#3}\int_{#2}\frac{|#1(\xi)-#1(\eta)|^p}{|\eta^{-1}\circ\xi|^{Q+sp}}\,d\eta\,d\xi}}
\newcommand{\dhn}[2]{|#2^{-1} \circ #1|}
\newcommand{\gag}[1][u]{{\frac{|#1(\xi,t)-#1(\eta,t)|^{p-2}\Big(#1(\xi,t)-#1(\eta,t)\Big)}{|\eta^{-1}\circ \xi|^{Q+sp}_{\hn}}}}

\def\Yint#1{\mathchoice
    {\YYint\displaystyle\textstyle{#1}}%
    {\YYint\textstyle\scriptstyle{#1}}%
    {\YYint\scriptstyle\scriptscriptstyle{#1}}%
    {\YYint\scriptscriptstyle\scriptscriptstyle{#1}}%
      \!\iint}
\def\YYint#1#2#3{{\setbox0=\hbox{$#1{#2#3}{\iint}$}
    \vcenter{\hbox{$#2#3$}}\kern-.51\wd0}}
\def\longdash{{-}\mkern-3.5mu{-}} 

\def\fiint{\Yint\longdash}

\usepackage{longtable}
\title[Local Boundedness with an Optimal Tail]{Nonlocal Quasilinear Parabolic Equations in Heisenberg Group : Local Boundedness with an Optimal Tail}
\author{Debraj Kar and Vivek Tewary}
\address{Department of Mathematics,  University of Kalyani and School of Interwoven Arts and Sciences, Krea University}
\email{debrajmath22@klyuniv.ac.in, vivek.tewary@krea.edu.in}

\begin{document}
	\begin{abstract}
		We prove local boundedness for a quasilinear parabolic equation on the Heisenberg group
		\[ 
			\partial_t u(\xi,t) + \text{p.v.}\int_{\hn} \frac{|u(\xi,t)-u(\eta,t)|^{p-2}(u(\xi,t)-u(\eta,t))}{|\eta^{-1}\circ \xi|^{Q+sp}}  \,d\eta = 0,
		\] with optimal regularity assumption on the tail term. We also prove interpolation inequalities and an extension theorem for fractional Sobolev spaces on the Heisenberg group.
	\end{abstract}

	\maketitle

\setcounter{tocdepth}{1}
\tableofcontents

	\section{Introduction}
	
	The aim of this paper is to prove local boundedness of the equation
	\begin{align}\label{maineq}
						\partial_t u(\xi,t) + \text{p.v.}\int_{\hn} K(\xi,\eta,t)|u(\xi,t)-u(\eta,t)|^{p-2}(u(\xi,t)-u(\eta,t))  \,d\eta = 0,
	\end{align} on the Heisenberg group $\hn$ where the kernel $K:\hn\times\hn\times\mathbb{R}\to\mathbb{R}$ satisfies
	\begin{align}
		\frac{\lambda}{\dhn{\xi}{\eta}^{Q+sp}} \leq K(\xi,\eta,t)\leq \frac{\Lambda}{\dhn{\xi}{\eta}^{Q+sp}}
	\end{align} for some numbers $\lambda,\Lambda>0$ a.e. $\xi,\eta,t\in \hn\times\hn\times\mathbb{R}$. The equation is modelled on the parabolic fractional $p$-Laplace equation and no regularity is imposed on the coefficients apart from boundedness.  Importantly, we impose optimal regularity requirement on the Tail, viz., $u\in L^{p-1}(0,T; L^{p-1}_{sp}(\hn))$. Previous work on quasilinear nonlocal parabolic equations imposed $L^\infty$ requirement in time, viz., \cite{Strom2019,Ding2021, PT23, Liao2024, APT8}. This changed in \cite{Kassweid2024} where boundedness estimates with an optimal tail were obtained in the linear case. A H\"older regularity result for quasilinear equations was found with an optimal tail in \cite{BK24}. We are able to prove local boundedness estimates for quasilinear nonlocal parabolic equations with an \emph{optimal tail} in the Heisenberg group setting. In the Euclidean case, the technique of \cite{Kassweid2024} is applied to local boundedness of parabolic p-Laplace equations in the preprint \cite{kumagai2024local}, however the solutions appear to be more regular than usual.

    The result is achieved by the concurrent use of the techniques proposed in \cite{Kassweid2024} and \cite{BK24}. Particularly, in \cite{Kassweid2024}, the tail is made part of the level of the functions in the De Giorgi iteration whereas in \cite{BK24}, the tail is treated as a source term. These techniques combined allow us to obtain the optimal regularity on the tail as expected from the existence theory of solutions to \cref{maineq}. For the optimal condition on the tail dictated by the existence theory, see \cite[Theorem A.2]{brasco2021continuity}. Apart from the concurrent use of these two techniques, we also implement and establish other important results that were not present in the literature. Since parabolic equations demand regularization in time, we prove the convergence theorems and estimates associated with regularization in time in \cref{SectionRegularize}. Due to the presence of a tail-type term $g(t)$ as a level in the test function, we prove the Caccioppoli inequality in complete detail in \cref{Caccioppolli}. Further, we could not find Gagliardo-Nirenberg interpolation inequality with the complete range of exponents required to establish regularity theory for parabolic equations, either in the full space or in domains for the Heisenberg group. Therefore, we establish these results too as part of the appendix to this paper.  With the auxilliary results on regularization and interpolation developed here as well as the Caccioppoli inequality with the good term, H\"older regularity with optimal tail follows in a straightforward manner as in \cite{BK24}.

    As remarked earlier, an important feature of this paper is that we establish a number of interpolation inequalities - one interpolating between Sobolev spaces and BMO (\cref{bmointerpol}) and another interpolating Sobolev spaces and H\"older spaces (\cref{interlope1}). The BMO interpolation inequality has appeared in the Euclidean setting in \cite{AzBed2015} We also prove an extension theorem for fractional Sobolev spaces for regular domains in Heisenberg group. While an abstract version of such a result exists \cite{DachunYang2022}, it is not known whether the abstract Sobolev spaces on homogeneous spaces defined using the Carnot-Caratheodory metric coincide with the fractional Sobolev spaces defined using merely the Gagliardo norm. Previously, the extension theorem for the usual Sobolev spaces on domains in the Heisenberg group was proved in the PhD thesis of \cite{Nhieu1996}.

    In this passage, we point out some related literature. Important results on regularity theory of parabolic p-Laplace type equations may be found in \cite{Strom2019,brasco2021continuity,Ding2021,agnid2023,APT8,Liao2024,Liaomodulus2024,Kass2024AnalPDE,Kassweid2024,liao2024timeinsensitivenonlocalparabolicharnack, diening2025, byundeepak2024, garain2025}. Some papers on the regularity theory for the p-Laplace type equations in the Heisenberg group setting are \cite{PicCalVar2022,MPPP23,FangZhang2024,Fang_Zhang_Zhang_2024}.

    In the rest of the Introduction, we set up the notation, state the main theorem and several auxilliary results required in the paper.

	
	\subsection{About the Heisenberg group}
	The Heisenberg group is the set $\mathbb{R}^{2N+1}$ equipped with the following group multiplication for $$\xi=(z,s)=(\xi,y,s)=(x_1, x_2, \ldots, x_N, y_1, y_2,\ldots, y_N, s)$$ and $$\xi':=(z',s')=(x',y',s')=({x'}_1, {x'}_2, \ldots, {x'}_N,{y'}_1, {y'}_2,\ldots, {y'}_N, s')$$ 
	\begin{align*}
		\xi\circ\xi' = (x+x', y+y', s+s'+2\langle x',y\rangle-2\langle x,y'\rangle).
	\end{align*}
	One defines a one-parameter group of automorphisms on $\hn$ as $\Phi_\lambda(\xi,y,s)=(\lambda x,\lambda y,\lambda^2 s)$ so that it is said to have a \emph{homogeneous dimesion} of $Q=2N+2$. A \emph{homogeneous norm} $d_0:\hn\to[0,\infty]$ is a function satisfying
	\begin{enumerate}
		\item $d_0(\Phi_\lambda(\xi))=\lambda d_0(\xi)$ for any $\xi\in\hn$ and $\lambda>0$.
		\item $d_0(\xi)=0$ if and only if $\xi=0$.
	\end{enumerate}
	One defines the standard \emph{homogeneous norm} on $\hn$ by 
	\[
		|\xi| = (|z^4|+|s|^2)^{\tfrac{1}{4}}.
	\] It is well known that all homogeneous norm on $\hn$ are equivalent. For a center $\xi_0\in\hn$ and radius $R>0$, we define the ball $B(\xi_0,R):=\{\xi\in\hn: |\xi_0^{-1}\circ\xi|<R\}$. 
	Moreover, Heisenberg group with any homogeneous norm $d_0$ satisfies a pseudo-triangle inequality in the form of the lemma below:
\begin{lemma}
	Let $d_0$ be a homogeneous norm on $\hn$. Then there exists a constant $c>0$ such that for all $\xi,\eta\in\hn$, it holds that
	\begin{enumerate}
		\item $d_0(\xi\circ\eta)=c(d_0(\xi)+d_0(\eta))$,
		\item $d_0(\xi\circ\eta)=\frac1c d_0(\xi)-d_0(\eta^{-1})$,
		\item $d_0(\xi\circ\eta)=\frac1c d_0(\xi)-c d_0(\eta)$.
	\end{enumerate}
\end{lemma}
	\begin{remark}
		It is known that for the standard homogeneous norm on $\hn$, the constant $c$ above is $1$ so we may assume that $\hn$ is a metric space with metric given by $|\cdot^{-1}\circ\cdot|$.
	\end{remark}
	We also note that the Haar measure on $\hn$ is the same as the Lebesgue measure on $\mathbb{R}^{2N+1}$ and it satisfies the measure doubling property, viz.,
	\begin{align*}
		\left|B(\xi_0,2R)\right|\leq C \left|B(\xi_0,R)\right|
	\end{align*} for any $\xi_0\in\hn$ and $R>0$. These properties make $(\hn,|\cdot^{-1}\circ\cdot|,\mathcal{L}^{2N+1})$ into a $Q$-regular measure metric space with a doubling measure.

For $z_0=(\xi_0,t_0)\in\hn\times\mathbb{R}$, let $\mathcal{Q}_{r}(z_0)$ denote the cylinder $\mathcal{Q}_r(z_0):=B_r(\xi_0)\times (t_0-\rho^{sp},t_0)$.

    \subsection{Definitions and Main Results}

    We define the fractional Sobolev space $W^{s,p}(\Omega)$ as
	\begin{align*}
		W^{s,p}(\Omega):=\left\{u\in L^p(\Omega):\int_\Omega\int_\Omega\frac{|u(\xi)-u(\eta)|^p}{|\eta^{-1}\circ\xi|^{Q+sp}}\,d\xi\,d\eta<\infty\right\}.
	\end{align*}
Now we define a tail space $L^m_\gamma(\hn)$ on Heisenberg Group as
\begin{align*}
    L^m_\gamma(\hn):=\Bigg\{v\in L^m_{\textup{loc}}(\hn):\int_\hn\frac{|v|^m}{1+|x|^{Q+\gamma}}dx<+\infty\Bigg\},m>0,\gamma>0
\end{align*}
In the following, we define the nonlocal parabolic tail as
\begin{align*}
    \mbox{Tail}_{p-1}(v;\mathcal{Q}_\tau)=\mbox{Tail}_{p-1}(v;\xi_0,r,\tau)=\Bigg(\fint_{t_0-\tau}^{t_0}r^{sp}\int_{\hn\setminus B_r(\xi_0)}\frac{|v(\eta,t)|^{p-1}}{|\eta-\xi_0|^{Q+sp}}d\eta dt\Bigg)^\frac{1}{p-1}
\end{align*}.
In particular, we write
\begin{align*}
   \mbox{\~{Tail}}_{p-1}(v:\mathcal{Q}_\tau)=\Bigg(\int_{t_0-\tau}^{t_0}\int_{\hn\setminus B_r(\xi_0)}\frac{|v(\eta,t)|^{p-1}}{|\eta-\xi_0|^{Q+sp}}d\eta dt\Bigg)^\frac{1}{p-1}
\end{align*}
We also use $\textup{Tail}^{p-1}(v(t);\rho,\xi_0):=\left( r^{sp}\int_{\hn\setminus B_r(\xi_0)}\frac{|v(\eta,t)|^{p-1}}{|\eta-\xi_0|^{Q+sp}}d\eta\right)^{\frac{1}{p-1}}$ to denote a time-dependent tail. Additionally we also note that a function $v\in L^{p-1}(I;L^{p-1}_{sp}(\hn)$ if
\begin{align*}
    \mbox{Tail}_{p-1}(v;\xi_0,r,I)<+\infty
\end{align*}
We will now remind the definition of a weak solution to the problem (\ref{maineq})
\begin{definition}
    (Weak solution) A function $u\in \hn\times [0,T]\rightarrow  \mathbb{R}$ satisfying
    \begin{equation*}
        u\in L^{p-1}_{\textup{loc}}(0,T;L^{p-1}_{sp}(\hn))\cap L^p_{\textup{loc}}(0,T;W^{s,p}_{\textup{loc}}(\Omega))\cap C_{\textup{loc}}(0,T;L^2_{\textup{loc}}(\Omega))
     \end{equation*}
     is a weak sub(super) solution to the problem (\ref{maineq}) if it satisfies
     \begin{align*}
         &-\int_{T_1}^{T_2}\int_\Omega u\partial_t\varphi\,dt\\
         &\quad+\int_{T_1}^{T_2}\iint_{\hn\times\hn}|u(\xi,t)-u(\eta,t)|^{p-2}(u(\xi,t)-u(\eta,t))(\varphi(\xi,t)-\varphi(\eta,t))\\
         &\qquad\qquad\qquad\qquad\times K(\xi,\eta,t)d\xi\,d\eta\,dt\leq (\geq)\,0
     \end{align*}
     for any $\varphi \in W^{1,2}(T_1,T_2;L^2(\Omega)\cap L^p(T_1,T_2;W^{s,p}(\Omega))$ with spatial compact support contained in $\Omega$ and $[T_1,T_2]\Subset [0,T]$.
\end{definition}
Now we are going to present our main result i.e. local boundedness with optimal tail
\begin{theorem}\label{mainresult}
    Let $p\in \left(\frac{2Q}{Q+2s},\infty\right)$. Let $u$ be a weak subsolution of the problem (\ref{maineq}). Further, we also assume that $u\in L^{p-1}_{\textup{loc}}(0,T;L^{p-1}_{sp}(\hn))$. Then the following estimate holds true.
    \begin{equation}\label{eq1.4*}
\begin{aligned}
    \sup_{\mathcal{Q}_{\sigma\rho}(z_0)} u_+(\xi,t) &\leq \frac{c}{(1-\sigma)^{c_1}}\Bigg[\left(\fiint_{\mathcal{Q}_{\rho}(z_0)} u_+^{p_2}\,d\xi\,dt\right)^{\frac{1}{p_1}} \\
    &\qquad+ \Big[ \fint_{t_0-\rho^{sp}}^{t_0} \left( \fint_{B_\rho(\xi_0)} u_+^{p-1}(\eta,t) \,d\eta\right)^{\frac{p}{p-1}}\,dt \Big]^\frac{\mathcal{B}'(p-1)}{pp_1} + 1 \Bigg]\\
    &\qquad\qquad+ \mathfrak{C}\rho^{-sp}\int_{t_0-\rho^{sp}}^{t_0}\textup{Tail}^{p-1}(u_+(t);\rho,\xi_0)\,dt.
\end{aligned}
\end{equation}
    for any $\mathcal{Q}_\rho(\xi_0)\Subset \Omega_T$ and $\sigma\in (0,1)$ with constants $c=c(\verb|data|)$, $c_1$ chosen in the proof and $\mathfrak{C}$ chosen in the proof of \cref{L6.1} and
    \begin{align}\label{defp1p2B}
        &p_1=\min \Bigg\{2,\frac{p(Q+2s)-2Q}{sp}\Bigg\},\,p_2=\max\{2,p\}\nonumber\\
        &\mathcal{B}=\frac{p+\frac{2s}{Q}-\frac{p-1}{p}(p-2)}{1+\frac{s}{Q}}
    \end{align}
\end{theorem}

    \subsection{Auxilliary lemmas}

    	We collect here some important estimates which will be used multiple times in the remaining article.

	\begin{lemma}\cite{AdiMali2018}
    \label{lem:annul}
		Let $\gamma>0$ and let $|\cdot|_{\hn}$ be the standard homogeneous norm on $\hn$. Then
		\[\int_{\hn\setminus B(\xi_0,r)}\frac{1}{\dhn{\xi}{\xi_0}^{Q+\gamma}}\,d\xi\leq c(N,\gamma)r^{-\gamma}.\]
	\end{lemma}

   We include below a version of Sobolev-Poincar\'e inequality on the Heisenberg group. We omit the proof since it is entirely similar to that in \cite{Mingione2003}.
    \begin{theorem}
        \label{sobolpoinc}
        Let $B\subset \hn$ be a ball and let $1\leq p<\infty$ and $0<s<1$. Then for all $u\in W^{s,p}(B)$,
        \begin{align*}
            \int_B|u(\xi)-u_B|^p\,d\xi \leq |\textup{diam}(B)|^{sp}\intgaglifrac{B},
        \end{align*} where $f_B:=\fint_B f(\xi)\,d\xi$.
    \end{theorem}
    \begin{theorem}(\cite[Theorem 1.1]{AdiMali2018})\label{sobolheisen}
        Let $s\in (0,1)$ and $sp<Q$. Then there exists a constant $C=C(Q,p,s)$ such that
        \begin{align}
            \|u\|_{L^{p^*}(\hn)}\leq C[u]_{W^{s,p}(\hn)} 
        \end{align} for all $u\in W^{s,p}(\hn)$ and $p^*=\frac{Qp}{Q-sp}$.
    \end{theorem}    

    The following theorem is proved in the appendix as \cref{gaginter}. 
    
    \begin{theorem}
        Let $B\subset\hn$ be a ball of radius $1$. Let $1\leq p\leq\infty$, and $r> q\geq1$ satisfy 
        \begin{align*}
            s-\frac{Q}p > -\frac{Q}{r},
        \end{align*}
        let $\theta\in (0,1)$ be such that
        \begin{align*}
            \theta\left(\frac1p-\frac{s}{Q}\right)+\frac{1-\theta}{q}=\frac1r.
        \end{align*}
        Then there exists a constant $c=c(Q,s,p,q,\theta)>0$ such that 
        \begin{align*}
            \|u\|_{L^r(B)}\leq c\|u\|^{1-\theta}_{L^q(B)}\|u\|^\theta_{W^{s,p}(B)}
        \end{align*} for every $u\in L^q(B)\cap W^{s,p}(B)$.
    \end{theorem}

	\subsubsection{Parabolic Inequalities} In this section, we introduce the parabolic solution space $V_s^{p,q}(B_R\times I):=L^p(I;W^{s,p}(B_R))\cap L^\infty(I;L^q(B_R))$ for $q>1$ and we define 
    \begin{align*}
        \|f\|_{V^{p,q}_s(B_R\times I)}:=\left[\int_I[f(\cdot,t)]^p_{W^{s,p}(B_R)}\,dt+\sup_{t\in I}\|f(\cdot,t)\|_{L^q(B_R)}^q\right]^{\frac1p}
    \end{align*}

  We will prove the following parabolic Sobolev inequality.
    \begin{theorem}\label{T4.10}
        Suppose $f\in V_s^{p,q}(B_R\times I)$ with $1<q\leq \max\{p,2\}$. Suppose $r$ and $m$ are positive numbers satisfying 
        \begin{align}\label{indexsetup2}
            \frac1r + \frac1m\left(\frac{sp}{Q}+\frac{p}{q}-1\right)=\frac1q
        \end{align} with admissible ranges
        \begin{align}
            \begin{cases}
                 r\in \left[q,\frac{Qp}{Q-sp}\right] ,& m\in [p,\infty),\mbox{ if }Q>sp,\\
                 r\in [q,\infty) ,& m\in \left(\frac{qsp}{Q}+p-q,\infty\right) \mbox{ if }Q=sp,\\
                 r\in [q,\infty) ,& m\in \left[ \frac{qsp}{Q}+p-q,\infty \right) \mbox{ if }Q<sp.
            \end{cases} 
        \end{align}Then, there is a constant $c$ depending on the data such that
        \begin{align}\label{parabolsobol}
        	\|f\|^{\frac{Qrm}{spr+Qm}}_{L^{r,m}(B_R\times I)} \leq c\left( \int_I[f(\cdot,t)]^p_{W^{s,p}(B_R)}\,dt + R^{-sp} \|f\|^p_{L^p(B_R)} +\sup_{t\in I} \|u(\cdot,t)\|^q_{L^q(B_R)}\right),
        \end{align} where
        \begin{align*}
        	\|f\|_{L^{r,m}(B_R\times I)}:=\left(\int_I \|f(\cdot,t)\|^m_{L^r(B_R)}\right)^{\frac{1}{m}}.
        \end{align*}
    \end{theorem}

    \begin{proof}
    	If $(r,m)=\left(\frac{Qp}{Q-sp},p\right)$ for $Q>sp$ then the result follows from \cref{sobolheisen2}. We now take $p<m<\infty$. Observe that $\theta=p/m\in (0,1)$ satisfies \cref{indexsetup1} due to \cref{indexsetup2}. By a scaled version of \cref{gaginter} and Young's inequality, we have
    	\begin{align*}
    		&\left(\int_I\|f(\cdot,t)\|_{L^r(B_R)}^m\,dt\right)^{\frac1m}\\
    		 &\qquad \leq c\left(\int_I\{\left([f(\cdot,t)]_{W^{s,p}(B_R)}+R^{-s}\|f(\cdot,t)\|_{L^p(B_R)}\right)^\theta\|f(\cdot,t)\|_{L^q(B_R)}^{1-\theta}\}^{m}\,dt\right)^{\frac1m}\\
    		&\qquad\leq c\left(\int_I[f(\cdot,t)]^p_{W^{s,p}(B_R)}+R^{-sp}\|f(\cdot,t)\|^p_{L^p(B_R)}\,dt\right)^{\frac1m}\left(\sup_{t\in I}\int_{B_R}|f(\xi,t)|^q\,d\xi\right)^{\frac{1-\theta}{q}}\\
    		&\qquad\leq c \left(\int_I[f(\cdot,t)]^p_{W^{s,p}(B_R)}+R^{-sp}\|f(\cdot,t)\|^p_{L^p(B_R)}\,dt\right)^{\frac{\theta}{p}+\frac{1-\theta}{q}} + c \left(\sup_{t\in I}\int_{B_R}|f(\xi,t)|^q\,d\xi\right)^{\frac{\theta}{p}+\frac{1-\theta}{q}}
    	\end{align*}
   	Since $\frac{\theta}{p} + \frac{1-\theta}{q} = \frac1r + \frac{\theta s}{Q} = \frac1r + \frac{sp}{mQ} = \frac{spr+Qm}{Qrm}$, we obtain \cref{parabolsobol} after raising to the power $\frac{Qrm}{spr+Qm}$ in the last display.
    \end{proof}
We are now going to state a lemma regarding the fractional seminorm of truncated functions.
\begin{lemma}
    Let $v\in W^{s,p}(B_R)$. Then for any $0\leq b\leq a$
    \begin{align*}
        [(v-b)_-]^p_{W^{s,p}(B_R)}\leq [(v-a)_-]^p_{W^{s,p}(B_R)}
    \end{align*}
\end{lemma}
\begin{proof}
    For proof, see \cite[Lemma 2.3]{BK24}
\end{proof}
We end this subsection with a technical lemma which will be used in a sequel to deal with iterative inequalities
\begin{lemma} \label{iterlemma}
    \cite[Chapter 1, Lemma 4.2]{DiB12} Let $M,b>1$ and $\kappa,\delta>0$ be given. For each $n\in\mathbb{N}$, we assume that
    \begin{align*}
        Y_{n+1}\leq Mb^n\big(Y_n^{1+\delta}+Z_n^{1+\kappa}Y_h^\delta\big)\,\mbox{ and }\, Z_{n+1}\leq Mb^n\big(Y_n+Z_n^{1+\kappa}\big)
    \end{align*}
    Further assume that $Y_0+Z_0^{1+\kappa}\leq (2M)^{-\frac{1+\kappa}{\zeta}}b^{-\frac{1+\kappa}{\zeta^2}}$ for $\zeta=\min\{\kappa,\delta\}$. Then we have
    \begin{align*}
        \lim_{n\rightarrow\infty}Y_n=\lim_{n\rightarrow\infty}Z_n=0\\
    \end{align*}
\end{lemma}

    \subsection{Plan of the Paper} The plan of the paper is as follows. In \cref{SectionRegularize}, we define a regularization in time and prove convergence results associated to it. \cref{Caccioppolli} contains the result related to the Caccioppolli inequality in \cref{L4.6} of the weak subsolution to the problem (\ref{maineq}) and also we revised this very inequality in \cref{L4.7} ({\em Caccioppolli inequality II}) by splitting the tail term into {\em middle distance term} and {\em faraway term}. In \cref{localbound}, we discuss the local boundedness result and prove our main result -- \cref{mainresult}. In \cref{sec:interineq}, we prove two interpolation inequalities, one interpolating Sobolev and H\"older spaces and another interpolating Sobolev and BMO spaces. In \cref{sec:extsob}, we prove an extension theorem for fractional Sobolev spaces on regular domains in the Heisenberg group. Finally, in \cref{sec:gagint}, we prove a Gagliardo-Nirenberg interpolation inequality suited for regularity theory for fractional quasilinear equations.

\section{Regularization in time}\label{SectionRegularize}
 To prove regularity results for parabolic equations, one needs to regularise in time. We will bring ideas from \cite{brasco2021continuity} to regularise the test functions with respect to time. Let us consider a nonnegative, smooth even function $\gamma:\mathbb{R}\rightarrow \mathbb{R}$ with compact support in $(-1/2,1/2)$ and $\int_\mathbb{R}\gamma=1$. Then we construct the following convolution 
\begin{align}\label{eq4.01}
    \phi^\varepsilon(t):=\frac{1}{\varepsilon}\int_{t-\varepsilon/2}^{t+\varepsilon/2}\gamma\Big(\frac{t-\tau}{\varepsilon}\Big)\phi(\tau)d\tau=\frac{1}{\varepsilon}\int_{-\varepsilon/2}^{\varepsilon/2}\gamma\Big(\frac{\alpha}{\varepsilon}\Big)\phi(t-\alpha)d\alpha,\;\;\;\;\;t\in(a,b)
\end{align}
for any $\phi\in L^1((a,b))$ and $0<\varepsilon<\min\{b-t,t-a\}$.\\
Before establishing the estimates, we want to focus on some technical lemmas that will be useful in the sequel to prove Caccioppoli inequality. We begin with a lemma whose proof can be found in \cite[Lemma 3.1]{DCKP16}
\begin{lemma}\label{L4.1}
    For $x,y\in \mathbb{R}$, $\varepsilon>0$ and $p\geq 1$, we find the following
    \begin{align*}
        |x|^p-|y|^p\leq c_P\varepsilon|y|^p+(1+C_p\varepsilon)\varepsilon^{1-p}|x-y|^p
    \end{align*}
    with $C_p:= (p-1)\Gamma(\max\{1,p-2\})$. where $\Gamma$ is the standard Gamma function. 
\end{lemma}
Now, we move to several essential lemmas that we will need during the proof of the energy estimates. We have used some fundamental inequalities, properties of convolutions and $\gamma$. We mainly try to prove these using the ideas related to the Steklov averages in \cite{CDG17}. In the course of the proofs of the following lemmas, we always assume $s>0$ and $p,q\geq 1$ and also assume $0<T_1<T_2$. Here are the lemmas and their proofs

\begin{lemma}\label{L4.2}
    Let $\Omega\subseteq\hn$ denote an open set. Let $q\in [1,\infty)$. If $\phi\in C([T_1,T_2];L^q(\Omega))$, then there is a positive number $\varepsilon_0$ such that for all  $0<\varepsilon<\varepsilon_0$ the sequence $\phi^\varepsilon(\cdot,t)$ converges to $\phi(\cdot,t)$ in $L^q(\Omega)$ for every $t\in (T_1+\varepsilon_0/2,T_2-\varepsilon_0/2)$ as $\varepsilon\rightarrow 0$.
\end{lemma}
\begin{proof}
   We want to prove this in two steps. We will show first that
\begin{equation}\label{eq4.1}
	\phi^\varepsilon(t)\in L^q(\Omega)\mbox{ for all } t\in[T_1,T_2]
\end{equation}
and second that
\begin{equation}\label{eq4.2}
	\phi^\varepsilon(\cdot,t)\rightarrow\phi(\cdot,t)\mbox{ in }L^q(\Omega)
\end{equation}
Now, for $t\in [T_1,T_2]$
\begin{align*}
	\|\phi^\varepsilon(\cdot,t)\|_{L^q(\Omega)}&=\frac{1}{\varepsilon}\left\|\int_{t-\varepsilon/2}^{t+\varepsilon/2}\gamma\Big(\frac{t-\tau}{\varepsilon}\Big)\phi(\tau)d\tau\right\|_{L^q(\Omega)}\leq C(\varepsilon)\int_{T_1}^{T_2}\|\phi(\cdot,\tau)\|_{L^q(\Omega)}\;d\tau<+\infty
\end{align*}
so, (\ref{eq4.1}) is proved. Now, Let us take  $T\in (T_1,T_2)$ and define
\begin{equation*}
	\frac{\varepsilon_0^T}{2}:=\min\{T-T_1,T_2-T\}
\end{equation*}
Moreover, let us take $0<\varepsilon<\frac{\varepsilon_0^T}{2}$. Then, by definition of continuity, for $\varepsilon'>0$, there exists $\delta>0$ such that
\begin{equation*}
	\|\phi(\cdot,t_1)-\phi(\cdot,t_2)\|_{L^q(\Omega)}<\varepsilon'\;\;\;\forall\;t_1,t_2\in [T_1,T_2]\;\;\;\mbox{with}\;\;|t_1-t_2|<\delta<\frac{\varepsilon_0^T}{2}
\end{equation*}
then for any $t\in [T_1+\varepsilon_0^T/2,T_2-\varepsilon_0^T/2]$ and $0<\varepsilon<\delta$, we obtain
\begin{align*}
	\|\phi^\varepsilon(\cdot,t)-\phi(\cdot,t)\|_{L^q(\Omega)}&\leq \left\|\frac{1}{\varepsilon}\int_{t-\varepsilon/2}^{t+\varepsilon/2}\gamma\Big(\frac{t-\tau}{\varepsilon}\Big)\Big(\phi(\tau)-\phi(t)\Big)d\tau\right \|_{L^q(\Omega)}\\
	&\leq \frac{1}{\varepsilon}\int_{t-\varepsilon/2}^{t+\varepsilon/2}\|\phi(\tau)-\phi(t)\|_{L^q(\Omega)}d\tau	<\varepsilon'
\end{align*}
This holds for any $T\in (T_1,T_2)$, hence the result.
\end{proof}
\begin{lemma}\label{L4.3}
    Let $p, q\in [1,\infty]$. Assume that $\phi\in L^q(T_1,T_2;L^p(\Omega))$. Then  there exists  $C=C(p,q)>0$ such that
	\begin{equation*}\label{L3.2}
		\|\phi^\varepsilon\|_{L^q(T_1+\varepsilon_0/2,\;T_2-\varepsilon_0/2;\;L^p(\Omega))}\leq C\;\;\;\;\mbox{for any }\varepsilon\leq\varepsilon_0
	\end{equation*}
\end{lemma}
\begin{proof}
    This will be proved in two steps.
    \item 
   \textbf{Step 1: }At first we want to show that, for any $t\in [T_1,T_2]$ and $1\leq p\leq +\infty$
\begin{equation*}
	\|\phi^\varepsilon\|_{L^p(\Omega)}< +\infty
\end{equation*}
If $1\leq q<+\infty$, then for any $t\in [T_1,T_2]$
\begin{equation*}
	\|\phi^\varepsilon(\cdot,t)\|_{L^p(\Omega)}=\|\frac{1}{\varepsilon}\int_{t-\varepsilon/2}^{t+\varepsilon/2}\gamma\left(\frac{t-\tau}{\varepsilon}\right)\phi(\cdot,\tau)d\tau\|_{L^p(\Omega)}
\end{equation*}
\begin{equation*}
	\leq \frac{1}{\varepsilon}\int_{t-\varepsilon/2}^{t+\varepsilon/2}\|\phi(\cdot,\tau)\|_{L^p(\Omega)}d\tau
\end{equation*}
\begin{equation*}
	\leq \frac{1}{\varepsilon^{1-\frac{1}{q'}}}\|\phi\|_{L^q(T_1,T_2;L^p(\Omega))}<\infty
\end{equation*}
If $q=+\infty$, then for any $t\in[T_1,T_2]$
\begin{equation*}
	\|\phi^\varepsilon(\cdot,t)\|_{L^p(\Omega)}=\frac{1}{\varepsilon}\int_{t-\varepsilon/2}^{t+\varepsilon/2}\|\phi(\cdot,\tau)\|_{L^p(\Omega)}d\tau
\end{equation*}
\begin{equation*}
	\leq \|\phi\|_{L^\infty(T_1,T_2;L^p(\Omega))}<\infty
\end{equation*}
\item 
\textbf{Step 2: } For $q=+\infty$, the result is obvious by step (I). Now, for $1\leq q<+\infty$
\begin{equation*}
	\|\phi^\varepsilon(\cdot,t)\|_{L^p(\Omega)}^q\leq \frac{1}{\varepsilon}\int_{t-\varepsilon/2}^{t+\varepsilon/2}\|\phi(\cdot,\tau)\|_{L^p(\Omega)}^qd\tau
\end{equation*}
Now consider $T_2<+\infty$. Now choose $T$ as was done in Lemma \ref{L4.2}. Then for every $0<\varepsilon\leq \varepsilon_0^T/2$, let us denote $I_T:=(T_1+\varepsilon_0^T/2,T_2-\varepsilon_0^T/2)$. Now,
\begin{align*}
	\int_{I_T}\|\phi^\varepsilon(\cdot,t)\|^q_{L^p(\Omega)}dt&\leq \frac{1}{\varepsilon}\int_{I_T}\Bigg(\int_{t-\varepsilon/2}^{t+\varepsilon/2}\|\phi(\cdot,\tau)\|_{L^p(\Omega)}^qd\tau\Bigg)dt\\
	&=\frac{1}{\varepsilon}\int_{T_1+\varepsilon_0^T/2}^{T_2-\varepsilon_0^T/2}\Big(A(t+\varepsilon/2)+A(t-\varepsilon/2)\Big)dt\\
	&=\frac{1}{\varepsilon}\Bigg[\int_{T_1+\varepsilon_0^T/2+\varepsilon/2}^{T_2-\varepsilon_0^T/2+\varepsilon/2}A(t)dt-\int_{T_1+\varepsilon_0^T/2-\varepsilon/2}^{T_2-\varepsilon_0^T/2-\varepsilon/2}A(t)dt\Bigg]\\
	&\leq \frac{1}{\varepsilon}\int_{T_2-\varepsilon_0^T/2-\varepsilon/2}^{T_2-\varepsilon_0^T/2+\varepsilon/2}A(t)dt\\
    &\leq \frac{1}{\varepsilon}\int_{T_2-\varepsilon_0^T/2-\varepsilon/2}^{T_2-\varepsilon_0^T/2+\varepsilon/2}A(T_2-\varepsilon_0^T/2)dt\\
	&=A(T_2-\varepsilon_0^T/2)\\
    &\leq \int_{I_T}\|\phi(\cdot,\tau)\|_{L^p(\Omega)}^qd\tau\\
	&\leq \int_{T_1}^{T_2}\|\phi(\cdot,\tau)\|_{L^p(\Omega)}^qd\tau<\infty
\end{align*} where 
\begin{align*}
A(s):=\int_{T_1+\varepsilon_0^T/2}^s\|\phi(\cdot,\tau)\|_{L^p(\Omega)}^qd\tau,\; s\in I_T
\end{align*}
If $T_2=+\infty$, then,
\begin{align*}
	\int_{T_1+\varepsilon_0/2}^\infty\|\phi^\varepsilon(\cdot,t)\|_{L^p(\Omega)}^qdt&=\lim\limits_{T_2\rightarrow \infty}\int_{T_1+\varepsilon_0/2}^{T_2+\varepsilon_0/2}\|\phi^\varepsilon(\cdot,t)\|_{L^p(\Omega)}^qdt\\
	&\leq \lim\limits_{T_2\rightarrow \infty}\int_{T_1+\varepsilon_0/2}^{T_2+\varepsilon_0/2}\|\phi(\cdot,t)\|_{L^p(\Omega)}^qdt\\
	&\leq \int_{T_1+\varepsilon_0/2}^\infty\|\phi(\cdot,t)\|_{L^p(\Omega)}^qdt 
\end{align*}
\end{proof}
\begin{lemma}\label{L4.4}
    	Let $p\in (1,\infty)$. If $\phi\in L^p(T_1,T_2;W^{s,p}(\Omega))$. Then, for some constant $C>0$ depending on $Q,p$, we have
	\begin{equation}\label{eq4.4}
		\|\phi^\varepsilon\|_{L^p(T_1+\varepsilon_0/2\;,\;T_2-\varepsilon_0/2\;;\;W^{s,p}(\Omega))}\leq C
	\end{equation}
for all $\varepsilon\leq \varepsilon_0$. 
\end{lemma}
\begin{proof}
    To show (\ref{eq4.4}), it is enough to prove that for any $\varepsilon\leq\varepsilon_0$,
\begin{equation}\label{eq4.5}
	\|\phi^\varepsilon\|_{L^p(T_1+\varepsilon_0/2,\;T_2-\varepsilon_0/2\;;\;W^{s,p}(\Omega))}\leq \|\phi\|_{L^p(T_1\;,\;T_2\;;\;L^p(\Omega))}+[\phi]_{L^p(T_1\;,\;T_2\;;\;W^{s,p}(\Omega))}.
\end{equation}
We know
\begin{equation*}
	\|\phi^\varepsilon\\|_{L^p(T_1+\varepsilon_0/2\;,\;T_2-\varepsilon_0/2\;;\;W^{s,p}(\Omega))}\leq \Bigg(\|\phi^\varepsilon\|^p_{L^p(T_1+\varepsilon/2\;,\;T_2-\varepsilon_0/2\;;\;L^p(\Omega))}+[\phi^\varepsilon]^p_{L^p(T_1+\varepsilon/2\;,\;T_2-\varepsilon_0/2\;;\;W^{s,p}(\Omega))}\Bigg)^{1/p}
\end{equation*}
\begin{equation}\label{eq4.6}
	\hspace{56mm}\leq C\Bigg(\underbrace{\|\phi^\varepsilon\|_{L^p(T_1+\varepsilon/2,\;T_2-\varepsilon_0/2;\;L^p(\Omega))}}_{A}+\underbrace{[\phi^\varepsilon]_{L^p(T_1+\varepsilon/2\;,\;T_2-\varepsilon_0/2\;;\;W^{s,p}(\Omega))}}_{B}\Bigg).
\end{equation}
Since $\phi\in L^p(T_1,T_2;W^{s,p}(\Omega))$ therefore $\phi\in L^p(T_1,T_2;L^p(\Omega))$. It follows from \cref{L3.2} that
\begin{equation*}
	A\leq C \|\phi^\varepsilon\|_{L^p(T_1,T_2; L^p(\Omega))}
\end{equation*}
for some constant $C=C(Q,p)$. It remains to show that $B$ is finite. To achieve this, let us calculate the following first
\begin{align}\label{eq4.7}
	[\phi^\varepsilon(\cdot,t)]^p_{W^{s,p}(\Omega)}&=\intgaglifrac[\phi^\varepsilon]{\Omega}\nonumber\\
	&=\int_\Omega\int_\Omega\frac{\big|\frac{1}{\varepsilon}\int_{-\varepsilon/2}^{\varepsilon/2}\gamma(\frac{\sigma}{\varepsilon})\big(\phi(\xi,t-\sigma)-\phi(\eta,t-\sigma)\big)d\sigma\big|^p}{|\eta^{-1}\circ \xi|^{Q+sp}}d\xi\,d\eta
\end{align}
Since $\zeta$ is measurable, for every measurable set $I\subset \mathbb{R}$, we define
\begin{equation}
	m(I):=\int_I\frac{\gamma(\sigma/\varepsilon)}{\varepsilon}d\sigma.
\end{equation}
Observe that $m$ is another measure on $\mathbb{R}$ such that
\begin{equation*}
	\frac{dm}{d\sigma}=\frac{\gamma(\sigma/\varepsilon)}{\varepsilon}
\end{equation*}
Now, by construction of $\zeta$, we have 
\begin{equation*}
	m(I)\in [0,1]\mbox{ for any measurable }I\subseteq \mathbb{R}.
\end{equation*}
Hence $m$ is a probability measure. Then we apply Jensen's Inequality to (\ref{eq4.7}) to obtain
\begin{align}\label{eq4.9}
	[\phi^\varepsilon(\cdot,t)]^p_{W^{s,p}(\Omega)} &\leq \int_\Omega\int_\Omega\frac{\int_{-\varepsilon/2}^{\varepsilon/2}|\phi(\xi,t-\sigma)-\phi(\eta,t-\sigma)|^p\;\frac{\gamma(\sigma/\varepsilon)}{\varepsilon}d\sigma}{|\eta^{-1}\circ \xi|^{Q+sp}}d\xi\,d\eta\nonumber\\
	&=\int_{-\varepsilon/2}^{\varepsilon/2}\frac{\gamma(\sigma/\varepsilon)}{\varepsilon}\int_\Omega\int_\Omega\frac{|\phi(\xi,t-\sigma)-\phi(\eta,t-\sigma)|^p}{|\eta^{-1}\circ \xi|^{Q+sp}}(d\xi\,d\eta)d\sigma\;\;\;\;\;\mbox{(By Fubini's Theorem)}\nonumber\\
	&=\frac{1}{\varepsilon}\int_{t-\varepsilon/2}^{t+\varepsilon/2}\gamma\Bigg(\frac{t-\tau}{\varepsilon}\Bigg)\Bigg[\int_\Omega\int_\Omega\frac{|\phi(\xi,t)-\phi(\eta,t)|^p}{|\eta^{-1}\circ \xi|^{Q+sp}}d\xi\,d\eta\Bigg]d\tau\nonumber\\
	&=\frac{1}{\varepsilon}\int_{t-\varepsilon/2}^{t+\varepsilon/2}\gamma\Bigg(\frac{t-\tau}{\varepsilon}\Bigg)[\phi(\cdot,\tau)]^p_{W^{s,p}(\Omega)}d\tau
\end{align}
Let us consider some $\varepsilon_0(>0)$. With the help of the fact that the compact support of $\gamma$ is $(-1/2,1/2)$ and $\int_\mathbb{R}\gamma=1$, we obtain by integrating $[\phi^\varepsilon(\cdot,t)]^p_{W^{s,p}(\Omega)}$ from $T_1+\varepsilon_0/2$ to $T_2-\varepsilon_0/2$, for every $\varepsilon\leq \varepsilon_0$.
\begin{align*}
	\int_{T_1+\varepsilon_0/2}^{T_2-\varepsilon_0/2}[\phi(\cdot,t)]^p_{W^{s,p}(\Omega)}dt&\leq\int_{T_1+\varepsilon/2}^{T_2-\varepsilon/2}\frac{1}{\varepsilon}\int_{t-\varepsilon/2}^{t+\varepsilon/2}\gamma\Big(\frac{t-\tau}{\varepsilon}\Big)[\phi(\cdot,\tau)]^p_{W^{s,p}(\Omega)}d\tau dt\\
	&\leq \int_{T_1}^{T_2-\varepsilon}\frac{1}{\varepsilon}\int_{T_1+\varepsilon/2}^{\tau+\varepsilon/2}\gamma\Big(\frac{t-\tau}{\varepsilon}\Big)[\phi(\cdot,\tau)]^p_{W^{s,p}(\Omega)}dt d\tau\\
    &\quad+\int_{T_1+\varepsilon}^{T_2}\frac{1}{\varepsilon}\int_{\tau-\varepsilon/2}^{T_2-\varepsilon/2}\gamma\Big(\frac{t-\tau}{\varepsilon}\Big)[\phi(\cdot,\tau)]^p_{W^{s,p}(\Omega)}dt d\tau\\
	&\qquad+\int_{T_1+\varepsilon}^{T_2-\varepsilon}\frac{1}{\varepsilon}\int_{T_1+\varepsilon/2}^{T_2-\varepsilon/2}\gamma\Big(\frac{t-\tau}{\varepsilon}\Big)[\phi(\cdot,\tau)]^p_{W^{s,p}(\Omega)}dt d\tau\\
	&\leq C[\phi]^p_{L^p(T_1,T_2;W^{s,p}(\Omega))}.
\end{align*}
This then implies (\ref{eq4.5}).
\end{proof}
\begin{lemma}\label{L4.5}
    If $\phi\in C([T_1,T_2];L^q(\Omega))$, then for some positive numbers $\varepsilon,\varepsilon_0$ with $0<\varepsilon<\varepsilon_0$ such that $\phi^\varepsilon(\cdot,t+\varepsilon/2)$ converges to $\phi(\cdot,t)$ in $L^q(\Omega)$ for every $t\in (T_1,T_2-\varepsilon_0/2)$ as $\varepsilon\rightarrow 0$.
\end{lemma}
\begin{proof}
   We omit this proof as it is similar to that of Lemma \ref{L4.2}
\end{proof}

\section{Caccioppolli inequality}\label{Caccioppolli}
In this section, we are going to prove the Caccioppolli estimates for the weak subsolution $u$ of the given problem (\ref{maineq}). This particular Caccioppoli inequality which uses a test function suggested in \cite{Kassweid2024} which allows for the Tail term to be subsumed in the level.
\begin{lemma}\label{L4.6}
    \textbf{(Caccioppolli Inequality)} Let $(\xi_0,t_0)\in \Omega_T$. Let $l,r,R\in \mathbb{R}_+$ satisfy $0<l<r<R$. Let $u$ be a local sub-solution to the problem (\ref{maineq}).  Consider two non-negative functions $\psi\in C^\infty_0(B_{(r+l)/2}(\xi_0))$ with $0\leq \psi\leq 1$ in $\hn$ and $\nu\in C^\infty(\mathbb{R})$ such that $\nu(t)=0$ iff $t\leq t_0-\theta_2$ and $\nu(t)=1$ if $t\geq t_0-\theta_1$ where ball $B_r\equiv B_r(\xi_0)$ such that $\bar{B_r}\subset\Omega$ and $0\leq\theta_1<\theta_2$ satisfying $[t_0-\theta_2,t_0]\subset (0,T)$. Let $w:= u-g-k$ with a level $k\in \mathbb{R}$ and $g(t)=\mathfrak{C}R^{-sp}\int_{t_0-\theta_2}^t\textup{Tail}^{p-1}(u_+(s);R,\xi_0)ds$. Then there exists a constant $c>0$ depending on $\mathtt{data}$ such that
	\begin{align*}
		\esssup_{t_0-\theta_2<t<t_0}&\int_{B_r}(w_+^2\psi^p)(\xi,t)d\xi +\int_{t_0-\theta_2}^{t_0}\int_{B_r}(w_+\psi^p\nu^2)(\xi,t)\Bigg\{\int_{B_r}\frac{w^{p-1}_-(\eta,t)}{|\eta^{-1}\circ \xi|^{Q+sp}_{\hn}}d\eta\Bigg\}d\xi\,dt\\
        &\quad+\int_{t_0-\theta_2}^{t_0}\iint_{B_r\times B_r}\frac{|(w_+\psi\nu^{\frac2p})(\xi,t)-(w_+\psi\nu^{\frac2p})(\eta,t)|^p}{|\eta^{-1}\circ \xi|^{Q+sp}_{\hn}}\; d\xi\,d\eta\,dt\\
    &\qquad	\leq c\Bigg(\frac{r^{1-s}}{r-l}\Bigg)^p\|w_+\|^p_{L^p(t_0-\theta_2,t_0;L^p(B_r))}+	\int_{t_0-\theta_2}^{t_0}\int_{B_r}(w_+^2\psi^p\nu)(\xi,t)\partial_t\nu(t)\;d\xi\,dt\\
    &\quad\quad\quad + c\int_{t_0-\theta_2}^{t_0}\intbr[r](w_+\psi^p\nu^2)(\xi,t)\Bigg\{\int_{\hn\setminus B_r}\frac{w_+^{p-1}(\eta,t)}{|\eta^{-1}\circ \xi|_{\hn}^{Q+sp}}d\eta\Bigg\}d\xi\,dt \\
    &\qquad\qquad - \mathfrak{C} R^{-sp} \int_{t_0-\theta_2}^{t_0}\int_{B_r} \psi^p(\xi)\nu^2(t) w_+(\xi,t) \textup{Tail}\,^{p-1}(u_+(t),R,\xi_0)\,d\xi\,dt
\end{align*}
\end{lemma}
\begin{proof}
    Since parabolic problems do not possess sufficient time regularity, we must regularise the test function. Let us choose $t_0\in(0,T)$ and $\theta_2(>0)$ in such a way that $[t_0-\theta_2,t_0]\in (0,T)$. Let us choose $\varepsilon_0$ such that
\begin{align}\label{eq4.10}
    	0<\varepsilon<\frac{\varepsilon_0}{2}=\frac{1}{4}\min\{t_0-\theta_2,T-t_0,\theta_2\}
\end{align}
Let $T_1=t_0-\theta_2-\varepsilon_0, T_2=t_0+\varepsilon_0$ and also set $t_1=t_0-\theta_2$. The number $t_2\in (t_0-\theta_2,t_0]$ is to be determined later. Choose $\phi=\phi(\xi,t)$ whose spatial support is contained compactly in $B_{(r+l)/2}$ for every $t\in (0,T)$. Now, using $\phi^\varepsilon(\xi,t)$ as a test function in the following subsolution formulation
\begin{align}\label{eq4.11}
    \int_{B_r}&u(\xi,t_2)\phi^\varepsilon(\xi,t_2)d\xi-\int_{B_r}u(\xi,t_1)\phi^\varepsilon(\xi,t_1)d\xi\underbrace{-\int_{t_1}^{t_2}\int_{B_r}u(\xi,t)\phi_t^\varepsilon(\xi,t)d\xi\,dt}_{D}\nonumber\\
    &+	\int_{t_1}^{t_2}\underbrace{\iint_{\hn\times\hn}\frac{1}{2}|u(\xi)-u(\eta)|^{p-2}(u(\xi)-u(\eta))(\phi^\varepsilon(\xi)-\phi^\varepsilon(\eta))K(\xi,\eta,t)\;d\eta\,d\xi}_{:=\mathcal{E}(u,\phi,t)}dt\leq 0
\end{align}
Let us estimate $D$ as follows. The calculation is similar to \cite[p. 38]{Ding2021}
\begin{align}\label{eq4.12}
    -\int_{t_1}^{t_2}\int_{B_r}u(\xi,t)\phi_t^\varepsilon(\xi,t)d\xi\,dt=E(\varepsilon)-\underbrace{\int_{B_r}\int_{t_1+\varepsilon/2}^{t_2-\varepsilon/2}u^\varepsilon(\xi,t)\phi_t(\xi,\tau)d\tau d\xi}_{F}
\end{align}
where 
\begin{align}\label{eq4.13}
    E(\varepsilon)&=-\int_{B_r}\int_{t_1-\varepsilon/2}^{t_1+\varepsilon/2}\Bigg(\frac{1}{\varepsilon}\int_{t_1}^{\tau+\varepsilon/2}u(\xi,t)\gamma\Big(\frac{\tau-t}{\varepsilon}\Big)dt\Bigg)\phi_\tau(\xi,\tau)d\tau\, d\xi\nonumber\\
    &-\int_{B_r}\int_{t_2-\varepsilon/2}^{t_2+\varepsilon/2}\Bigg(\frac{1}{\varepsilon}\int_{\tau-\varepsilon/2}^{t_2}u(\xi,t)\gamma\Big(\frac{\tau-t}{\varepsilon}\Big)dt\Bigg)\phi_\tau(\xi,\tau)d\tau \,d\xi
\end{align}
Now from (\ref{eq4.12}) and (\ref{eq4.13}), we have in (\ref{eq4.11})
\begin{align}\label{eq4.14}
    \int_{t_1}^{t_2}\mathcal{E}(u,\phi,t)dt&+E(\varepsilon)+\int_{t_1+\varepsilon/2}^{t_2-\varepsilon/2}\int_{B_r}u_t^\varepsilon(\xi,t)\phi(\xi,t)d\xi\,dt\leq -\int_{B_r}u(\xi,t_2)\phi(\xi,t_2)d\xi+\int_{B_r}u(\xi,t_1)\phi(\xi,t_1)d\xi\nonumber\\
    &\qquad+\int_{B_r}u^\varepsilon(\xi,t_2-\varepsilon/2)\phi(\xi,t_2-\varepsilon/2)d\xi-\int_{B_r}u^\varepsilon(\xi,t_1+\varepsilon/2)\phi(\xi,t_1+\varepsilon/2)d\xi.
\end{align}
Let us consider $0<l<r\leq R$ and consider the cut-off function $\psi\in C^\infty_0(B_r)$ such that $0\leq \psi\leq 1$ in $\hn$, supp($\psi$)=$B_{\frac{r+l}2}$, $\psi\equiv 1$ in $B_l$ and $|\nabla_{\hn}\psi|\leq \frac{C}{r-l}$. Now, we take test function $\phi(\xi,t)$ as follows
\begin{align}\label{eq4.15}
\phi(\xi,t)=w_+^\varepsilon(\xi,t)\psi^p(\xi)\nu^2(t)\;\;\;\mbox{further denoted as } (w^\varepsilon\psi^p\nu^2)(\xi,t)
\end{align}
where $w^\varepsilon_+(\xi,t):=((u^\varepsilon(\xi,t)-g(t))-k)_+$ and $w_+(\xi,t):=((u(\xi,t)-g(t))-k)_+$, where $g(t)=CR^{-sp}\int_{t_0-r^{sp}}^t\textup{Tail}^{p-1}(u_+(s);R,\xi_0)ds$. Also, we denote 
\begin{align*}
    \psi^i(\xi)\eta^j(t) \mbox{  as }(\psi^i\eta^j)(\xi,t)\mbox{ for some non-negative real numbers } i,j
\end{align*}
Then (\ref{eq4.14}) becomes
\begin{align}\label{eq4.16}
    \frac{1}{2}&\underbrace{\int_{t_1}^{t_2}\iint_{B_r\times B_r}|u(\xi,t)-u(\eta,t)|^{p-2}(u(\xi,t)-u(\eta,t))\Big((w_+^\varepsilon\psi^p\nu^2)^\varepsilon(\xi,t)-(w_+^\varepsilon\psi^p\nu^2)^\varepsilon(\eta,t)\Big)K(\xi,\eta,t)d\xi\,d\eta\,dt}_{D_1^\varepsilon}\nonumber\\
    &\quad+\underbrace{\int_{t_1}^{t_2}\int_{\hn\setminus B_r}\int_{B_r}|u(\xi,t)-u(\eta,t)|^{p-2}(u(\xi,t)-u(\eta,t))(w_+^\varepsilon\psi^p\nu^2)^\varepsilon(\xi,t)K(\xi,\eta,t)d\xi\,d\eta\,dt}_{D_2^\varepsilon}\nonumber\\
    &\quad\quad+\underbrace{\int_{t_1+\varepsilon/2}^{t_2+\varepsilon/2}\int_{B_r}u^\varepsilon_t(\xi,t)(w_+^\varepsilon\psi^p\nu^2)(\xi,t)d\xi\,dt}_{D_3^\varepsilon}+\overline{E}(\varepsilon)\nonumber\\
    &\quad\quad\quad\leq \underbrace{\int_{B_r}\Big\{u^\varepsilon(\xi,t_2-\varepsilon/2)(w_+^\varepsilon\psi^p\nu^2)(\xi,t_2-\varepsilon/2)-u(\xi,t_2)(w_+^\varepsilon\psi^p\nu^2)^\varepsilon(\xi,t_2)\Big\}d\xi}_{D_4^\varepsilon}\nonumber\\
    &\quad\quad\quad\quad+\underbrace{\int_{B_r}\Big\{u(\xi,t_1)(w_+^\varepsilon\psi^p\nu^2)^\varepsilon(\xi,t_1)-u^\varepsilon(\xi,t_1+\varepsilon/2)(w_+^\varepsilon\psi^p\nu^2)(\xi,t_1+\varepsilon/2)\Big\}d\xi}_{D_5^\varepsilon}
\end{align}
where $E(\varepsilon)$ is now denoted as $\bar{E}(\varepsilon)$ with $\phi$ replaced with the expression in \eqref{eq4.15}. 

\paragraph{\textbf{Passage to limit as $\varepsilon\to 0$:}}

Now, at first, we need to estimate $D_i^\varepsilon$ separately for $i=1,2,3,4,5$. We begin with $D_3^\varepsilon$. Here, we have via integration by parts
\begin{align}\label{eq4.17}
    \int^{t_2-\varepsilon/2}_{t_1+\varepsilon/2}\int_{B_r}u_t^\varepsilon(\xi,t)(w_+^\varepsilon\psi^p\nu^2)(\xi,t)d\xi\,dt &= \int^{t_2-\varepsilon/2}_{t_1+\varepsilon/2}\int_{B_r}\partial_tw^\varepsilon_+w^\varepsilon_+\psi^p(\xi)\nu^2(t)d\xi\,dt\nonumber\\
    &\quad+\mathfrak{C}{\int_{t_1+\epsilon/2}^{t_2-\epsilon/2}\int_{B_r}R^{-sp}\;\textup{Tail}^{p-1}(u_+(t);R,x_0)(w^\epsilon_+\psi^p\nu^2)(\xi,t)d\xi\,dt}\nonumber\\
    &=\frac{1}{2}\int^{t_2-\varepsilon/2}_{t_1+\varepsilon/2}\int_{B_r}\partial_t[w^\varepsilon(\xi,t)^2_+]\psi^p(\xi)\nu^2(t)d\xi\,dt\nonumber\\
    &\quad+\mathfrak{C}{\int_{t_1+\epsilon/2}^{t_2-\epsilon/2}\int_{B_r}R^{-sp}\;\textup{Tail}^{p-1}(u_+(t);R,x_0)(w^\epsilon_+\psi^p\nu^2)(\xi,t)d\xi\,dt}\nonumber\\
    &= \frac{1}{2}\int_{B_r}w^\varepsilon_+(\xi,t_2-\varepsilon/2)^2\psi^p(\xi)\nu^2(t_2-\varepsilon/2)d\xi\nonumber\\
    &\quad-\frac{1}{2}\int_{B_r}w^\varepsilon_+(\xi,t_1+\varepsilon/2)^2\psi^p(\xi)\nu^2(t_1+\varepsilon/2)d\xi\nonumber\\
    &\quad- \int^{t_2-\varepsilon/2}_{t_1+\varepsilon/2}\int_{B_r} w^\varepsilon_+(\xi,t)^2\psi^p(\xi)\nu(t)d_t\nu(t)d\xi\,dt\nonumber\\
    &\quad+\mathfrak{C}\underbrace{\int_{t_1+\epsilon/2}^{t_2-\epsilon/2}\int_{B_r}R^{-sp}\;\textup{Tail}^{p-1}(u_+(t);R,x_0)(w^\epsilon_+\psi^p\nu^2)(\xi,t)d\xi\,dt}_{M}.
\end{align}
Let us estimate $M$ as follows
\begin{align*}
    M=& R^{-sp}\int_{t_1+\epsilon/2}^{t_2-\epsilon/2}{\textup{Tail}^{p-1}(u_+(t);R,\xi_0)}{\int_{B_r}(w^\epsilon_+\psi^p\nu^2)(\xi,t)d\xi}\;dt\\
    =&  R^{-sp}\int_{t_1}^{t_2}\underbrace{\mathbbm{1}_{[t_1+\epsilon/2,t_2-\epsilon/2]}{\textup{Tail}^{p-1}(u_+(t);R,\xi_0)}{\int_{B_r}(w^\epsilon_+\psi^p\nu^2)(\xi,t)d\xi}}_{G(t)}\;dt\\
    \end{align*}
    
Since $u\in L^\infty(T_1,T_2;L^{2}(B_r))$,  by \cref{L4.3}, $u^\varepsilon\in L^\infty(T_1,T_2;L^{2}(B_r))$ with $$\|u^\varepsilon\|_{L^\infty(T_1,T_2;L^{2}(B_r))}\leq \|u\|_{L^\infty(T_1,T_2;L^{2}(B_r))}.$$ 

Since $t\mapsto g(t)$ is a continuous function, $w^\varepsilon_+ \in L^\infty(T_1,T_2;L^{2}(B_r))$ and 
$$\|w_+^\varepsilon\|_{L^\infty(T_1,T_2;L^{2}(B_r))}\leq \|w_+\|_{L^\infty(T_1,T_2;L^{2}(B_r))}.$$ 

Consequently $$\|w_+^\varepsilon\|_{L^\infty(T_1,T_2;L^{1}(B_r))}\leq C \|w_+^\varepsilon\|_{L^\infty(T_1,T_2;L^{2}(B_r))}\leq C\|w_+\|_{L^\infty(T_1,T_2;L^{2}(B_r))}.$$ 

As a result, the integrand $G(t)$ is bounded a.e. in $t$.

\begin{align*}
|G(t)|&=\left|\mathbbm{1}_{[t_1+\epsilon/2,t_2-\epsilon/2]}{\textup{Tail}^{p-1}(u_+(t);R,\xi_0)}{\int_{B_r}(w^\epsilon_+\psi^p\nu^2)(\xi,t)d\xi}\right|\\
&\lesssim |\textup{Tail}^{p-1}(u_+(t);R,\xi_0)|\|w_+\|_{L^\infty(T_1,T_2;L^{2}(B_r))}.
\end{align*}

Moreover, we have by an argument similar to the proof of Lemma \ref{L4.2},   as $\varepsilon\rightarrow 0$ for every $t$, $\{w^\epsilon_+\psi^p\nu^2\}$ converges to  $w_+\psi^p\nu^2$ in $L^1(B_r)$. Therefore, by an application of dominated convergence theorem, we have
 \begin{align*}
     M\rightarrow \int_{t_1}^{t_2}\int_{B_r}R^{-sp}\;\textup{Tail}^{p-1}(u_+(t);R,\xi_0)(w_+\psi^p\nu^2)(\xi,t)d\xi\,dt\;\;\;\mbox{as}\;\varepsilon\rightarrow 0
 \end{align*}
 The arguments for convergences of the remaining terms in $D_3^\varepsilon$ are made similarly. Therefore, finally we have
\begin{align}\label{eq4.18}
    D_3^\varepsilon \rightarrow D_3 & :=  \frac{1}{2}\int_{B_r}w_+^2(\xi,t_2)\psi^p(\xi)\nu^2(t_2)d\xi-\frac{1}{2}\int_{B_r}w_+^2(\xi,t_1)\psi^p(\xi)\nu^2(t_1)d\xi\nonumber\\
    &\quad -\int^{t_2}_{t_1}\int_{B_r} w_+^2(\xi,t)\psi^p(\xi)\nu(t)d_t\nu(t)d\xi\,dt\nonumber\\
    &\qquad+\mathfrak{C}\int_{t_1}^{t_2}\int_{B_r}R^{-sp}\;\textup{Tail}^{p-1}(u_+(t);R,\xi_0)(w_+\psi^p\nu^2)(\xi,t)d\xi\,dt.
\end{align}

Now, let us consider $D_1^\varepsilon$
\begin{align}\label{eq4.19}
    D_1^\varepsilon &=\int_{t_1}^{t_2}\iint_{B_r\times B_r}{|u(\xi,t)-u(\eta,t)|^{p-2}(u(\xi,t)-u(\eta,t))}((w_+^\varepsilon\psi^p\nu^2)^\varepsilon(\xi,t)-(w_+^\varepsilon\psi^p\nu^2)^\varepsilon(\eta,t))d\mu\,dt\nonumber\\
    &=\underbrace{\int_{t_1}^{t_2}\iint_{B_r\times B_r}{|u(\xi,t)-u(\eta,t)|^{p-2}(u(\xi,t)-u(\eta,t))}((w_+^\varepsilon\psi^p\nu^2)^\varepsilon(\xi,t)-(w_+\psi^p\nu^2)(\xi,t))d\mu\,dt}_{J_1}\nonumber\\
    &+ \underbrace{\int_{t_1}^{t_2}\iint_{B_r\times B_r}{|u(\xi,t)-u(\eta,t)|^{p-2}(u(\xi,t)-u(\eta,t))}((w_+\psi^p\nu^2)(\eta,t)-(w_+^\varepsilon\psi^p\nu^2)^\varepsilon(\eta,t))d\mu\,dt}_{J_2}\nonumber\\
    &+ \underbrace{\int_{t_1}^{t_2}\iint_{B_r\times B_r}{|u(\xi,t)-u(\eta,t)|^{p-2}(u(\xi,t)-u(\eta,t))}((w_+\psi^p\nu^2)(\xi,t)-(w_+\psi^p\nu^2)(\eta,t))d\mu\,dt}_{D_1},
\end{align} where $d\mu = K(\xi,\eta,t)\,d\xi\,d\eta$. 

Now, let us consider $J_1$,
\begin{align}\label{eq4.20}
    J_1\lesssim \int_{t_1}^{t_2}\iint_{B_r\times B_r}\frac{|u(\xi,t)-u(\eta,t)|^{p-1}}{|\eta^{-1}\circ \xi|_{\hn}^{(\frac{Q}{p}+s)(p-1)}} \;\frac{|(w_+^\varepsilon\psi^p\nu^2)^\varepsilon(\xi,t)-(w_+\psi^p\nu^2)(\xi,t)|}{|\eta^{-1}\circ \xi|_{\hn}^{\frac{Q}{p}+s}}d\xi\,d\eta\,dt
\end{align}
Now, from Lemma \ref{L4.4}, we have $\{(w_+^\varepsilon \psi^p\nu^2)^\varepsilon\}_{\varepsilon\in (0,\varepsilon_0/2)}$ has an  $\varepsilon$-independent bound in the space $L^p((t_1,t_2);W^{s,p}(B_r))$ i.e. there exists some constant $c=c(Q,p)$ 
\begin{align}\label{eq4.21}
\Bigg|\Bigg|\frac{(w_+^\varepsilon\psi^p\nu^2)^\varepsilon(\xi,t)}{|\eta^{-1}\circ \xi|_\hn^{\frac{Q}{p}+s}}\Bigg|\Bigg|_{L^p(t_1,t_2;L^p(B_r\times B_r))}\leq c
\end{align}
and also considering pointwise convergence of $(w_+^\varepsilon\psi^p\nu^2)^\varepsilon\rightarrow (w_+\psi^p\nu^2)$ a.e. in $B_r\times (t_1,t_2)$. Then we have
\begin{align}\label{eq4.22}
    	\frac{(w_+^\varepsilon\psi^p\nu^2)^\varepsilon(\xi,t)}{|\eta^{-1}\circ \xi|_{\hn}^{\frac{Q}{p}+s}} \rightharpoonup \frac{(w_+\psi^p\nu^2)(\xi,t)}{|\eta^{-1}\circ \xi|_{\hn}^{\frac{Q}{p}+s}}\;\;\;\;\;\mbox{in }\;L^p(t_1,t_2;L^p(B_r\times B_r))
\end{align}
On the other hand, it is easy to check
\begin{align}\label{eq4.23}
    \frac{|u(\xi,t)-u(\eta,t)|^{p-2}(u(\xi,t)-u(\eta,t))}{|\eta^{-1}\circ \xi|_{\hn}^{(\frac{Q}{p}+s)(p-1)}}\in L^\frac{p}{p-1}(t_1,t_2;L^\frac{p}{p-1}(B_r\times B_r))
\end{align}
Combining (\ref{eq4.21}), (\ref{eq4.22}) and (\ref{eq4.23}), we have in (\ref{eq4.20})
\begin{align}\label{eq4.24}
    J_1\rightarrow 0\;\;\;\;\;\;\;\mbox{as }\varepsilon\rightarrow 0
\end{align}
By similar arguments, we have
\begin{align}\label{eq4.25}
    J_2\rightarrow 0\;\;\;\;\;\;\;\mbox{as }\varepsilon\rightarrow 0
\end{align}
Hence from (\ref{eq4.24}) and (\ref{eq4.25}), we have
from (\ref{eq4.19}) as $\varepsilon\rightarrow 0$
\begin{align}\label{eq4.26}
    D_1^\varepsilon \rightarrow D_1:= \int_{t_1}^{t_2}\iint_{B_r\times B_r}{|u(\xi,t)-u(\eta,t)|^{p-2}(u(\xi,t)-u(\eta,t))}((w_+\psi^p\nu^2)(\xi,t)-(w_+\psi^p\nu^2)(\eta,t))d\mu\,dt.
\end{align}
Consider $D_4^\varepsilon$
\begin{align}\label{eq4.27}
    D_4^\varepsilon&=\int_{B_r}\Big\{u^\varepsilon(\xi,t_2-\varepsilon/2)(w_+^\varepsilon\psi^p\nu^2)(\xi,t_2-\varepsilon/2)-u(\xi,t_2)(w_+^\varepsilon\psi^p\nu^2)^\varepsilon(\xi,t_2)\Big\}d\xi\nonumber\\
    &=\int_{B_r}\Big(u^\varepsilon(\xi,t_2-\varepsilon/2)-u(\xi,t_2)\Big)(w_+^\varepsilon\psi^p\nu^2)(\xi,t_2-\varepsilon/2)d\xi\nonumber\\
    & \quad+\int_{B_r}u(\xi,t_2)\Big\{(w_+^\varepsilon\psi^p\nu^2)(\xi,t_2-\varepsilon/2)-(w_+\psi^p\nu^2)(\xi,t_2)\Big\}d\xi\nonumber\\
    &\qquad+\int_{B_r}u(\xi,t_2)\Big\{(w_+\psi^p\nu^2)(\xi,t_2)-(w_+^\varepsilon\psi^p\nu^2)^\varepsilon(\xi,t_2)\Big\}d\xi
\end{align}
Since $u\in C(t_1,t_2;L^2(B_r))$, we use \cref{L4.3} and \cref{L4.5} successively to get
\begin{align}\label{eq4.28}
    D_4^\varepsilon\rightarrow 0\;\;\;\;\;\;\mbox{as }\;\;\; \varepsilon\rightarrow 0
\end{align}
By a similar argument, we can conclude the following
\begin{align}\label{eq4.29}
    D_5^\varepsilon\rightarrow 0\;\;\;\;\;\; \mbox{as }\;\;\;\varepsilon\rightarrow 0
\end{align}
We will now pass to the limit in the expression $\bar{E}(\varepsilon)$. Integrating by parts, we can rewrite $\bar{E}(\varepsilon)$ as follows, where $\phi(\xi,\tau)=(w^\varepsilon\psi^p\nu^2)(\xi,\tau)$.
\begin{align*}
    \bar{E}(\varepsilon)&= -\int_{B_r}\int_{t_1-\varepsilon/2}^{t_1+\varepsilon/2}\Bigg(\frac{1}{\varepsilon}\int_{t_1}^{\tau+\varepsilon/2}u(\xi,t)\gamma\Big(\frac{\tau-t}{\varepsilon}\Big)dt\Bigg)\phi_\tau(\xi,\tau)d\tau d\xi  \\
    &\quad-\int_{B_r}\int_{t_2-\varepsilon/2}^{t_2+\varepsilon/2}\Bigg(\frac{1}{\varepsilon}\int_{\tau-\varepsilon/2}^{t_2}u(\xi,t)\gamma\Big(\frac{\tau-t}{\varepsilon}\Big)dt\Bigg)\phi_\tau(\xi,\tau)d\tau d\xi\\
    & =-\intbr[r]\Bigg(\frac{1}{\varepsilon}\int_{t_1}^{t_1+\varepsilon}u(\xi,t)\;\gamma\Big(\frac{t_1-t}{\varepsilon}+\frac{1}{2}\Big)dt\Bigg)(w^\varepsilon\psi^p\nu^2)(\xi,t_1+\varepsilon/2)d\xi \\
    &\quad+ \intbr[r]\int_{t_1-\varepsilon/2}^{t_1+\varepsilon/2}\Bigg(\frac{1}{\varepsilon^2}\int_{t_1}^{\tau+\varepsilon/2}u(\xi,t)\gamma_\tau\Big(\frac{\tau-t}{\varepsilon}\Big)dt\Bigg)(w^\varepsilon\psi^p\nu^2)(\xi,\tau)d\tau d\xi\\
    &\qquad+\intbr[r]\Bigg(\frac{1}{\varepsilon}\int_{t_2-\varepsilon}^{t_2}u(\xi,t)\;\gamma\Big(\frac{t_2-t}{\varepsilon}-\frac{1}{2}\Big)dt\Bigg)(w^\varepsilon\psi^p\nu^2)(\xi,t_2-\varepsilon/2)d\xi \\
    &\quad\qquad- \intbr[r]\int_{t_2-\varepsilon/2}^{t_2+\varepsilon/2}\Bigg(\frac{1}{\varepsilon^2}\int_{\tau-\varepsilon/2}^{t_2}u(\xi,t)\gamma_\tau\Big(\frac{\tau-t}{\varepsilon}\Big)dt\Bigg)(w^\varepsilon\psi^p\nu^2)(\xi,\tau)d\tau d\xi.
    \end{align*}
    Further, by a change of variables, we can rewrite $\bar{E}(\varepsilon)$ as follows:
    \begin{align*}
    \bar{E}(\varepsilon)& =-\intbr[r]\Bigg(\int_{-1/2}^{1/2}u(\xi,t_1-\varepsilon\alpha+\varepsilon/2)\;\gamma(\alpha)d\alpha\Bigg)(w^\varepsilon\psi^p\nu^2)(\xi,t_1+\varepsilon/2)d\xi\\
    &\quad+\intbr[r]\int_{-1/2}^{1/2}\Bigg(\int_{-1/2}^\alpha u(\xi,\varepsilon\alpha+t_1-\varepsilon\beta)\;\gamma_\tau(\beta)d\beta\Bigg)\phi(\xi,\varepsilon\alpha+t_1)d\alpha d\xi\\
    &\qquad+\intbr[r]\Bigg(\int_{-1/2}^{1/2}u(\xi,t_2-\varepsilon\alpha-\varepsilon/2)\;\gamma(\alpha)d\alpha\Bigg)(w^\varepsilon\psi^p\nu^2)(\xi,t_2-\varepsilon/2)d\xi\\
    &\qquad\quad-\intbr[r]\int_{-1/2}^{1/2}\Bigg(\int_{\alpha}^{1/2} u(\xi,\varepsilon\alpha+t_2-\varepsilon\beta)\;\gamma_\tau(\beta)d\beta\Bigg)(w^\varepsilon\psi^p\nu^2)(\xi,\varepsilon\alpha+t_2)d\alpha d\xi\\
    &=\underbrace{-\int_{-1/2}^{1/2}\gamma(\alpha)\Bigg\{\intbr[r]u(\xi,t_1-\varepsilon\alpha+\varepsilon/2)(w^\varepsilon\psi^p\nu^2)(\xi,t_1+\varepsilon/2)d\xi\Bigg\}d\alpha}_{J_1}\\
    &\quad+\underbrace{\int_{-1/2}^{1/2}\int_{-1/2}^\alpha \gamma_\tau(\beta)\Bigg\{\intbr[r]u(\xi,\varepsilon\alpha+t_1-\varepsilon\beta)(w^\varepsilon\psi^p\nu^2)(\xi,\varepsilon\alpha+t_1)\Bigg\}d\beta d\alpha}_{J_2}\\
    &\qquad+\underbrace{\int_{-1/2}^{1/2}\gamma(\alpha)\Bigg\{\intbr[r]u(\xi,t_2-\varepsilon\alpha-\varepsilon/2)(w^\varepsilon\psi^p\nu^2)(\xi,t_2-\varepsilon/2)d\xi\Bigg\}d\alpha}_{J_3}\\
    &\quad\qquad-\underbrace{\int_{-1/2}^{1/2}\int_{\alpha}^{1/2} \gamma_\tau(\beta)\Bigg\{\intbr[r]u(\xi,\varepsilon\alpha+t_2-\varepsilon\beta)(w^\varepsilon\psi^p\nu^2)(\xi,\varepsilon\alpha+t_2)\Bigg\}d\beta d\alpha}_{J_4}.
\end{align*}
In the last line, we have used Fubini's theorem. Now consider each part separately.
\begin{equation*}
	\begin{aligned}
    J_1& \stackrel{\phantom{Theorem 111}}{=} \int_{-1/2}^{1/2}\gamma(\alpha)\Bigg\{\intbr[r]u(\xi,t_1-\varepsilon\alpha+\varepsilon/2)w^\varepsilon_+(\xi,t_1+\varepsilon/2)\psi^p(\xi)d\xi\Bigg\}\nu^2(t_1+\varepsilon/2)d\alpha\\
    & \stackrel{\phantom{Theorem 111}}{\leq} \int_{-1/2}^{1/2}\gamma(\alpha)\Bigg\{\|u(\cdot,t_1-\varepsilon\alpha+\varepsilon/2)\|_{L^2(B_r)}\|w^\varepsilon_+(\cdot,t_1+\varepsilon/2)\|_{L^2(B_r)}\Bigg\}\nu^2(t_1+\varepsilon/2)d\alpha\\
    & \stackrel{\phantom{Theorem 111}}{\leq} \int_{-1/2}^{1/2}\gamma(\alpha)\Bigg\{\|u\|_{L^\infty_tL^2_\xi}\|w^\varepsilon_+\|_{L^\infty_tL^2_\xi}\Bigg\}\nu^2(t_1+\varepsilon/2)d\alpha\\
    & \stackrel{\cref{L4.3}}{\leq} \int_{-1/2}^{1/2}\gamma(\alpha)\Bigg\{\|u\|_{L^\infty_tL^2_\xi}\|w_+\|_{L^\infty_tL^2_\xi}\Bigg\}\nu^2(t_1+\varepsilon/2)d\alpha
    \end{aligned}
\end{equation*}
By pointwise convergence of $\nu(t_1+\varepsilon/2)\to\nu(t_1)=0$, we have as $\varepsilon\to 0$
\begin{align}\label{eq4.30}
    J_1\rightarrow 0.
\end{align}
By similar arguments, we can conclude the followings 
\begin{equation}\label{eq4.31}
    \mbox{ as}\;\;\;\varepsilon\rightarrow 0\;\;\;
    \begin{cases}
		& J_2\rightarrow 0,\\
		& J_3\rightarrow \int_{-1/2}^{1/2}\gamma(\alpha)\Bigg\{\intbr[r] u(\xi,t_2)(w_+\psi^p\nu^2)(\xi,t_2)d\xi\Bigg\}d\alpha,\\
		& J_4\rightarrow - \int_{-1/2}^{1/2}\int_{\alpha}^{1/2}\gamma_\tau(\beta)\Bigg\{\intbr[r]u(\xi,t_2)(w_+\psi^p\nu^2)(\xi,t_2)d\xi\Bigg\}d\beta d\alpha
	\end{cases}
\end{equation}
Hence, we have as $\varepsilon\rightarrow 0$ and $\int_{-1/2}^{1/2}\gamma(\xi)d\xi=1$
\begin{align}\label{eq4.325}
    \bar{E}(\varepsilon)&\rightarrow \int_{-1/2}^{1/2}\gamma(\alpha)d\alpha\intbr[r]u(\xi,t_2)(w_+\psi^p\nu^2)(\xi,t_2)d\xi- \int_{-1/2}^{1/2}\int_{\alpha}^{1/2}\gamma_\tau(\beta)d\beta d\alpha\intbr[r]u(\xi,t_2)(w_+\psi^p\nu^2)(\xi,t_2)d\xi\nonumber\\
    &=2\intbr[r]u(\xi,t_2)\phi(\xi,t_2)d\xi \nonumber\\
    &:= \mathbb{E}
\end{align}
Now consider $D_2^\varepsilon$. Before proceeding, let us denote
\begin{align*}
    D_2:= \int_{t_1}^{t_2}\int_{\hn\setminus B_r}\intbr[r]|u(\xi,t)-u(\eta,t)|^{p-2}(u(\xi,t)-u(\eta,t))(w_+\psi^p\nu^2)(\xi,t)d\mu\,dt
\end{align*}
Now consider
\begin{align}\label{eq4.33}
    D_2^\varepsilon-D_2&\lesssim \int_{t_1}^{t_2}\int_{\hn\setminus B_r}\intbr[r]\frac{|u(\xi,t)-u(\eta,t)|^{p-1}}{|\eta^{-1}\circ \xi|^{Q+sp}_{\hn}}\big|(w_+^\varepsilon\psi^p\nu^2)^\varepsilon(\xi,t)-(w_+\psi^p\nu^2)(\xi,t)\big|d\xi\,d\eta\,dt\nonumber\\
    &\leq c \int_{t_1}^{t_2}\intbr[r]\big|(w_+^\varepsilon\psi^p\nu^2)^\varepsilon(\xi,t)-(w_+\psi^p\nu^2)(\xi,t)\big|\Bigg\{\int_{\hn\setminus B_r}\frac{|u(\xi,t)|^{p-1}+|u(\eta,t)|^{p-1}}{|\eta^{-1}\circ \xi|_{\hn}^{Q+sp}}d\eta\Bigg\}d\xi\,dt\nonumber\\
    &\leq c\underbrace{\int_{t_1}^{t_2}\intbr[r]\big|(w_+^\varepsilon\psi^p\nu^2)^\varepsilon(\xi,t)-(w_+\psi^p\nu^2)(\xi,t)\big|\big|u(\xi,t)\big|^{p-1}\Bigg\{\int_{\hn\setminus B_r}\frac{d\eta}{|\eta^{-1}\circ \xi|_{\hn}^{Q+sp}}\Bigg\}d\xi\,dt}_{J_1}\nonumber\\
    &\qquad+c \underbrace{\int_{t_1}^{t_2}\intbr[r]\big|(w_+^\varepsilon\psi^p\nu^2)^\varepsilon(\xi,t)-(w_+\psi^p\nu^2)(\xi,t)\big|\Bigg\{\int_{\hn\setminus B_r}\frac{|u(\eta,t)|^{p-1}}{|\eta^{-1}\circ \xi|_{\hn}^{Q+sp}}\;d\eta\Bigg\}d\xi\,dt}_{J_2}
\end{align}
Now, let us calculate an inequality which will be useful to estimate the rest part of $D_2^\varepsilon$. Let $\xi\in \mbox{supp }(\psi)\subset B_r$. By triangle inequality in Heisenberg norm
\begin{align}\label{eq4.34}
    |\eta^{-1}\circ \xi|_{\hn}\geq |\eta^{-1}\circ \xi_0|_{\hn}-|\xi_0^{-1}\circ \xi|_{\hn}
\end{align}
If $\eta\in \hn\setminus B_r(\xi_0)$ and $\xi\in \mbox{supp }(\psi)\subset B_{\frac{r+l}2}(\xi_0)$. Then we obtain
\begin{align}\label{eq4.35}
    |\xi_0^{-1}\circ \xi|_{\hn}<\frac{r+l}{2}\;\;\;\;\mbox{and}\;\;\;\;|\eta^{-1}\circ \xi_0|_{\hn}>r
\end{align}
Then (\ref{eq4.35}) implies from (\ref{eq4.34})
\begin{align}\label{eq4.36}
    |\eta^{-1}\circ \xi|_{\hn}\geq \frac{r-l}{2}
\end{align}
From (\ref{eq4.34}) and (\ref{eq4.36})
\begin{align}\label{eq4.37}
    |\eta^{-1}\circ \xi|_{\hn}>\frac{r-l}{2r}\;|\eta^{-1}\circ \xi_0|_{\hn}
\end{align}
Using the inequality (\ref{eq4.37}) and \cite[Lemma 2.6]{MPPP23}
\begin{align*}
    J_1&\leq c\Bigg(\frac{r}{r-l}\Bigg)^{Q+sp} \int_{t_1}^{t_2}\Bigg[\intbr[r]\big|(w_+^\varepsilon\psi^p\nu^2)^\varepsilon(\xi,t)-(w_+\psi^p\nu^2)(\xi,t)\big|\big|u(\xi,t)\big|^{p-1}d\xi\int_{\hn\setminus B_r}\frac{d\eta}{|\xi_0^{-1}\circ \eta|^{Q+sp}_{\hn}}\Bigg]dt\\
    & \leq \Bigg(\frac{r}{r-l}\Bigg)^{Q+sp} r^{-sp}\int_{t_1}^{t_2}\Bigg(\int_{B_r}\big|(w_+^\varepsilon\psi^p\nu^2)^\varepsilon(\xi,t)-(w_+\psi^p\nu^2)(\xi,t)\big|^p d\xi\Bigg)^{1/p}\Bigg(\intbr[r]|u(\xi,t)|^pd\xi\Bigg)^\frac{p-1}{p}dt
\end{align*}
Since $u\in L^p(t_1,t_2; L^p(B_r))$. So using the Lemma \ref{L4.2}, we can conclude
\begin{align*}
    J_1\rightarrow 0\;\;\;\;\;\;\mbox{as}\;\;\;\;\varepsilon\rightarrow 0
\end{align*}
Using the inequality (\ref{eq4.37}), we get 
\begin{align}\label{eq4.38}
    	J_2&\leq c\Bigg(\frac{r}{r-l}\Bigg)^{Q+sp} \int_{t_1}^{t_2}\Bigg[\intbr[r]\big|(w_+^\varepsilon\psi^p\nu^2)^\varepsilon(\xi,t)-(w_+\psi^p\nu^2)(\xi,t)\big|d\xi\int_{\hn\setminus B_r}\frac{|u(\eta,t)|^{p-1}}{|\xi_0^{-1}\circ\eta|^{Q+sp}_{\hn}}d\eta\Bigg]dt.
\end{align}
By an argument similar to the passage to limit for the term $M$ in \eqref{eq4.17}, we obtain
\begin{align*}
    	J_2\rightarrow 0\;\;\;\;\;\mbox{as}\;\;\;\;\varepsilon\rightarrow 0
\end{align*}
Consequently, we have as $\varepsilon\rightarrow 0$
\begin{align}\label{eq4.40}
    D_2^\varepsilon\rightarrow D_2
\end{align}
Finally (\ref{eq4.18}), (\ref{eq4.26}), (\ref{eq4.28}), (\ref{eq4.29}), (\ref{eq4.325}) and (\ref{eq4.40}) implies (\ref{eq4.16}) as follows
\begin{align}\label{eq4.41}
    \frac{1}{2}D_1+\frac{1}{2}D_2+D_3+\mathbb{E}\leq 0
\end{align}
we shall estimate $D_1, D_2, D_3$ and $\mathbb{E}$ separately to have the required estimate.\\
Before moving to the estimates, let's construct a set as follows
\begin{align*}
    \mbox{For every }t\in I,\;\;A_+^g(k,r,t):=\{\xi\in B_r(\xi_0): u(\xi,t)-g(t)>k\}
\end{align*}
\textbf{Estimation of $D_3$ : } Since $\nu(t_1)=0$ and $\nu(t_2)=1$, we obtain
\begin{align}\label{eq4.32}
    D_3 &=\frac{1}{2}\intbr[r](w_+^2\psi^p)(\xi,t_2)d\xi-\int_{t_1}^{t_2}\intbr[r](w_+^2\psi^p\nu^2)(\xi,t)d_t\nu(t)\;d\xi\,dt\nonumber\\
    &\quad+\mathfrak{C}\int_{t_1}^{t_2}\int_{B_r}R^{-sp}\;\textup{Tail}^{p-1}(u_+(t);R,x_0)(w_+\psi^p\nu^2)(\xi,t)d\xi\,dt
\end{align}
\textbf{Estimation of $\mathbb{E}$: } Since $\nu(t_2)=1$, we get
\begin{align*}
    \mathbb{E}&=2\intbr[r]u(\xi,t_2)w_+(\xi,t_2)\psi^p(\xi)\nu^2(t_2)d\xi\\
    &=2\int_{A_+^g(k,r,t)}u(\xi,t_2) w_+(\xi,t_2)\psi^p(\xi)d\xi
\end{align*}
\begin{align}\label{eq4.43}
    \geq 2\int_{A_+^g(k,r,t)} w^2_+(\xi,t_2)\psi^p(\xi)d\xi
\end{align}
we now consider the contribution from the double integration over the domain $B_r\times B_r$.\\
\textbf{Estimation of $D_1$ : } Recall
\begin{align}\label{eq4.44}
    	D_1 \gtrsim \int_{t_1}^{t_2}\iint_{B_r\times B_r}\frac{|u(\xi,t)-u(\eta,t)|^{p-2}(u(\xi,t)-u(\eta,t))}{|\eta^{-1}\circ \xi|_{\hn}^{Q+sp}}\big((w_+\psi^p\nu^2)(\xi,t)-(w_+\psi^p\nu^2)(\eta,t)\big)\;d\xi\,d\eta\,dt
\end{align}
Initially, we want to consider the numerator part of the integrand of (\ref{eq4.44}). Let $I=(t_1,t_2)$. Now we will consider several conditions and their outcomes as far as the numerator of the integrand is concern.
\begin{itemize}
\item If $x,y \notin A_+^g(k,r,t)$, then $w_+(\xi,t),w_+(\eta,t)=0$. Hence
\begin{align}\label{eq5.45}
    |u(\xi,t)-u(\eta,t)|^{p-2}(u(\xi,t)-u(\eta,t))\big((w_+\psi^p\nu^2)(\xi,t)-(w_+\psi^p\nu^2)(\eta,t)\big)=0
\end{align}
\item If $x\in A_+^g(k,r,t)$ and $y\in B_r\setminus A_+^g(k,r,t)$, then
\begin{align}\label{eq5.46}
    |u(\xi,t)-u(\eta,t)|^{p-2}(u(\xi,t)-u(\eta,t))&\big((w_+\psi^p\nu^2)(\xi,t)-(w_+\psi^p\nu^2)(\eta,t)\big)\nonumber\\
    &=\big(w_+(\xi,t)+w_-(\eta,t)\big)^{p-1}(w_+\psi^p\nu^2)(\xi,t)
\end{align}
Here we consider two cases.
\begin{itemize}
\item[$\circ$] If $p\geq 2$, then using  \cite[Lemma 4.1]{C17} with $\theta=0$.
\begin{align*}
    \big(w_+(\xi,t)+w_-(\eta,t)\big)^{p-1}(w_+\psi^p\nu^2)(\xi,t)&\geq \big[w_+^{p-1}(\xi,t)+w_-^{p-1}(\eta,t)\big](w_+\psi^p\nu^2)(\xi,t)\\
    &=\big[|w_+(\xi,t)-w_+(\eta,t)|^p+w_-^{p-1}(\eta,t)w_+(\xi,t)\big]\psi^p(\xi)\nu^2(t)
\end{align*}
\item[$\circ$] If $p\in (1,2)$, then using Jensen's Lemma
\begin{align*}
    \big(w_+(\xi,t)+w_-(\eta,t)\big)^{p-1}(w_+\psi^p\nu^2)(\xi,t)\geq 2^{p-2}\big[|w_+(\xi,t)-w_+(\eta,t)|^p+w_-^{p-1}(\eta,t)w_+(\xi,t)\big]\psi^p(\xi)\nu^2(t)
\end{align*}
Considering both cases, we can conclude in (\ref{eq5.46})
\begin{align}\label{eq5.47}
    |u(\xi,t)-u(\eta,t)|^{p-2}&(u(\xi,t)-u(\eta,t))\big((w_+\psi^p\nu^2)(\xi,t)-(w_+\psi^p\nu^2)(\eta,t)\big)\nonumber\\
    &\geq \min\{1,2^{p-2}\}\big[|w_+(\xi,t)-w_+(\eta,t)|^p+w_-^{p-1}(\eta,t)w_+(\xi,t)\big]\psi^p(\xi)\nu^2(t)
\end{align}
\end{itemize}
\item If $x,y\in A_+^g(\xi,r,t)$, we have 
\begin{align*}
    &|u(\xi,t)-u(\eta,t)|^{p-2}(u(\xi,t)-u(\eta,t))\big((w_+\psi^p\nu^2)(\xi,t)-(w_+\psi^p\nu^2)(\eta,t)\big)\\
    	&\quad=
	\begin{cases}
		& \big(w_+(\xi,t)-w_+(\eta,t)\big)^{p-1}\big(w_+(\xi,t)\psi^p(\xi)-w_+(\eta,t)\psi^p(\eta)\big)\nu^2(t)\;\;\;\;\;\;\;\;\mbox{if}\;\;u(\xi,t)\geq u(\eta,t)\\
	   & \big(w_+(\eta,t)-w_+(\xi,t)\big)^{p-1}\big(w_+(\eta,t)\psi^p(\eta)-w_+(\xi,t)\psi^p(\xi)\big)\nu^2(t)\;\;\;\;\;\;\;\;\mbox{if}\;\;u(\eta,t)\geq u(\xi,t)
	\end{cases}
\end{align*}
For $u(\xi,t)\geq u(\eta,t)$, we have the following.
\begin{itemize}
\item[$\circ$] If $\psi(\xi)\geq \psi(\eta)$, then
\begin{align*}
    \big(w_+(\xi,t)-w_+(\eta,t)\big)^{p-1}\big(w_+(\xi,t)\psi^p(\xi)-w_+(\eta,t)\psi^p(\eta)\big)\nu^2(t)\geq \big(w_+(\xi,t)-w_+(\eta,t)\big)^p\psi^p(\xi)\nu^2(t)
\end{align*}
\item[$\circ$] If $\psi(\xi)\leq \psi(\eta)$, then 
\begin{align*}
    &\big(w_+(\xi,t)-w_+(\eta,t)\big)^{p-1}\big(w_+(\xi,t)\psi^p(\xi)-w_+(\eta,t)\psi^p(\eta)\big)\nu^2(t)\\
    &\quad=\Big[\big(w_+(\xi,t)-w_+(\eta,t)\big)^p\psi^p(\eta)-\underbrace{\big(w_+(\xi,t)-w_+(\eta,t)\big)^{p-1}w_+(\xi,t)\big(\psi^p(\eta)-\psi^p(\xi)\big)}_{M}\Big]\nu^2(t)
\end{align*}
\end{itemize}
Now, apply \cite[Lemma 4.3]{C17} with $a=\psi(\eta), b=\psi(\xi)$ and $\varepsilon=\frac{1}{2}\frac{w_+(\xi,t)-w_+(\eta,t)}{w_+(\xi,t)}$ to obtain
\begin{align*}
    M\leq \frac{1}{2}\big(w_+(\xi,t)-w_+(\eta,t)\big)^p\psi^p(\eta)+\big[2(p-1)\big]^{p-1}w_+^p(\xi,t)\big(\psi(\eta)-\psi(\xi)\big)^p
\end{align*}
\end{itemize}
Combining all the cases, we obtain
\begin{align}\label{eq5.48}
    &|u(\xi,t)-u(\eta,t)|^{p-2}(u(\xi,t)-u(\eta,t))\big((w_+\psi^p\nu^2)(\xi,t)-(w_+\psi^p\nu^2)(\eta,t)\big)\nonumber\\
    &\quad\geq \Big[c\big(w_+(\xi,t)-w_+(\eta,t)\big)^p\max\big\{\psi^p(\xi),\psi^p(\xi)\big\}-c(p)w_+^p(\xi,t)|\psi(\xi)-\psi(\eta)|^p\Big]\nu^2(t)
\end{align}
we have obtained (\ref{eq5.48}) for $u(\xi,t)\geq u(\eta,t)$. By changing the role of $\xi$ and $\eta$, we can handle the situation for $u(\xi,t)\leq u(\eta,t)$. Then finally we have
\begin{align}\label{eq5.49}
    &|u(\xi,t)-u(\eta,t)|^{p-2}(u(\xi,t)-u(\eta,t))\big((w_+\psi^p\nu^2)(\xi,t)-(w_+\psi^p\nu^2)(\eta,t)\big)\nonumber\\
    &\quad\geq \Big[c\big(w_+(\xi,t)-w_+(\eta,t)\big)^p\max\big\{\psi^p(\xi),\psi^p(\xi)\big\}-c(p)\max\big\{w_+^p(\xi,t),w_+^p(\eta,t)\big\}|\psi(\xi)-\psi(\eta)|^p\Big]\nu^2(t)
\end{align}
Consequently,
\begin{align}\label{eq5.50}
    	D_1&\geq c \Bigg[\int_{t_1}^{t_2}\iint_{B_r\times B_r}\frac{|w_+(\xi,t)-w_+(\eta,t)|^p}{|\eta^{-1}\circ \xi|^{Q+sp}_{\hn}}\max\big\{\psi^p(\xi),\psi^p(\eta)\big\}\nu^2(t)\; d\xi\,d\eta\,dt\nonumber\\
        &\quad+\int_{t_1}^{t_2}\intbr[r](w_+\psi^p\nu^2)(\xi,t)\Bigg\{\intbr[r]\frac{w^{p-1}_-(\eta,t)}{|\eta^{-1}\circ \xi|^{Q+sp}_{\hn}}d\eta\Bigg\}d\xi\,dt\nonumber\\
        &\qquad-\int_{t_1}^{t_2}\iint_{B_r\times B_r}\max\big\{w^p_+(\xi,t),w^p_+(\eta,t)\big\}\frac{|\psi(\xi)-\psi(\eta)|^p}{|\eta^{-1}\circ \xi|^{Q+sp}_{\hn}}\nu^2(t)\;d\xi\,d\eta\,dt\Bigg]
\end{align}
\textbf{Estimation of $D_2$: } Recall
\begin{align}\label{eq5.51}
    D_2&\gtrsim\int_{t_1}^{t_2}\int_{\hn\setminus B_r}\intbr[r]\gag[u](w_+\psi^p\nu^2)(\xi,t)\;d\xi\,d\eta\,dt\nonumber\\
    &= \int_{t_1}^{t_2}\intbr[r](w_+\psi^p\nu^2)(\xi,t)\Bigg\{\int_{\hn\setminus B_r}\gag[u]d\eta\Bigg\}d\xi\,dt\nonumber\\
    &\geq c\underbrace{\int_{t_1}^{t_2}\int_{A_+^g(k,r,t)}(w_+\psi^p\nu^2)(\xi,t)\Bigg\{\int_{(B_R(\xi)\cap \{u(\xi,t)>u(\eta,t)\})\setminus B_r}\frac{\big(u(\xi,t)-u(\eta,t)\big)^{p-1}}{|\eta^{-1}\circ \xi|_{\hn}^{Q+sp}}d\eta\Bigg\}d\xi\,dt}_{M_1}\nonumber\\
    &\quad-c\underbrace{\int_{t_1}^{t_2}\int_{A_+^g(k,\frac{r+l}{2},t)}(w_+\psi^p\nu^2)(\xi,t)\Bigg\{\int_{\{u(\eta,t)>u(\xi,t)\}\setminus B_r}\frac{\big(u(\eta,t)-u(\xi,t)\big)^{p-1}}{|\eta^{-1}\circ \xi|_{\hn}^{Q+sp}}d\eta\Bigg\}d\xi\,dt}_{M_2}
\end{align}
{\em Estimation of $M_1$:}
\begin{align*}
    M_1&\geq \int_{t_1}^{t_2}\intbr[r](w_+\psi^p\nu^2)(\xi,t)\Bigg[\int_{\big(B_R(\xi)\cap A_-^g(k,t)\big)\setminus B_r}\frac{\big(w_+(\xi,t)+w_-(\eta,t)\big)^{p-1}}{|\eta^{-1}\circ \xi|_{\hn}^{Q+sp}}d\eta\Bigg]d\xi\,dt\\
    &=\int_{t_1}^{t_2}\intbr[r](w_+\psi^p\nu^2)(\xi,t)\Bigg[\int_{B_R(\xi)\setminus B_r}\frac{w_-^{p-1}(\eta,t)}{|\eta^{-1}\circ \xi|_{\hn}^{Q+sp}}d\eta\Bigg]d\xi\,dt
\end{align*}
{\em Estimation of $M_2$:}
\begin{align*}
    M_2&\leq \int_{t_1}^{t_2}\int_{B_\frac{r+s}{2}}(w_+\psi^p\nu^2)(\xi,t)\Bigg\{\int_{\hn\setminus B_r}\frac{w_+^{p-1}(\eta,t)}{|\eta^{-1}\circ \xi|_{\hn}^{Q+sp}}d\eta\Bigg\}d\xi\,dt\\
    &\leq c\int_{t_1}^{t_2}\intbr[r](w_+\psi^p\nu^2)(\xi,t)\Bigg\{\int_{\hn\setminus B_r}\frac{w_+^{p-1}(\eta,t)}{|\eta^{-1}\circ \xi|_{\hn}^{Q+sp}}d\eta\Bigg\}d\xi\,dt
\end{align*}
combining all these estimates, we conclude in (\ref{eq5.51})
\begin{align*}
    D_2\geq \underbrace{c\int_{t_1}^{t_2}\intbr[r](w_+\psi^p\nu^2)(\xi,t)\Bigg[\int_{\big(B_R(\xi)\cap A_-^g(k,t)\big)\setminus B_r}\frac{\big(w_+(\xi,t)+w_-(\eta,t)\big)^{p-1}}{|\eta^{-1}\circ \xi|_{\hn}^{Q+sp}}d\eta\Bigg]d\xi\,dt}_{\mbox{\small (we can ignore this term which is on the left side)}}
\end{align*}
\begin{align}\label{eq5.52}
    -c \int_{t_1}^{t_2}\intbr[r](w_+\psi^p\nu^2)(\xi,t)\Bigg\{\int_{\hn\setminus B_r}\frac{w_+^{p-1}(\eta,t)}{|\eta^{-1}\circ \xi|_{\hn}^{Q+sp}}d\eta\Bigg\}d\xi\,dt
\end{align}
Now substituting estimations of $D_1,D_2,D_3$ and $\mathbb{E}$ in \cref{eq4.41}, we obtain
\begin{align}\label{eq5.53}
    \intbr[r]&(w_+^2\psi^p)(\xi,t_2)d\xi+\underbrace{\int_{t_0-\theta_2}^{t_0}\iint_{B_r\times B_r}\frac{|w_+(\xi,t)-w_+(\eta,t)|^p}{|\eta^{-1}\circ \xi|^{Q+sp}_{\hn}}\max\big\{\psi^p(\xi),\psi^p(\eta)\big\}\nu^2(t)\; d\xi\,d\eta\,dt}_{I}\nonumber\\
    &\qquad+\int_{t_0-\theta_2}^{t_0}\intbr[r](w_+\psi^p\nu^2)(\xi,t)\Bigg\{\intbr[r]\frac{w^{p-1}_-(\eta,t)}{|\eta^{-1}\circ \xi|^{Q+sp}_{\hn}}d\eta\Bigg\}d\xi\,dt\nonumber\\
    &\lesssim \underbrace{\int_{t_0-\theta_2}^{t_0}\iint_{B_r\times B_r}\max\big\{w^p_+(\xi,t),w^p_+(\eta,t)\big\}\frac{|\psi(\xi)-\psi(\eta)|^p}{|\eta^{-1}\circ \xi|^{Q+sp}_{\hn}}\nu^2(t)\;d\xi\,d\eta\,dt}_{II}\nonumber\\
    &\quad+\int_{t_0-\theta_2}^{t_0}\intbr[r](w_+\psi^p\nu^2)(\xi,t)d_t\nu(t)\;d\xi\,dt +\int_{t_0-\theta_2}^{t_0}\intbr[r](w_+\psi^p\nu^2)(\xi,t)\Bigg\{\int_{\hn\setminus B_r}\frac{w_+^{p-1}(\eta,t)}{|\eta^{-1}\circ \xi|_{\hn}^{Q+sp}}d\eta\Bigg\}d\xi\,dt\nonumber\\
    &\qquad -\mathfrak{C}\int_{t_0-\theta_2}^{t_0}\int_{B_r}R^{-sp}\;\textup{Tail}^{p-1}(u_+(t);R,x_0)(w_+\psi^p\nu^2)(\xi,t)d\xi\,dt
\end{align}
First, by the estimate for $I_2$ in \cite[Lemma 3.3]{Ding2021}, we get
\begin{align}
    I \geq \frac12\int_{t_0-\theta_2}^{t_0}\iint_{B_r\times B_r}\frac{|(w_+\psi\nu^{\frac2p})(\xi,t)-(w_+\psi\nu^{\frac2p})(\eta,t)|^p}{|\eta^{-1}\circ \xi|^{Q+sp}_{\hn}}\; d\xi\,d\eta\,dt - c\,II
\end{align}
By the properties of $\psi$, we get
\begin{align}\label{eq5.54}
    \frac{|\psi(\xi)-\psi(\eta)|}{|\eta^{-1}\circ \xi|_{\hn}}\leq \frac{c}{r-l}
\end{align}
Since $0\leq \eta\leq 1$, we have
\begin{align}\label{eq5.55}
    II&\leq \int_{t_0-\theta_2}^{t_0}\Bigg[\iint_{B_r\times B_r}\max\big\{w^p_+(\xi,t),w^p_+(\eta,t)\big\}\frac{1}{(r-s)^p}\frac{1}{|\eta^{-1}\circ \xi|^{-(1-s)p+Q}_{\hn}}\;d\xi\,d\eta\Bigg]dt\nonumber\\
    &\leq \frac{2c}{(r-l)^p}\int_{t_0-\theta_2}^{t_0}\Bigg[\iint_{B_r\times B_r}w_+^p(\xi,t)\frac{1}{|\eta^{-1}\circ \xi|^{-(1-s)p+Q}_{\hn}}\;d\xi\,d\eta\Bigg]dt\nonumber\\
    &= \frac{2c}{(r-l)^p}\int_{t_0-\theta_2}^{t_0}\intbr[r]w_+^p(\xi,t)\Bigg[\intbr[r]\frac{d\eta}{|\eta^{-1}\circ \xi|^{Q-(1-s)p}_{\hn}}\Bigg]d\xi\,dt\nonumber\\
    &\leq \frac{cr^{p(1-s)}}{(r-l)^p}\int_{t_0-\theta_2}^{t_0}\intbr[r]w_+^p(\xi,t)\Bigg[\intbr[r]\frac{d\eta}{|\eta^{-1}\circ \xi|^{Q}_{\hn}}\Bigg]d\xi\,dt\nonumber\\
    &=c\Bigg(\frac{r^{1-s}}{r-l}\Bigg)^p\|w_+\|^p_{L^p(t_0-\theta_2,t_0;L^p(B_r))}.
\end{align}
Lastly, we separately choose $t_2=t_0$ and $t_2\in (t_1,t_0]$ with the property
\begin{align}\label{eq5.57}
    \intbr[r](w_+^2\psi^p)(\xi,t_2)d\xi= \esssup_{t_1<t<t_0}\intbr[r](w_+^2\psi^p)(\xi,t)d\xi
\end{align}
\end{proof}

Now, we will do a further calculation by splitting the term on the right hand side (RHS) involving the tail into a middle distance term and a faraway tail term, which will cancel with the final term in the previous Caccioppoli inequality.

\begin{lemma}\label{L4.7}
    \textbf{(Caccioppolli Inequality II)} Let $(\xi_0,t_0)\in \Omega_T$. Let $\sigma, l,r,R\in \mathbb{R}_+$ satisfy $0<\sigma R\leq l<\frac{l+r}{2}\leq \frac{(5\sigma+3)R}{8}\leq R$. Let $u$ be a local sub-solution to the problem (\ref{maineq}).  Consider two non-negative functions $\psi\in C^\infty_0(B_{(r+l)/2}(\xi_0))$ with $0\leq \psi\leq 1$ in $\hn$ and $\nu\in C^\infty(\mathbb{R})$ such that $\nu(t)=0$ iff $t\leq t_0-\theta_2$ and $\nu(t)=1$ if $t\geq t_0-\theta_1$ where ball $B_r\equiv B_r(\xi_0)$ such that $\bar{B_r}\subset\Omega$ and $0\leq\theta_1<\theta_2$ satisfying $[t_0-\theta_2,t_0]\subset (0,T)$. Let $w:= u-g-k$ with a level $k\in \mathbb{R}$ and $g(t)=\mathfrak{C}R^{-sp}\int_{t_0-\theta_2}^t\textup{Tail}^{p-1}(u_+(s);R,\xi_0)ds$. Then there exist constants $c,\mathfrak{C}>0$ depending on $\mathtt{data}$ such that
	\begin{align*}
		\esssup_{t_0-\theta_2<t<t_0}&\int_{B_r}(w_+^2\psi^p)(\xi,t)d\xi\\
        &\quad+\int_{t_0-\theta_2}^{t_0}\iint_{B_r\times B_r}\frac{|w_+(\xi,t)-w_+(\eta,t)|^p}{|\eta^{-1}\circ \xi|^{Q+sp}_{\hn}}\max\big\{\psi^p(\xi),\psi^p(\eta)\big\}\nu^2(t)\; d\xi\,d\eta\,dt\\
    &\qquad	\leq c\Bigg(\frac{r^{1-s}}{r-l}\Bigg)^p\|w_+\|^p_{L^p(t_0-\theta_2,t_0;L^p(B_r))}+	\int_{t_0-\theta_2}^{t_0}\int_{B_r}(w_+^2\psi^p\nu^2)(\xi,t)d_t\nu(t)\;d\xi\,dt\\
    &\quad\quad\quad+cR^{-sp} \left(\frac{R}{r-l}\right)^{Q+sp}\int_{t_0-\theta_2}^{t_0}\int_{B_r}w_+(\xi,t)\psi^p(\xi)\nu^2(t)\Bigg(\fint_{B_R} u^{p-1}_+(\eta,t)\,d\eta\Bigg)d\xi\,dt.
    \end{align*}
\end{lemma}
\begin{proof}
  We work on the following term in the RHS of the Caccioppoli inequality from \cref{L4.6}.
    \begin{align}\label{cacc2.1}
        A:=c\Biggl[\int_{t_0-\theta_2}^{t_0}\int_{B_r}(w_+\psi^p\nu^2)(\xi,t)\Bigg(\int_{\hn\setminus B_r}\frac{w^{p-1}_+(\eta,t)}{|\eta^{-1}\circ \xi|_{\hn}^{Q+sp}}d\eta\Bigg)d\xi\,dt\Biggr].
    \end{align}
    Let $\xi\in \mbox{supp }(\psi)\subset B_r$. By triangle inequality in Heisenberg norm
\begin{align}\label{cacc2eq4.34}
    |\eta^{-1}\circ \xi|_{\hn}\geq |\eta^{-1}\circ \xi_0|_{\hn}-|\xi_0^{-1}\circ \xi|_{\hn} = |\eta^{-1}\circ \xi_0|_{\hn}\left(1-\frac{|\xi_0^{-1}\circ \xi|_{\hn}}{|\eta^{-1}\circ \xi_0|_{\hn}}\right).
\end{align}
  We write
     \begin{align*}
        \sup_{\xi\in B_{\frac{r+l}{2}(\xi_0)}}&\int_{\hn\setminus B_r}\frac{w^{p-1}_+(\eta,t)}{|\eta^{-1}\circ \xi|_{\hn}^{Q+sp}}d\eta = \sup_{\xi\in B_{\frac{r+l}{2}(\xi_0)}}\int_{B_R\setminus B_r}\frac{w^{p-1}_+(\eta,t)}{|\eta^{-1}\circ \xi|_{\hn}^{Q+sp}}d\eta + \sup_{\xi\in B_{\frac{r+l}{2}(\xi_0)}}\int_{\hn\setminus B_R}\frac{w^{p-1}_+(\eta,t)}{|\eta^{-1}\circ \xi|_{\hn}^{Q+sp}}d\eta\\
        & \leq R^{-sp}\left(\frac{2R}{r-l}\right)^{Q+sp}\fint_{B_R} w^{p-1}_+(\eta,t)\,d\eta + \left(\frac{8}{5(1-\sigma)}\right)^{Q+sp}\int_{\hn\setminus B_R}\frac{w^{p-1}_+(\eta,t)}{|\eta^{-1}\circ \xi_0|_{\hn}^{Q+sp}}d\eta\\
        & = R^{-sp}\left(\frac{2R}{r-l}\right)^{Q+sp}\fint_{B_R} u^{p-1}_+(\eta,t)\,d\eta + \left(\frac{8}{5(1-\sigma)}\right)^{Q+sp}R^{-sp}\, \textup{Tail}^{p-1}(u_+(t),R,\xi_0),
    \end{align*} where in the second inequality, we have used \cref{cacc2eq4.34} and $\frac{r+l}{2}\leq \frac{(5\sigma + 3)R}{8}$.
    Substituting the last inequality into \cref{cacc2.1}, we get
    \begin{align*}
        A &\leq cR^{-sp} \Biggl[\int_{t_0-\theta_2}^{t_0}\int_{B_r}(w_+\psi^p\nu^2)(\xi,t)\Bigg( R^{-sp}\left(\frac{2R}{r-l}\right)^{Q+sp} \fint_{B_R} u^{p-1}_+(\eta,t)\,d\eta \\
        &\qquad+ \left(\frac{8}{5(1-\sigma)}\right)^{Q+sp}R^{-sp}\textup{Tail}^{p-1}(u_+(t),R,\xi_0)\Bigg)d\xi\,dt\Biggr]\\
        &=c R^{-sp} \frac{(2R)^{Q+sp}}{(r-l)^{Q+sp}} \int_{t_0-\theta_2}^{t_0}\int_{B_r}w_+(\xi,t)\psi^p(\xi)\nu^2(t)\Bigg(\fint_{B_R} u^{p-1}_+(\eta,t)\,d\eta\Bigg)d\xi\,dt\\
        &\qquad+ c\left(\frac{8}{5(1-\sigma)}\right)^{Q+sp}R^{-sp}\int_{t_0-\theta_2}^{t_0}\int_{B_r}w_+(\xi,t)\psi^p(\xi)\nu^2(t)\textup{Tail}^{p-1}(u_+(t),R,\xi_0)d\xi\,dt
    \end{align*}

    The last term in the previous estimate gets cancelled once we choose $\mathfrak{C} = c\left(\frac{8}{5(1-\sigma)}\right)^{Q+sp}$ in the term    
    \begin{align}
        B:= \mathfrak{C} R^{-sp} \int_{t_0-\theta_2}^{t_0}\int_{B_r} \psi^p(\xi)\nu^2(t) w_+(\xi,t) \textup{Tail}\,^{p-1}(u_+(t),R,\xi_0)\,d\xi\,dt.
    \end{align}
\end{proof}

 \section{Local boundedness estimate}\label{localbound}
  This section is dedicated to proving a local boundedness estimate for \cref{maineq}. Having obtained a version of Caccioppoli inequality where the tail term is only present in the De Giorgi levels, we are now in a position to use the methods in \cite{BK24} to obtain the local boundedness estimate. 
  
\subsection{Parameters} Our choice of parameters have been inspired from those in \cite{BK24}. Let us define
\begin{align}\label{eq6.1}
    \kappa=\frac{s}{Q}
\end{align}
and 
\begin{align}\label{eq6.2}
    \hat{q}=\mathcal{B}(1+\kappa)
\end{align}
Then we note that
\begin{align}\label{eq6.3}
    \frac{1}{\hat{q}}+\frac{1}{\hat{q}}\Bigg(\frac{s}{Q}+\frac{1}{2}-\frac1p\Bigg)=\frac{1}{2},\;\;\;\;\;\hat{q}\in (2,p_s)
\end{align}
and
\begin{align}\label{eq6.4}
    p_1\leq \mathcal{B}=\frac{p+2\kappa-\frac{(p-1)(p-2)}{q}}{1+\kappa} = 2 + \frac{1-\frac2p}{1+\kappa}\leq p_2
\end{align}
with the fact that
\begin{align*}
    \begin{cases}
        & 2\leq \mathcal{B}\leq p\;\;\;\;\;\mbox{if }p\geq 2\\
        & p\leq \mathcal{B}\leq 2\;\;\;\;\;\mbox{if }p< 2
    \end{cases}
\end{align*}
and also $p_s$ defined as
\begin{align}\label{eq6.5}
    p_s= \begin{cases}
        &\frac{Qp}{Q-sp}\;\;\;\; \mbox{if } Q>sp \\
          &\infty\;\;\;\;\;\;\;\;\;\mbox{if } Q\leq sp
    \end{cases}
\end{align}
Now we define the following set. For $x_0\in \Omega$, $\rho>0$ and $k\geq 0$ with $B_\rho=B_\rho(x_0)\subset \Omega$ and $t\in (0,T)$ , we set
\begin{align}\label{eq6.6}
    \mathcal{A}^g_\pm (t;x_0,\rho,k)=\{x\in B_\rho| (u(\xi,t)-g(t)-k)_\pm>0\}
\end{align}
and also additionally we set
\begin{align}\label{eq6.8}
    \zeta=\frac{Q+2s}{Q+sp}
\end{align}
This $\zeta$ is introduced to handle singular and degenerate cases at the same time. 

In the next lemma, we refine the energy estimate of weak sub-solution u to (\ref{maineq}) in a suitable form which will be useful later.
\begin{lemma}\label{L6.1}
    Let $u$ be a weak sub-solution to the problem (\ref{maineq}). Then the following inequality holds for any $\mathcal{Q}_R(\xi_0,t_0)\Subset \Omega_T$ with $\rho\in (\sigma R,R)$ and $0<\rho<\frac{R+\rho}{2}\leq \frac{(5\sigma+3)R}{8}$.
\begin{align*}
    \int_{t_0-\rho^{sp}}^{t_0}&\iint_{B_\rho(x_0)\times B_\rho(\xi_0)} \frac{|w_+(\xi,t)-w_+(\eta ,t)|^p}{|\eta^{-1}\circ \xi|^{Q+sp}}d\xi d\eta dt+ \esssup_{t\in [t_0-\rho^{sp},t_0]}\int_{B_{\rho}(x_0)} w^2_+(\xi,t)d\xi\\
    &\leq c\frac{R^{p(1-s)}}{(R-\rho)^p}\iint_{\mathcal{Q}_{R}(z_0)}w^p_+(\xi,t)d\xi dt+\frac{c}{R^{sp}-\rho^{sp}}\iint_{\mathcal{Q}_{R}(z_0)}w^2_+(\xi,t)d\xi dt\\
    &\qquad+ck^{p_1}\Big(\frac{R}{R-\rho}\Big)^{\mathcal{B}'(Q+sp)}R^{-s}\Bigg(\int_{t_0-R^{sp}}^{t_0}|\mathcal{A}^g_+(t;\xi_0,R,k)|^p dt\Bigg)^\frac{1}{p},
\end{align*}
with $k$ satisfying
\begin{align*}
    k\geq \Big[ \fint_{t_0-R^{sp}}^{t_0} \left( \fint_{B_R(\xi_0)} u_+^{p-1}(\eta,t) \,d\eta\right)^{\frac{p}{p-1}}\,dt \Big]^\frac{\mathcal{B}'(p-1)}{pp_1}.
\end{align*}
\end{lemma}

\begin{proof}
    In the energy inequality from \cref{L4.7}, we choose $r=R$, $l=\rho$, and $\theta_2=R^{sp}$ to get
    \begin{align*}
	J_{0,1}+J_{0,2}&:=\int_{t_0-R^{sp}}^{t_0}\iint_{B_R\times B_R}\frac{|(w_+\psi\nu^{\frac2p})(\xi,t)-(w_+\psi\nu^{\frac2p})(\eta,t)|^p}{|\eta^{-1}\circ \xi|^{Q+sp}_{\hn}}\; d\xi\,d\eta\,dt\\
    &+\esssup_{t_0-R^{sp}<t<t_0}\int_{B_R}(w_+^2\psi^p)(\xi,t)d\xi\\
    &\qquad	\leq c\Bigg(\frac{R^{1-s}}{R-\rho}\Bigg)^p\|w_+\|^p_{L^p(t_0-R^{sp},t_0;L^p(B_R))}+	\int_{t_0-R^{sp}}^{t_0}\int_{B_R}(w_+^2\psi^p\nu^2)(\xi,t)d_t\nu(t)\;d\xi\,dt\\
    &\quad\quad\quad+\underbrace{cR^{-sp}\left(\frac{R}{R-\rho}\right)^{Q+sp}\int_{t_0-R^{sp}}^{t_0}\int_{B_R}w_+(\xi,t)\psi^p(\xi)\nu^2(t)\Bigg(\fint_{B_R} u^{p-1}_+(\eta,t)\,d\eta\Bigg)d\xi\,dt}_J.
    \end{align*}
    \paragraph{\textbf{Estimate of J}} We begin by noting that $G(t)=\fint_{B_R} u_+^{p-1}(\eta,t)\,d\eta\in L^{\infty,\frac{p}{p-1}}_{\textup{loc}}(\Omega_T)$. 

    We write and estimate $J$ as

    \begin{align*}
       J &= cR^{-sp} \left(\frac{R}{R-\rho}\right)^{Q+sp} \int_{t_0-R^{sp}}^{t_0}\int_{B_r}G(t)\,w_+(\xi,t)\psi^p(\xi)\nu^2(t)d\xi\,dt\\
       & \leq cR^{-sp} \left(\frac{R}{R-\rho}\right)^{Q+sp} \int_{t_0-R^{sp}}^{t_0}\int_{B_r}G(t)\,w_+(\xi,t)(\psi\nu^{\frac 2p})^{\frac{p_2}{2}}(\xi,t)d\xi\,dt,
    \end{align*} on account of $0\leq\psi\nu^{\frac2p}\leq 1$ and $\frac{p_2}2\leq p$.

    On applying H\"older's inequality and Young's inequality, we obtain the following sequence of estimates.

    \begin{align*}
        J & \leq cR^{-sp} \left(\frac{R}{R-\rho}\right)^{Q+sp} \|G\|_{L^{\infty,\frac{p}{p-1}}(\mathcal{Q}_R(\xi_0))} \|w_+(\psi\nu^{\frac 2p})^{\frac{p_2}{2}}\|_{L^{1,p}(\mathcal{Q}_R(\xi_0))}\\
        &\leq c \left(\frac{R}{R-\rho}\right)^{Q+sp} \|k^{\frac{p_1}{\mathcal{B}'}}R^{-s}w_+(\psi\nu^{\frac 2p})^{\frac{p_2}{2}}\|_{L^{1,p}(\mathcal{Q}_R(\xi_0))}\\
        &\leq \frac{ck^{p_1}}{\epsilon^{\frac{1}{\mathcal{B}-1}}}\left(\frac{R}{R-\rho}\right)^{\mathcal{B}'(Q+sp)}\frac{1}{R^s} \|\chi_{\{u>g(\cdot)+k\}}\|_{L^{1,p}(\mathcal{Q}_R(\xi_0))} + \epsilon\frac{1}{R^s}\left\|\left(w_+(\psi\nu^{\frac 2p})^{\frac{p_2}{2}}\right)^{\mathcal{B}}\right\|_{L^{1,p}(\mathcal{Q}_R(\xi_0))}.
    \end{align*}
    Once again using H\"older's inequality and \cref{T4.10} with $q=2$ and $m=pr$, we obtain
    \begin{align*}
        J & \leq \frac{ck^{p_1}}{\epsilon^{\frac{1}{\mathcal{B}-1}}}\left(\frac{R}{R-\rho}\right)^{\mathcal{B}'(Q+sp)}\frac{1}{R^s} \left(\int_{t_0-R^{sp}}^{t_0} |\mathcal{A}_+(t;\xi_0,R,g(t)+k)|^p\,dt \right)^{\frac1p}\\
        &\qquad + \epsilon\left\|w_+(\psi\nu^{\frac 2p})^{\frac{p_2}{2}}\right\|^{\mathcal{B}}_{L^{r,m}(\mathcal{Q}_R(\xi_0))}\\
        & \leq \frac{ck^{p_1}}{\epsilon^{\frac{1}{\mathcal{B}-1}}}\left(\frac{R}{R-\rho}\right)^{\mathcal{B}'(Q+sp)}\frac{1}{R^s} \left(\int_{t_0-R^{sp}}^{t_0} |\mathcal{A}_+(t;\xi_0,R,g(t)+k)|^p\,dt \right)^{\frac1p}\\
        &\qquad + \epsilon \left( \tilde{J}_{0,1} + \tilde{J}_{0,2} + R^{-sp} \iint_{\mathcal{Q}_R(\xi_0} w_+^p(\xi,t)\,d\xi\,dt \right),
    \end{align*}
    where
    \begin{align*}
        \tilde{J}_{0,1} + \tilde{J}_{0,2} &= \int_{t_0-R^{sp}}^{t_0}\iint_{B_R\times B_R}\frac{|(w_+(\psi\nu^{\frac2p})^{\frac{p_2}{2}})(\xi,t)-(w_+(\psi\nu^{\frac2p})^{\frac{p_2}{2}})(\eta,t)|^p}{|\eta^{-1}\circ \xi|^{Q+sp}_{\hn}}\; d\xi\,d\eta\,dt\\
    &\qquad+\esssup_{t_0-R^{sp}<t<t_0}\int_{B_R}(w_+(\psi\nu^{\frac2p})^{\frac{p_2}{2}})^2(\xi,t)d\xi.
    \end{align*}
    If $p_2=2$, then $\tilde{J}_{0,1} + \tilde{J}_{0,2} \leq {J}_{0,1} + {J}_{0,2} $ by $(\psi\nu^{\frac2p})^2\leq (\psi\nu^{\frac2p})^p$. On the other hand, if $p_2=p>2$, then $\tilde{J}_{0,2}=J_{0,2}$ and
    \begin{align*}
        \tilde{J}_{0,1} &\leq c \int_{t_0-R^{sp}}^{t_0}\iint_{B_R\times B_R}\frac{|(w_+\psi\nu^{\frac2p})(\xi,t)-(w_+\psi\nu^{\frac2p})(\eta,t)|^p(\psi\nu^{\frac2p})^{p(\frac{p_2}{2}-1)}(\xi,t)}{|\eta^{-1}\circ \xi|^{Q+sp}_{\hn}}\; d\xi\,d\eta\,dt\\
        &\qquad+\int_{t_0-R^{sp}}^{t_0}\iint_{B_R\times B_R}\frac{|(w_+\psi\nu^{\frac2p}(\eta,t)((\psi\nu^{\frac2p})^{\frac{p_2}{2}-1}(\xi,t)-(\psi\nu^{\frac2p})^{\frac{p_2}{2}-1}(\eta,t))|^p}{|\eta^{-1}\circ \xi|^{Q+sp}_{\hn}}\; d\xi\,d\eta\,dt\\
        & \leq c J_{0,1} + c \frac{R^{p(1-s)}}{(R-\rho)^p}\iint_{\mathcal{Q}_{R}(\xi_0)} w_+^p\,d\xi\,dt.
    \end{align*}
    In this estimate, we use a similar argument as in \cref{eq5.55}. Finally, take $\epsilon=\frac{1}{16c}$ to see that
    \begin{align*}
        J\leq ck^{p_1}\left(\frac{R}{R-\rho}\right)^{\mathcal{B}'(Q+sp)}\frac{1}{R^s} \|\chi_{\{u>g(\cdot)+k\}}\|_{L^{1,p}(\mathcal{Q}_R(\xi_0))} + \frac{1}{16}\left( J_{0,1} + J_{0,2} + \frac{R^{p(1-s)}}{(R-\rho)^p}\iint_{\mathcal{Q}_{R}(\xi_0)} w_+^p\,d\xi\,dt \right).
    \end{align*}
    Finally combining all the estimates together and sending $J_{0,1}$ and $J_{0,2}$ to the left, we get the estimate in the lemma.
\end{proof}

\subsection{Proof of local boundedness}\label{SS6.2}

 In this subsection, we prove local boundedness using De Giorgi iteration. This proof follows the one in \cite{BK24}.

 \begin{proof}(Proof of \cref{mainresult})
     Let $(\xi_0,t_0)\in\Omega_T$, $\rho>0$ and $\mathcal{Q}_{\rho,\tau}(z_0)=B_\rho(\xi_0)\times (t_0-\tau,t_0]$ with $z_0=(\xi_0,t_0)\in \hn\times \mathbb{R}$ , $\overline{B_\rho}(\xi_0)\subset \Omega$ and $[t_0-\tau,t_0]\in (0,T)$.
 For any $\xi,\eta\in\hn$ and $t\in (-1,0]$ , define
\begin{align}\label{eq6.7}
    \hat{u}(\xi,t)=u(\rho \xi+\xi_0,\rho^{sp}t+t_0).
\end{align} Then $\hat{u}$ is a weak subsolution to
\begin{align*}
    \partial \hat{u} + \mathcal{L}_{\tilde{K}} \hat{u}=0 \mbox{ in } \mathcal{Q}_1.
\end{align*}
 
 Define 
 \begin{align}\label{eq6.9}
     \hat{k}=\Big[ \fint_{-1}^{0} \left( \fint_{B_1(0)} \hat{u}_+^{p-1}(\eta,t) \,d\eta\right)^{\frac{p}{p-1}}\,dt \Big]^\frac{\mathcal{B}'(p-1)}{pp_1}+1.
 \end{align}
 Now, for any non-negative integer $j$, we set
 \begin{align}\label{eq6.10}
     \rho_j=\sigma+\frac{(1-\sigma)}{2^{j+1}}\;\;\;\mbox{and}\;\;\;k_j=d+d\Bigg(1-\frac{1}{2^j}\Bigg)
 \end{align}
where $d>2\hat{k}$ which will be determined later. Then it is obvious that
 \begin{align}\label{eq6.11}
     \rho_j-\rho_{j+1}=\frac{1-\sigma}{2^{j+2}}\;\;\;\mbox{and}\;\;\;k_{j+1}-k_j=\frac{d}{2^{j+1}}.
 \end{align}
 We also verify that
 \begin{align*}
     \frac{\rho_j+\rho_{j+1}}{2}= \sigma + \frac{3(1-\sigma)}{2^{i+3}} \leq \sigma + \frac{3(1-\sigma)}{8} = \frac{5\sigma}{8} + \frac38 < 1.
 \end{align*}
 Additionally, for iterative inequalities, we set
 \begin{align}\label{eq6.12}
     \begin{cases}
         & Y_j=\frac{1}{d^{p_1}}\fiint_{\mathcal{Q}_{\rho_j}}(\hat{u}-\hat{g}(t)-k_j)^{p_2}_+d\xi\,dt\\
         & Z_j=\Bigg(\fint_{-\rho_j^{sp}}^0\Big(\frac{|\hat{\mathcal{A}}^g_+(t;0,\rho_j,k_{j+1})|}{|B_{\rho_j}|}\Big)^p \,dt\Bigg)^\frac{1}{p(1+\kappa)}\\
         & \Pi_j=\fint_{-\rho_j^{sp}}^0\frac{|\hat{\mathcal{A}}^g_+(t;0,\rho_j,k_{j+1})|}{|B_{\rho_j}|}dt
     \end{cases}
 \end{align}
 where $\hat{\mathcal{A}}^g_\pm$ is equivalent to (\ref{eq6.6})  by changing $u$ with $\hat{u}$. 
 
 Observe from \cref{defp1p2B}, \cref{eq6.8} and \cref{eq6.12} that
 \begin{align}\label{lbesty}
     \frac{p\zeta}{p_2}\leq 1,\,\frac{p}{p_2} - \frac{Q}{Q+2s} > 0,
 \end{align} and
 \begin{align}\label{eq6.13}
     \Pi_j\leq (k_{j+1}-k_j)^{-p_2}d^{p_1}Y_j\leq C2^{jp_2}d^{p_1-p_2}Y_j.
 \end{align}
 Applying \cref{L6.1} with $\Omega_T = \mathcal{Q}_1$, $u=\hat{u}$, $z_0=0$, $\rho = \rho_{j+1}$, $R=\rho_j$, and $k=k_{j+1}$, we obtain

\begin{equation}\label{lbest2}
    \begin{aligned}
        &\left[(\hat{u}-g-k_{j+1})_+\right]^{p}_{V^{p,2}_s(\mathcal{Q}_{\rho_{j+1}})}\\
        &\qquad \leq c \frac{\rho_{j+1}^{p(1-s)}}{(\rho_j-\rho_{j+1})^p}\iint_{\mathcal{Q}_{\rho_j}} (\hat{u}-g-k_{j+1})_+^p\,d\xi\,dt +  \frac{c}{\rho_j^{sp}-\rho_{j+1}^{sp}}\iint_{\mathcal{Q}_{\rho_j}} (\hat{u}-g-k_{j+1})_+^2\,d\xi\,dt \\
        &\qquad\qquad + c\hat{k}^{p_1}\left(\frac{\rho_{j+1}}{\rho_j-\rho_{j+1}}\right)^{\mathcal{B}'(Q+sp)} \rho_j^{-s}\left(\int_{-\rho_j^{sp}}^0|\hat{\mathcal{A}}^g_+(t,0,\rho_j,k_{j+1})|^p\,dt\right)^{\frac1p} =: J_1 +J_2 + J_3.
    \end{aligned}
\end{equation}

We will now estimate each of the terms $J_1$, $J_2$ and $J_3$ on the right hand side (RHS).

\paragraph{\textbf{Estimate of $J_1$}} 

We have the following estimate for $J_1$.

\begin{align*}
    J_1 & \leq c \frac{\rho_{j+1}^{p}}{(\rho_j-\rho_{j+1})^p}\frac{1}{\rho_j^{sp}}\iint_{\mathcal{Q}_{\rho_j}} (\hat{u}-\hat{g}-k_{j+1})_+^p\,d\xi\,dt\\
    & \leq c\frac{2^{jp}}{(1-\sigma)^p} \frac{1}{\rho_j^{sp}} \iint_{\mathcal{Q}_j} \frac{(\hat{u}-\hat{g}-k_j)^{p_2}_+}{(k_{j+1}-k_j)^{p_2-p}} \,d\xi\,dt\\
    &\leq c\frac{2^{jp_2}}{(1-\sigma)^p}\frac{1}{\rho_j^{sp}}d^{-p_2+p}d^{p_1}Y_j|\mathcal{Q}_{\rho_j}|\\
    &\leq c\frac{2^{jp_2}}{(1-\sigma)^p}\frac{1}{\rho_j^{sp}}d^{p_1}Y_j|\mathcal{Q}_{\rho_j}|,
\end{align*} where, in the last inequality, we used $d\geq 1$ and $p_2\geq p$.

\paragraph{\textbf{Estimate of $J_2$}}

Observe that, since $\frac12\rho_j \leq \rho_{j+1}\leq \rho_j$, we have
\begin{align}\label{lbest1}
 \rho_j^{sp} - \rho_{j+1}^{sp} = sp \int_{\rho_j}^{\rho_{j+1}} t^{sp-1}\,dt \geq \frac{1}{c}\rho_j^{sp-1}(\rho_j-\rho_{j+1}).
\end{align} Using \cref{lbest1}, we get
\begin{align*}
        J_2 & \leq c \frac{\rho_{j+1}}{(\rho_j-\rho_{j+1})}\frac{1}{\rho_j^{sp}}\iint_{\mathcal{Q}_{\rho_j}} (\hat{u}-\hat{g}-k_{j+1})_+^2\,d\xi\,dt.
\end{align*}
Estimating similarly to $J_1$ and using $2\leq p_2$, we receive
\begin{align*}
    J_2 \leq c\frac{2^{jp_2}}{(1-\sigma)^p}\frac{1}{\rho_j^{sp}}d^{p_1}Y_j|\mathcal{Q}_{\rho_j}|
\end{align*}

\paragraph{\textbf{Estimate of $J_3$}} Using \cref{eq6.12}, we estimate
\begin{align*}
    J_3 \leq cd^{p_1}\left(\frac{2^j}{1-\sigma}\right)^{\mathcal{B}'(Q+sp)}\rho_j^{-s}\rho^{Q+s}Z_j^{1+\kappa} = cd^{p_1}\left(\frac{2^j}{1-\sigma}\right)^{\mathcal{B}'(Q+sp)}\frac{1}{\rho_j^{sp}}|\mathcal{Q}_{\rho_j}|Z_j^{1+\kappa}.
\end{align*}

We are ready to prove iterative inequalities for $Y_j$ and $Z_j$.

\paragraph{\textbf{Iterative Estimate for $Y_j$}} In the next sequence of estimates, we will successively use H\"older's inequality, \cref{T4.10} with $q=2$, $r=p(1+2s/Q)$ and $m=p(1+2s/Q)$, Young's inequality, \cref{defp1p2B}, \cref{eq6.8}, \cref{eq6.12} and \cref{lbesty}.
\begin{equation}
    \begin{aligned}
        Y_{j+1}^{\frac{p\zeta}{p_2}} & = d^{-\frac{p_1p\zeta}{p_2}}\left( \fiint_{\mathcal{Q}_{\rho_{j+1}}} (\hat{u}-\hat{g}-k_{j+1})_+^{p_2}\,d\xi\,dt \right)^{\frac{p\zeta}{p_2}}\\
        & \leq cd^{-\frac{p_1p\zeta}{p_2}}\Pi_j^{\frac{p\zeta}{p_2}-\frac{\zeta}{1+\frac{2s}{Q}}}\left( \fiint_{\mathcal{Q}_{\rho_{j+1}}} (\hat{u}-\hat{g}-k_{j+1})_+^{p\left(1+\frac{2s}{Q}\right)}\,d\xi\,dt \right)^{\frac{Q}{Q+sp}}\\
        &\leq cd^{-\frac{p_1p\zeta}{p_2}}\Pi_j^{\frac{p\zeta}{p_2}-\frac{\zeta}{1+\frac{2s}{Q}}}|\mathcal{Q}_{\rho_j}|^{-\frac{Q}{Q+sp}}[(\hat{u}-\hat{g}-k_{j+1})_+]^p_{V^{p,2}_s(\mathcal{Q}_{\rho_{j+1}})}\\
        &\qquad\qquad + cd^{-\frac{p_1p\zeta}{p_2}}\Pi_j^{\frac{p\zeta}{p_2}-\frac{\zeta}{1+\frac{2s}{Q}}}|\mathcal{Q}_{\rho_j}|^{-\frac{Q}{Q+sp}}\rho_{j+1}^{-sp}\|(\hat{u}-\hat{g}-k_{j+1})_+\|_{L^p(\mathcal{Q}_{\rho_{j+1}})}^p.
    \end{aligned}
    \end{equation}
    Now using \cref{eq6.13} and
    \begin{align*}
        0\leq\frac{p\zeta}{p_2}-\frac{\zeta}{1+\frac{2s}{Q}}\leq 1,\,\rho_{j+1}^{-sp}\leq \frac{\rho_{j+1}^{p(1-s}}{(\rho_j-\rho_{j+1})^p},\,Y_{j+1}\leq 2^{Q+1}Y_j,
    \end{align*} as well as
    \begin{align*}
        \|(\hat{u}-\hat{g}-k_{j+1})_+\|^p_{L^p(\mathcal{Q}_{\rho_{j+1}})}\leq\|(\hat{u}-\hat{g}-k_j)_+\|^p_{L^p(\mathcal{Q}_{\rho_j})},
    \end{align*} we get
    \begin{equation}\label{lbest3}
        \begin{aligned}
            Y_{j+1} &= Y_{j+1}^{1-\frac{p\zeta}{p_2}}Y_{j+1}^{\frac{p\zeta}{p_2}}\\
            &\leq c2^{jp_2}d^{-p_1}Y_j^{\frac{sp}{Q+sp}}|\mathcal{Q}_{\rho_j}|^{-\frac{Q}{Q+sp}}[(\hat{u}-\hat{g}-k_{j+1})_+]^p_{V^{p,2}_s(\mathcal{Q}_{\rho_{j+1}})}\\
        &\qquad\qquad + c2^{jp_2}d^{-p_1}Y_j^{\frac{sp}{Q+sp}}|\mathcal{Q}_{\rho_j}|^{-\frac{Q}{Q+sp}}\frac{\rho_{j+1}^{p(1-s)}}{(\rho_j-\rho_{j+1})^p}\|(\hat{u}-\hat{g}-k_j)_+\|_{L^p(\mathcal{Q}_{\rho_j})}^p.
        \end{aligned}
    \end{equation}

    Combining \cref{lbest2}, \cref{lbest3} and the estimates for $J_1, J_2,$ and $J_3$, we get
    \begin{align}\label{lbest4}
        Y_{j+1} \leq c\frac{2^{2jp_2}}{(1-\sigma)^p}Y_j^{1+\frac{sp}{Q+sp}} + c\frac{2^{j(p_2+\mathcal{B}'(Q+sp))}}{(1-\sigma)^{\mathcal{B}'(Q+sp)}}Y_j^{\frac{sp}{Q+sp}}Z_j^{1+\kappa}.
    \end{align}

\paragraph{\textbf{Iterative Estimate for $Z_j$}} From \cref{eq6.2}, \cref{T4.10} with $q=2$, $r=\hat{q}$ and $m=p\hat{q}$, and $d^{p_1-\mathcal{B}}\leq 1$, we get
\begin{equation}\label{lbest4.5}
    \begin{aligned}
        Z_{j+1} & \leq (\rho_{j+1}^{Q+s})^{-\frac{1}{1+\kappa}} \left(\int_{-\rho_{j+1}^{sp}}^0\left(\int_{B_{\rho_{j+1}}}\left(\frac{(\hat{u}-\hat{g}-k_{j+1})_+}{k_{j+2}-k_{j+1}}\right)^{\hat{q}}\,d\xi\right)^{p}\,dt\right)^{\frac{\mathcal{B}}{p\hat{q}}}\\
        & = |\mathcal{Q}_{\rho_j}|^{-\frac{Q}{Q+sp}}(k_{j+2}-k_{j+1})^{-\mathcal{B}} \|(\hat{u}-\hat{g}-k_{j+1})_+\|^{\mathcal{B}}_{L^{\hat{q},p\hat{q}}(\mathcal{Q}_{\rho_{j+1}})}\\
        &\leq c\frac{2^{j\mathcal{B}}}{d^{p_1}}|\mathcal{Q}_{\rho_{j}}|^{-\frac{Q}{Q+sp}} \left([(\hat{u}-\hat{g}-k_{j+1})_+]_{V_s^{p,2}(\mathcal{Q}_{\rho_{j+1}})}+\rho_{j+1}^{-sp}\|(\hat{u}-\hat{g}-k_{j+1})_+\|_{L^p(\mathcal{Q}_{\rho_{j+1}})}^p\right)\\
        &\leq c\frac{2^{j\mathcal{B}}}{d^{\mathcal{B}}}|\mathcal{Q}_{\rho_{j}}|^{-\frac{Q}{Q+sp}} \left([(\hat{u}-\hat{g}-k_{j+1})_+]_{V_s^{p,2}(\mathcal{Q}_{\rho_{j+1}})}+\rho_{j+1}^{-sp}\|(\hat{u}-\hat{g}-k_{j+1})_+\|_{L^p(\mathcal{Q}_{\rho_{j+1}})}^p\right).
        \end{aligned}
\end{equation}
        Once again, arguing as for \cref{lbest4}, we obtain
        \begin{align}\label{lbest5}
            Z_{j+1} \leq c\frac{2^{j(\mathcal{B}+p_2)}}{(1-\sigma)^p}Y_j + c\frac{2^{j(\mathcal{B}+\mathcal{B}'(Q+sp))}}{(1-\sigma)^{\mathcal{B}'(Q+sp)}} Z_j^{1+\kappa}.
        \end{align}
        We further simplify \cref{lbest4} and \cref{lbest5} as
        \begin{align}\label{lbest6}
            Y_{j+1}&\leq \frac{c}{(1-\sigma)^L}2^{Lj}\left(Y_j^{1+\frac{sp}{Q+sp}}+ Z^{1+\kappa}Y_j^{\frac{sp}{Q+sp}}\right),\\
            Z_{j+1}&\leq \frac{c}{(1-\sigma)^L}2^{Lj}\left(Y_j+Z^{1+\kappa}\right)
        \end{align} for some constant $c=c(\verb|data|)$ and $L=\mathcal{B} + 2p_2 + \mathcal{B}'(Q+sp)$.

To use the iteration lemma (\cref{iterlemma}), we need estimates for $Y_0$ and $Z_0$.

\paragraph{\textbf{Estimate for $Y_0$}} From \cref{eq6.12}, we have
\begin{align}\label{lbest7}
    Y_0 \leq \frac{c}{d^{p_1}}\fiint_{\mathcal{Q}_1} (\hat{u}-\hat{g})_+^{p_2}\,d\xi\,dt.
\end{align}

\paragraph{\textbf{Estimate for $Z_0$}}
Now, like \cref{lbest4.5}, we have
\begin{equation}\label{lbest8}
    \begin{aligned}
        Z_0 & \leq c\left(\int_{-\rho_0^{sp}}^0\left(\int_{B_{\rho_0}}\left(\frac{\hat{u}-\hat{g}-\hat{k}}{d-\hat{k}}\right)^{\hat{q}}\,d\xi\right)^{p}\,dt\right)^{\frac{\mathcal{B}}{p\hat{q}}}\\
        &\leq c\left(\frac{1}{d-\hat{k}}\right)^{\mathcal{B}}\left([(\hat{u}-\hat{g}-\hat{k})_+]^p_{V^{p,2}_s(\mathcal{Q}_{\rho_0})}+\|(\hat{u}-\hat{g}-\hat{k})_+\|_{L^p(\mathcal{Q}_{\rho_0})}^p\right)
    \end{aligned}
\end{equation}

Similarly to the estimate for $J_1$, we have
\begin{align}
    \label{lbest9}
    \frac{1}{(1-\rho_0)^p}\iint_{\mathcal{Q}_1}(\hat{u}-\hat{g}-\hat{k})_+^p\,d\xi\,dt \leq \frac{c}{(1-\sigma)^p} \iint_{\mathcal{Q}_1}\frac{(\hat{u}-\hat{g})_+^{p_2}}{\hat{k}^{p_2-p}}\,d\xi\,dt \leq \frac{c}{(1-\sigma)^{p}}\fiint_{\mathcal{Q}_1} (\hat{u}-\hat{g})_+^{p_2}\,d\xi\,dt,
    \end{align} where we have used $\hat{k}\geq 1$ and $p_2\geq p$.

Similarly to the estimate for $J_2$, we have
\begin{align}
    \label{lbest10}
    \frac{1}{1-\rho_0^{sp}}\iint_{\mathcal{Q}_1}(\hat{u}-\hat{g}-\hat{k})_+^1\,d\xi\,dt \leq \frac{c}{(1-\sigma)^p} \fiint (\hat{u}-\hat{g})_+^{p_2}\,d\xi\,dt
\end{align}

Similarly to the estimate for $J_3$, we have by $Z_j\leq 1$ that
\begin{align}
    \label{lbest11}
    \hat{k}^{p_1}\left(\frac{\rho_0}{1-\rho_0}\right)^{\mathcal{B}'(Q+sp)}\left(\int_{-1}^0|\hat{\mathcal{A}}_+^g(t;0,1,\hat{k})|^p\,dt\right)^{\frac1p} \leq \hat{k}^{p_1}\left(\frac{1}{1-\sigma}\right)^{\mathcal{B}'(Q+sp)}.
\end{align}

Rewriting \cref{lbest2} with $k_{j+1}=\hat{k}$, $\rho_{j+1}=\rho_0$ and $\rho_j=1$, we receive
\begin{equation}\label{lbest12}
    \begin{aligned}
        \relax[(\hat{u}-\hat{g}-\hat{k})_+]^p_{V^{p,2}_s(\mathcal{Q}_{\rho_0})} &\leq c \frac{\rho_0^{p(1-s)}}{(1-\rho_0)^p}\iint_{\mathcal{Q}_1} (\hat{u}-\hat{g}-\hat{k})_+^p\,d\xi\,dt + \frac{c}{1-\rho_0^{sp}}\iint_{\mathcal{Q}_1} (\hat{u}-\hat{g}-\hat{k})_+^2\,d\xi\,dt\\
        &\qquad + c\hat{k}^{p_1}\left(\frac{\rho_0}{1-\rho_0}\right)^{\mathcal{B}'(Q+sp)}\left(\int_{-1}^0|\hat{\mathcal{A}}_+^g(t;0,1,\hat{k})|^p\,dt\right)^{\frac1p}.
    \end{aligned}
\end{equation}

Combining the estimates in \cref{lbest8}, \cref{lbest9}, \cref{lbest10}, \cref{lbest11}, and \cref{lbest12}, we obtain the bound
\begin{align}
    Z_0 \leq \frac{c}{(1-\sigma)^L} \frac{1}{(d-\hat{k})^{\mathcal{B}}} \left(\fiint_{\mathcal{Q}_1}(\hat{u}-\hat{g})_+^{p_2}\,d\xi\,dt + \hat{k}^{p_1}\right) \leq \frac{c}{(1-\sigma)^L} \frac{1}{(d-\hat{k})^{p_1}} \left(\fiint_{\mathcal{Q}_1}(\hat{u}-\hat{g})_+^{p_2}\,d\xi\,dt + \hat{k}^{p_1}\right)
\end{align} where we use $(d-\hat{k})^{p_1-\mathcal{B}}\leq 1$ from \cref{eq6.4}.

Now, taking $d$ sufficiently large such that
\begin{align}
    \label{lbest13}
    d = \left(\frac{4c}{(1-\sigma)^{2L}}\right)^{\frac{1+\kappa}{p_1\Upsilon}} 2^{L\frac{1+\kappa}{p_1\Upsilon^2}+1} \left[\hat{k} + \left(\fiint_{\mathcal{Q}_1}(\hat{u}-\hat{g})_+^{p_2}\,d\xi\,dt\right)^{\frac{1}{p_1}}\right]
\end{align} where 
\begin{align}
    \label{lbest14}
    \Upsilon := \min\{\kappa, \frac{sp}{Q+sp}\},
\end{align} we get
\begin{align}
    Y_0+Z_0^{1+\kappa} \leq \left( \frac{2c}{(1-\sigma)^L} \right)^{-\frac{1+\kappa}{\Upsilon}} 2^{-L\frac{1+\kappa}{\Upsilon^2}}.
\end{align}

Applying \cref{T4.10} and scaling back, we get, for $c_1=2L(\frac{1+\kappa}{p_1\Upsilon}+1)$
\begin{equation}
\begin{aligned}
    \sup_{\mathcal{Q}_{\sigma\rho}(z_0)} (u(\xi,t)-g(t_0))_+ &\leq \frac{c}{(1-\sigma)^{c_1}}\Bigg[\left(\fiint_{\mathcal{Q}_{\rho}(z_0)} (u-g)_+^{p_2}\,d\xi\,dt\right)^{\frac{1}{p_1}} \\
    &\qquad+ \Big[ \fint_{t_0-\rho^{sp}}^{t_0} \left( \fint_{B_\rho(\xi_0)} u_+^{p-1}(\eta,t) \,d\eta\right)^{\frac{p}{p-1}}\,dt \Big]^\frac{\mathcal{B}'(p-1)}{pp_1} + 1 \Bigg].
\end{aligned}
\end{equation}

This implies that
\begin{equation}
\begin{aligned}
    \sup_{\mathcal{Q}_{\sigma\rho}(z_0)} u_+(\xi,t) &\leq \frac{c}{(1-\sigma)^{c_1}}\Bigg[\left(\fiint_{\mathcal{Q}_{\rho}(z_0)} u_+^{p_2}\,d\xi\,dt\right)^{\frac{1}{p_1}} \\
    &\qquad+ \Big[ \fint_{t_0-\rho^{sp}}^{t_0} \left( \fint_{B_\rho(\xi_0)} u_+^{p-1}(\eta,t) \,d\eta\right)^{\frac{p}{p-1}}\,dt \Big]^\frac{\mathcal{B}'(p-1)}{pp_1} + 1 \Bigg]\\
    &\qquad\qquad+ \mathfrak{C}\rho^{-sp}\int_{t_0-\rho^{sp}}^{t_0}\textup{Tail}^{p-1}(u_+(t);\rho,\xi_0)\,dt.
\end{aligned}
\end{equation}

 \end{proof}

\appendix
\section{Some Interpolation Inequalities}\label{sec:interineq} In this appendix, we prove a few interpolation inequalities. One of them is in the critical case $sp=Q$ and the other one is for the supercritical case $sp>Q$. Of great interest is the BMO interpolation inequality (\cref{bmointerpol}) which is new in this context. 

We will make use of the following embedding.
\begin{theorem}(\cite[Theorem 1.5]{AdiMali2018})\label{morreyheisen}
        Let $s\in (0,1)$ and $sp>Q$. Then there exists a constant $C=C(Q,p,s)$ such that
        \begin{align}
            [u]_{C^{0,\beta}(\hn)}\leq C[u]_{W^{s,p}(\hn)} 
        \end{align} for all $u\in L^{p}(\hn)$ and $\beta=s-\frac{Q}{p}$.
    \end{theorem}

We will first prove the following interpolation result whose proof is modified from \cite[Prop.12.84]{Leoni17}.
    \begin{theorem}\label{interlope1}
        Let $u:\hn\to\mathbb{R}$ be H\"older continuous with exponent $0<\alpha<1$ and such that $u\in L^q(\hn)$ for some $q\in [1,\infty]$. Then $u$ belongs to $L^r(\hn)$ for all $r\in [q,\infty]$ with
        \begin{align*}
            \|u\|_{L^r(\hn)}\lesssim \|u\|_{L^q(\hn)}^{\theta_1} [u]_{C^{0,\alpha}(\hn)}^{1-\theta_1},
        \end{align*} where $\theta_1=\frac{\alpha+Q/r}{\alpha+Q/q}$.
    \end{theorem}
    \begin{proof}
        If $q=r$ or if $q=r=\infty$, the proposed inequality is true trivially. Therefore, let us assume that $q<r<\infty$. We begin by proving that $u$ is bounded. If $[u]_{C^{0,\alpha}}=0$, then $u$ is a constant and, in fact, the zero function since $\|u\|_{L^q(\hn)}<\infty$. Therefore, we further assume that $[u]_{C^{0,\alpha}}>0$. Given $R>0$, $\xi\in\hn$ and $\eta\in B(\xi,R)\setminus\{\xi\}$, we get
        \begin{align*}
            |u(\xi)|\leq u(\eta) + |u(\xi)-u(\eta)|\leq |u(\eta)| + R^\alpha [u]_{C^{0,\alpha}}.
        \end{align*} Averaging both sides over $\eta$ in $B(\xi,R)$ and using H\"older's inequality for $q>1$, we get
        \begin{align*}
            |u(\xi)|&\leq \frac{1}{\omega_{\hn}R^Q}\int_{B(\xi,R)}u(\eta)\,d\eta + R^\alpha [u]_{C^{0,\alpha}}\\
            & \leq (\omega_{\hn}R^Q)^{-\frac1q}\|u\|_{L^q} + R^\alpha [u]_{C^{0,\alpha}}.
        \end{align*}
        Minimizing the function $g(t)=\omega_{\hn}^{-\frac1q}t^{-\frac{Q}{q}}\|u\|_{L^q} + t^\alpha [u]_{C^{0,\alpha}}$ over all $t>0$ gives us the identity
        \begin{align*}
            \|u\|_{L^\infty(\hn)} \lesssim \|u\|_{L^q(\hn)}^{\theta} [u]_{C^{0,\alpha}}^{1-\theta},
        \end{align*} where $\theta=\frac{\alpha}{\alpha+Q/q}$.
        we substitute this in the simple inequality
        \begin{align*}
            \|u\|_{L^r(\hn)}\leq \|u\|_{L^\infty(\hn)}^{1-\frac{q}r}\|u\|_{L^q(\hn)}^{\frac{q}r}
        \end{align*} to get the proposed inequality.
    \end{proof}

    We say that a locally integrable function $f:\hn\to\mathbb{R}$ is in $\text{BMO}(\hn)$ if there exists a constant $c$ such that 
    \begin{align}
        \fint_B |f(\xi)-f_B|\,d\xi\leq c
    \end{align} for all balls $B\subset \hn$. Here and elsewhere we use the notation $f_B: = \fint_B f(\xi)\,d\xi$. The space of all such functions is equipped with the seminorm
    \begin{align}
        [f]_{\text{BMO}}=\sup_{B\subset\hn}\fint_B|f(\xi)-f_B|\,d\xi.
    \end{align}

    We shall require the following embedding whose proof is modified from that in \cite[Theorem 7.11]{Leoni17}.
    \begin{theorem}\label{bmoembedsobol}
        Let $p\in (Q,\infty)$. Then $|u|_{\text{BMO}(\hn)}\lesssim [u]_{W^{Q/p,p}(\hn)}$ for $u\in W^{Q/p,p}(\hn)$.
    \end{theorem}

    \begin{proof}
        By Sobolev-Poincar\'e inequality, we have for any ball $B\subset\hn$ that
        \begin{align*}
            \int_B|u(\xi)-u_B|^p\,d\xi \leq |\text{diam}(B)|^{Q}\intgaglifrac{B},
        \end{align*} where we used $sp=Q$. This is equivalent to 
        \begin{align*}
            \fint_B|u(\xi)-u_B|^p\,d\xi \leq \intgaglifrac{B}. 
        \end{align*}Since $B$ is arbitrary, we may take supremum over $B$ on the left hand side and conclude.
    \end{proof}

    The following theorem, also known as the John-Nirenberg lemma, is taken from \cite{Aaltoetal11} where it is stated for doubling metric measure spaces.

    \begin{theorem}\label{JNL}(\cite[Theorem 5.2]{Aaltoetal11})
        Let $f\in\text{BMO}(\hn)$. Then
        \begin{align*}
            |\{\xi\in B:|f-f_B|>\lambda\}| \leq c_1|B|e^{-c_2\lambda/[f]_{\text{BMO}}}
        \end{align*} for all balls $B\subset \hn$ and $\lambda>0$ with the constants $c_1$ and $c_2$ independent of $f$ and $\lambda$.
    \end{theorem}

    \subsection{A BMO interpolation inequality}

    In this section, we prove the following interpolation inequality for $\text{BMO}(\hn)$ spaces whose proof we have adapted from \cite{AzBed2015}. This inequality may be of independent interest.

    \begin{theorem}\label{bmointerpol}
        Let $q\in (1,\infty)$. Let $f\in L^q(\hn)\cap \text{BMO}(\hn)$. Then for $r\in (q,\infty)$ it holds that
        \begin{align}
            \|f\|_{L^r(\hn)}\lesssim_{N,q,r} \|f\|_{L^q(\hn)}^{\frac{q}{r}}[f]_{\text{BMO}(\hn)}^{1-\frac{q}{r}}.
        \end{align}
    \end{theorem}

    The proof of \cref{bmointerpol} requires Calder\'on-Zygmund decomposition lemma in addition to the John-Nirenberg lemma. We state the Calder\'on-Zygmund decomposition theorem for $L^p$ spaces below. Its proof is the $L^q$ modification of \cite{Cw71} as in \cite{Feff1969}. We omit the proof since it is well-known.

    \begin{lemma}\label{CZD}(Calder\'on-Zygmund decomposition) Let $f\in L^q(\hn)$ and let $\lambda>0$ be given. Then there exists a countable family $\{B_i\}_i$ of disjoint balls such that
    \begin{enumerate}
        \item[(i)] $\{\frac14 B_i\}$ are pairwise disjoint.
        \item[(ii)] The $B_i$'s have small total volume, i.e., 
        \begin{align*}
            \sum_i|B_i| \leq \frac{C}{\lambda^q}\|f\|^q_{L^q(\hn)}
        \end{align*}
        \item[(iii)] $|f(\xi)|\leq C \lambda$ a.e. $\xi\in \hn \setminus \bigcup_i B_i$,
        \item[(iv)] For any $i$, $\lambda^q\lesssim\fint_{B_i} |f(\xi)|^q\,d\xi\lesssim_{Q}\lambda^q$.
    \end{enumerate} The balls satisfying the above conditions are called Calder\'on-Zygmund balls at level $\lambda$. Moreover, if $\lambda_0\leq \lambda_1\leq\cdots\leq \lambda_K$, then the Calder\'on-Zygmund balls corresponding to different levels $\lambda_n$ may be chosen in such a way that each $\frac14 B_i(\lambda_{n+1})$ is contained in some $B_j(\lambda_n)$.        
    \end{lemma}

    \begin{proof}
        (Proof of \cref{bmointerpol}) Define $E_\alpha = \{\mathcal{M}f>\alpha\}$, $f_1=f\mathbbm{1}_{E_\alpha^c}$ and $f_2=f-f_1$. As in the proof of Calder\'on-Zygmund decomposition $E_\alpha\subset \bigcup_i B_i$. By \cref{CZD}, $f_1\lesssim\alpha$ so that
        \begin{align}\label{boundonf1}
            \|f_1\|^r_{L^r(\hn)} \lesssim \alpha^{r-q}\|f\|_{L^q(\hn)}^q.
        \end{align}
        We note here that as a consequence of \cref{JNL}, for any ball $B_i$
        \begin{align*}
            \fint_{B_i}|f-f_{B_i}|^r \lesssim [f]_{\text{BMO}(\hn)}^r.
        \end{align*}
        To see this, we observe that for \textit{any} ball $B$
        \begin{align*}
            \fint |f-f_B|^r = \frac{\int_0^\infty r\, \alpha^{r-1} \left|\xi\in B:|f-f_B|>\alpha\right|\,d\alpha}{|B|} \lesssim\int_0^\infty \alpha^{r-1} e^{-c_3\alpha/[f]_{\text{BMO}}}\,d\alpha\lesssim [f]_{\text{BMO}(\hn)}^r.
        \end{align*}
        Therefore
        \begin{align*}
            \int_{\hn} |f_2|^r \leq \sum_i \int_{B_i} |f|^r &\leq 2^r \sum_i\left(\int_{B_i}|f-f_{B_i}|^r + \int_{B_i} |f_{B_i}|^r  \right)\\
            &\lesssim \sum_i \left( [f]_{\text{BMO}(\hn)} + \alpha^r |B_i| \right)\\
            &\leq  \left( [f]_{\text{BMO}(\hn)} + \alpha^r \right) \left|\xi\in\hn: \mathcal{M}f(\xi)>\alpha\right|\\
            &\lesssim  \left( [f]_{\text{BMO}(\hn)} + \alpha^r \right) \frac{\|\mathcal{M}f\|^q_{L^q(\hn)}}{\alpha^q}\\
            &\lesssim  \left( [f]_{\text{BMO}(\hn)} + \alpha^r \right) \frac{\|f\|^q_{L^q(\hn)}}{\alpha^q},
        \end{align*} where in the first inequality, we used the fact that $f_{B_i} = \fint_{B_i} f\leq \left(\fint_{B_i} |f|^q\right)^{\frac1q}\lesssim \alpha$; in the second inequality, we used that $E_\alpha$ is covered by the collection of $B_i$'s; in the third inequality, we use Chebyshev's inequality; and in the last inequality, we use the boundedness of maximal function in $L^q$ for $q>1$. Now choosing $\alpha = \|f\|_{L^q(\hn)}$, we get
        \begin{align*}
            \int_\hn |f_2|^r \lesssim [f]_{\text{BMO}(\hn)}^r + \|f\|_{L^q(\hn)}^r.
        \end{align*}
                Also, from \cref{boundonf1}, we have 
            $\|f_1\|^r_{L^r(\hn)} \lesssim \|f\|_{L^q(\hn)}^r$. Combining the previous two inequalities, we obtain
            \begin{align}\label{bmoineq1}
            \|f\|_{L^r(\hn)} \lesssim [f]_{\text{BMO}(\hn)} + \|f\|_{L^q(\hn)}.
            \end{align} 
            Now we notice that in the class $BMO(\hn)\cap L^q(\hn)$, for $\xi=(\xi,y,s)$ and the function $f_\delta(\xi)=f(\delta x,\delta y,\delta^2 s)$, we have
            \begin{align*}
                \|f_\delta\|_{L^r(\hn)}&=\delta^{-\frac{Q}r}\|f\|_{L^r(\hn)},\\
                \|f_\delta\|_{L^q(\hn)}&=\delta^{-\frac{Q}q}\|f\|_{L^q(\hn)},\,\text{and}\\
                [f_\delta]_{\text{BMO}(\hn)} &= [f]_{\text{BMO}(\hn)},
            \end{align*} and the inequality \cref{bmoineq1} holds for $f_\delta$. Letting $\eta=\delta^{-\frac1Q}$, we have for all $\eta>0$
            \begin{align*}
                \eta^{\frac1r}\|f\|_{L^r(\hn)} \lesssim [f]_{BMO(\hn)} + \eta^{\frac1q}\|f\|_{L^q(\hn)}.
            \end{align*} Finally, we obtain the proposed interpolation inequality by taking $$\eta=\left(\frac{[f]_{\text{BMO}(\hn)}}{\|f\|_{L^q(\hn)}}\right)^q.$$
        \end{proof}

        \section{An extension theorem for fractional Sobolev spaces}\label{sec:extsob}

        Here we prove an extension theorem for fractional Sobolev spaces in Heisenberg group. The local version of this extension theorem appears in the thesis of \cite{Nhieu1996}. However, it is found that working in the general setting of metric measure spaces is beneficial. Although it is not immediately clear whether \emph{suitably defined} Lipschitz domains are extension domains. However, it is reasonably clear that a certain measure-density condition on the domain is essential to extension theorem results starting with \cite{Koskela2008}, \cite{Zhou2015}, \cite{DachunYang2022}. The result that we want can be repurposed from these articles, and we supply a proof for completeness.
	\begin{definition}
		Let $\Omega$ be a connected open subset of $\hn$. It is said to be a  regular domain, i.e., it satisfies the \emph{measure density condition}: if there exists a constant $C_1>0$ such that
		\begin{align}\label{measuredense}
			|B(\xi,r)\cap\Omega|\geq C_1r^Q\mbox{ for all }\xi\in\Omega\mbox{ and all }r\in(0,1].
		\end{align}
	\end{definition}
	\begin{theorem}\label{extendthm}
		Let $\Omega$ be a connected open subset of $\hn$. If $\Omega$ is a regular domain then $\Omega$ is a $W^{s,p}$-extension domain for all $s\in (0,1)$ and for all $p\in(1,\infty)$, that is, every function $u\in W^{s,p}(\Omega)$ can be extended to a function $\tilde{u}\in W^{s,p}(\hn)$ so that $\|\tilde{u}\|_{W^{s,p}(\hn)}\leq C\|u\|_{W^{s,p}(\Omega)}$.
	\end{theorem}
	
	\begin{remark}
		The proof for $p=1$ requires the use of Poincar\'e inequality while the rest of the proof is similar as in \cite{Koskela2008}. We do not prove this result here. One can also use median value instead of mean to avoid this complication as in \cite{Zhou2015}.
	\end{remark}
	
	\begin{lemma}
		If an open set $\Omega\subset\hn$ satisfies \cref{measuredense} then $|\overline{\Omega}\setminus\Omega|=0$.
	\end{lemma}
	\begin{proof}
		For $\xi\in \overline{\Omega}\setminus\Omega$ and $r\in (0,1]$, take a sequence $\xi_n\in\Omega$ such that $\xi_n\to \xi$ as $n\to\infty$. Then
		\begin{align*}
			|B(\xi,r)\cap\Omega|=\lim_{n\to\infty}|B(\xi_n,r)\cap\Omega|\geq C_1r^Q
		\end{align*} so that
		\begin{align*}
			\limsup_{r\to 0}\frac{|B(\xi,r)\cap(\overline{\Omega}\setminus\Omega)|}{|B(\xi,r)|}\leq 1-C_1C<1.
		\end{align*} Therefore $\xi$ is not a point of density for $\overline{\Omega}\setminus\Omega$. Since $\xi$ was arbitrary, it means that $|\overline{\Omega}\setminus\Omega|=0$. Lebesgue differentiation theorem  for measure metric spaces with doubling measure, of which Lebesgue density theorem is a consequence, can be found in \cite{Shanmugalingam2015}.
	\end{proof}
	We record two results that we will require in the course of the proof. The proof of Whitney Decomposition Theorem is to be found in \cite{Macias1979}.
	\begin{theorem}(Whitney Decomposition Theorem)\label{whitney}
		Suppose that $(\xi,d,\mu)$ is a metric measure space, where $\mu$ is a doubling measure. Let $U$ be an open set in $X$. Then there exists $M\in\mathbb{N}$ and a sequence of points $\xi_i\in X, i\in\mathbb{N}$ and radii $r_i>0$ such that
		\begin{enumerate}
			\item the balls $B(\xi_i,r_i/4)$ are pairwise disjoint.
			\item $U=\bigcup_{i\in\mathbb{N}}B(\xi_i,r_i).$
			\item $B(\xi_i,5r_i)\subset U$.
			\item If $x\in B(\xi,5r_i)$ then $5r_i<\text{dist}(\xi,X\setminus U)<15r_i$.
			\item there is $\xi_i^*\in X\setminus U$ such that $d(\xi_i,\xi_i^*)<15r_i$.
			\item $\sum_{i=1}^\infty \mathbbm{1}_{B(\xi_i,5r_i)}(\xi)\leq M$ for all $\xi\in U$. 
		\end{enumerate} 
	\end{theorem}
	\begin{theorem} (Partition of Unity)\label{partuni}
		Let $\{B_i\}_{i\in\mathbb{N}}$ be the whitney covering of $U$. Then there exists a family of  functions $\{\varphi_i\}_{i\in\mathbb{N}}$ satisfying
		\begin{enumerate}
			\item $\text{supp}\, \varphi_i \subset 2B_i$.
			\item $\varphi_i(\xi)\geq \tfrac1M$ for all $\xi\in B_i$.
			\item there is a constant $K$ such that each $\varphi_i$ is $\tfrac{K}{r_i}$-Lipschitz.
			\item $\sum_{i=1}^\infty \varphi_i(\xi)=\mathbbm{1}_U(\xi)$.
		\end{enumerate}
	\end{theorem}
	
	\begin{proof}(Proof of \cref{extendthm}) Let $U=\hn\setminus\overline{\Omega}$. Consider a Whitney covering subordinate to $U$ as in \cref{whitney}. Let $J$ be the collection of all $i\in\mathbb{N}$ such that $r_i\leq 1$. Define $V=\{\xi\in\hn: \text{dist}(\xi,\Omega)\leq 3\}$. For each $u\in W^{s,p}(\Omega)$, we define
		\begin{align}
			\tilde{E}(u(\xi))=\begin{cases}
				u(\xi)&\mbox{ for }\xi\in\Omega\\
				0&\mbox{ for }\xi\in\overline{\Omega}\setminus\Omega\\
				\sum_{i\in J}\varphi_i(\xi)\, m_{B(\xi_i^*,r_i)\cap\Omega}(u)&\mbox{ for }\xi\in U.
			\end{cases}
		\end{align}
		Here $\xi_i^*$ is as in \cref{whitney} and $m$ denotes the mean or average over a set defined by 
		\begin{align*}
			m_{G}(u):=\frac1{|G|}\int_G\,u(\xi)\,d\xi.
		\end{align*}
		Let $\Psi:\hn\to [0,1]$ be a Lipschitz continuous function such that $\Psi\equiv 1$ on $\Omega$, $\Psi\equiv 0$ on $\hn\setminus V$. Set 
		\begin{align*}
			Eu=\Psi\tilde{E}u. 
		\end{align*} The content of this theorem is to prove that $\|Eu\|_{W^{s,p}(\hn)}\lesssim\|u\|_{W^{s,p}(\Omega)}$. For this purpose, it is enough to prove that $\tilde{E}u\in W^{s,p}(V)$ and 
		\begin{align}\label{secineq}
			\|\tilde{E}u\|_{W^{s,p}(V)}\lesssim\|u\|_{W^{s,p}(\Omega)}.
		\end{align} To see this, we observe that since $\Psi\in[0,1]$, we have 
		\begin{align}
			\|Eu\|_{L^{p}(\hn)}\leq \|\tilde{E}u\|_{L^{p}(V)}\lesssim\|u\|_{L^{p}(\Omega)}.
		\end{align} Also note that
		\begin{align*}
				\intgaglifrac[Eu]{\hn} & = \intgaglifrac[\Psi\tilde{E}]{V}\\
				&\lesssim \underbrace{\int_V\left(\int_{V\setminus B(\xi,1)} \frac{|\tilde{E}u|^p+|\tilde{E}u|^p}{\dhn{\xi}{\eta}^{Q+sp}} \,d\eta\right)\,d\xi)}_{I_1}\\
				&\qquad +\underbrace{\int_V\left(\int_{V\cap B(\xi,1)}\gaglifrac[\Psi\tilde{E}u]\,d\eta\right)\,d\xi}_{I_2}
		\end{align*}
		For $I_1$, we use Fubini's theorem as well as \cref{lem:annul} to get
		\begin{align*}
			I_1 &\leq \int_V|\tilde{E}u(\xi)|^p\left( \int_{V\setminus B(\xi,1)} \frac1{\dhn{\eta}{\xi}^{Q+sp}} \,d\eta \right)\,d\xi\\
			&\qquad \int_V|\tilde{E}u(\xi)|^p\left( \int_{V\setminus B(\eta,1)} \frac1{\dhn{\eta}{\xi}^{Q+sp}} \,d\xi \right)\,d\eta\\
			&\lesssim\|\tilde{E}u\|_{L^p(V)}.
		\end{align*}
		On the other hand, making use of \cite[Proposition 5.4.4]{Bonfig2007} and Lipschitz property of $\Psi$, we have
		\begin{align*}
			I_2 &\leq \int_V\left(\int_{V\cap B(\xi,1)}\gaglifrac[\tilde{E}u]\,d\eta\right)\,d\xi\\
			&\qquad  \int_V\left(\int_{V\cap B(\xi,1)}\gaglifrac[\Psi]|\tilde{E}u(\xi)|^p\,d\eta\right)\,d\xi\\
			&\lesssim \|\tilde{E}u\|^p_{W^{s,p}(V)} + \int_V|\tilde{E}u(\xi)|^p\left(\int_{B(\xi,1)}\frac{1}{\dhn{\xi}{\eta}^{Q+sp}}\,d\eta\right)\,d\xi\\
			&=\|\tilde{E}u\|^p_{W^{s,p}(V)} +\int_V|\tilde{E}u(\xi)|^p\left(Q\omega_{\hn}\int_0^1\frac{r^{Q-1}}{r^{Q+sp-p}}\,dr \right)\,d\xi\\
			&\lesssim \|\tilde{E}u\|_{W^{s,p}(V)}.
		\end{align*}
		with \cref{secineq}, this proves that $Eu\in W^{s,p}(\hn)$ and $\|Eu\|_{W^{s,p}(\hn)}\lesssim\|u\|_{W^{s,p}(\Omega)}$. Thus, it remains to prove \cref{secineq}. This will be achieved in multiple steps.
		\paragraph{\textbf{Step 1.}} we prove that $\|\tilde{E}u\|_{L^p(V)}\leq \|u\|_{L^p(\Omega)}$. For $\xi\in V\setminus \overline{\Omega}$ denote by $I_{\xi}$ the collection of those $i\in\mathbb{N}$ for which $\xi\in B(x_i,2r_i)$. The number of elements in $I_{\xi}$ is not more than $M$ by \cref{whitney}. 
		Let $z\in\Omega$, then for any $i\in \mathbb{N}$, $5B_i\subset X\setminus\overline{\Omega}$, therefore $|\xi_i^{-1}\circ z|>5r_i$. On the other hand, $\xi\in 2B_i$. Therefore, $|\xi^{-1}\circ z|\geq |\xi^{-1}\circ z|-|\xi^{-1}\circ\xi_i|>5r_i-2r_i=3r_i$. Taking supremum over all $z\in\Omega$, we get 
		\begin{align}
			\text{dist}(\xi,\Omega)>3r_i\mbox{ for }\xi\in 2B_i.
		\end{align}
		Also, for $\xi\in 2B_i$ and $z\in B(\xi_i^*,r_i)$, we have
		\begin{align*}
			|\xi^{-1}\circ z| \leq |{\xi_i^*}^{-1}\circ z|+|\xi_i^{-1}\circ \xi_i^*|+|\xi^{-1}\circ \xi_i|<r_i+15r_i+2r_i=18r_i<6\,\text{dist}(\xi,\Omega),
		\end{align*} therefore
		\begin{align}\label{cond1}
			B(\xi_i^*,r_i)\subset B(\xi,6\, \text{dist}(\xi,\Omega)).
		\end{align}
		Note also by the measure density condition \cref{measuredense}
		\begin{align}\label{cond2}
			|B(\xi_i^*,r_i)\cap\Omega|\sim|B(\xi,6\,\text{dist}(\xi,\Omega))\cap\Omega|\sim|B(\xi,6\,\text{dist}(\xi,\Omega))|.
		\end{align}
		Now, if $i\notin J$ then $r_i\geq 1$ and hence $\text{dist}(\xi,\Omega)>3r_i\geq 3$ for all $\xi\in 2B_i$. Therefore $2B_i\cap V=\emptyset$ and $I_{\xi}\subset J$. Moreover, when $\xi\notin I_\xi$ then $\varphi(\xi)=0$, thus 
        \begin{align}\label{pouS}\sum_{i\in I_{\xi}}\varphi(\xi)=\sum_{i\in J} \varphi(\xi)=\sum_{i\in \mathbb{N}}\varphi(\xi)=1.\end{align}  
    Combining these observations, we have the following chain of inequalities:
		\begin{align*}
			\tilde{E}u(\xi)&\leq \sum_{i\in I_{\xi}}\varphi(\xi)\fint_{B(\xi_i^*,r_i)\cap\Omega}|u(\eta)|\,d\eta\\
			&\leq \sum_{i\in I_{\xi}}\varphi(\xi)\fint_{B(\xi,r_i)}|u(\eta)|\,d\eta\\
   			&\lesssim \fint_{B(\xi,r_i)}|u(\eta)|\,d\eta\lesssim (\mathcal{M}u)(\xi)
		\end{align*} where $\mathcal{M}u$ is the Hardy-Littlewood maximal function of $u$ and although $u$ is defined only on $\Omega$, we extend it outside of $\Omega$ by zero for the last step in the above display. Now, the claim follows readily by the $L^p$-boundedness of the maximal function (see \cite[Theorem 3.5.6]{Shanmugalingam2015} for a proof). 
    \paragraph{\textbf{Step 2.}} Now, we prove $[\tilde{E}u]_{W^{s,p}(V)}\leq \|u\|_{W^{s,p}(\Omega)}$ where \[\displaystyle[f]_{W^{s,p}(V)}:=\intgaglifrac[f]{V}\] is the Gagliardo seminorm. We write
    \begin{align*}
        \intgaglifractwo[\tilde{E}u]{V}{V}&=\underbrace{\intgaglifrac[u]{\Omega}}_{II_1} + 2 \underbrace{\int_{V\setminus\Omega}\int_\Omega \gaglifractwo{\tilde{E}u(\eta)}{u(\xi)}{\xi}{\eta}\,d\xi\,d\eta}_{II_2} \\
        &\qquad + \underbrace{\intgaglifrac[\tilde{E}u]{V\setminus\Omega}}_{II_3}.
    \end{align*}
    \paragraph{\textbf{Substep 1:}} Clearly, $II_1\leq [u]_{W^{s,p}(\Omega)}\leq \|u\|_{W^{s,p}(\Omega)}$.
    \paragraph{\textbf{Substep 2:}} Let $\eta\in V\setminus\Omega$ and $\xi\in\Omega$. By \cref{cond2}, \cref{partuni} and other observations as in Step 1, we get 
    \begin{align*}
        |\tilde{E}u(\eta)-u(\xi)| &= \left|\sum_{i\in I_{\eta}} \phi_i(\eta)(m_{B_i^*\cap\Omega}(u)-u(\xi))\right|\\
        &\leq \sum_{i\in I_{\eta}}\phi_{i}(\eta)\fint_{B_i^*\cap\Omega}|u(\zeta)-u(\xi)|\,d\zeta\\
        &\lesssim \fint_{B_\eta(6\,\text{dist}(\eta,\Omega))\cap\Omega}|u(\zeta)-u(\xi)|\,d\zeta.
    \end{align*}
    For $\xi\in\Omega$ and $\zeta\in B_{\eta}(6\,\text{dist}(\eta,\Omega))$, we have
    \begin{align*}
        \dhn{\zeta}{\xi}\leq \dhn{\zeta}{\eta} + \dhn{\eta}{\xi}\leq 6\,\text{dist}(\xi,\Omega) + \dhn{\eta}{\xi} \lesssim \dhn{\eta}{\xi}.
    \end{align*}
    Therefore
    \begin{align*}
        \left(\gaglifractwo{\tilde{E}u(\eta)}{u(\xi)}{\xi}{\eta}\right)^{\tfrac1p} \lesssim \fint_{B_\eta(6\,\text{dist}(\eta,\Omega))\cap\Omega}\frac{|u(\zeta)-u(\xi)|}{\dhn{\zeta}{\xi}^{Q/p+s}}\,d\zeta \lesssim \mathcal{M}\left( \frac{|u(\cdot)-u(\xi)|}{\dhn{\cdot}{\xi}^{\frac{Q}p+s}} \mathbbm{1}_{\Omega}(\cdot)\right)(\eta).
    \end{align*} Finally, using $L^p$ boundedness of Hardy-Littlewood maximal function, we get
    \begin{align*}
        II_2 & \lesssim \int_{V\setminus\Omega}\int_\Omega \left[ \mathcal{M}\left( \frac{|u(\cdot)-u(\xi)|}{\dhn{\cdot}{\xi}^{\frac{Q}p+s}} \mathbbm{1}_{\Omega}(\cdot)\right)(\eta) \right]^p \,d\xi\,d\eta\\
        & \lesssim \int_\Omega\int_{\hn}\left[ \mathcal{M}\left( \frac{|u(\cdot)-u(\xi)|}{\dhn{\cdot}{\xi}^{\frac{Q}p+s}} \mathbbm{1}_{\Omega}(\cdot)\right)(\eta) \right]^p\,d\eta\,d\xi\\
        &\leq \int_\Omega\int_{\hn} \frac{|u(\eta)-u(\xi)|^p}{\dhn{\eta}{\xi}^{Q+sp}}\mathbbm{1}_{\Omega}(\eta)\,d\eta\,d\xi\\
        &\leq \int_\Omega\int_{\Omega} \frac{|u(\eta)-u(\xi)|^p}{\dhn{\eta}{\xi}^{Q+sp}}\,d\eta\,d\xi = \|u\|^p_{W^{s,p}(\Omega)}.
    \end{align*}
    \paragraph{\textbf{Substep 3:}} we finally estimate $II_3$ by further splfurther splitting $V\setminus\Omega$ into two parts:
    \begin{align*}
        D_1(\xi)\equiv\left\{\eta\in V\setminus\Omega:\dhn{\xi}{\eta}\geq\tfrac12\max\{\text{dist}(\xi,\Omega),\text{dist}(\eta,\Omega)\}\right\},\mbox{ and}
    \end{align*}
    \begin{align*}
        D_2(\xi)\equiv\left\{\eta\in V\setminus\Omega:\dhn{\xi}{\eta}<\tfrac12\max\{\text{dist}(\xi,\Omega),\text{dist}(\eta,\Omega)\}\right\}.
    \end{align*}
    we write
    \begin{align*}
        II_3 = \underbrace{\intgaglifracflip[\tilde{E}u]{D_1(\xi)}{V\setminus\Omega}}_{II_{3,1}} + \underbrace{\intgaglifracflip[\tilde{E}u]{D_2(\xi)}{V\setminus\Omega}}_{II_{3,2}}.
    \end{align*}
    \paragraph{\textbf{Estimate for $II_{3,1}$:}} If $\xi\in V\setminus\Omega$ and $\eta\in D_1(\xi)$, we have
    \begin{align*}
        \tilde{E}u(\xi)-\tilde{E}u(\eta) &= \sum_{i\in I_\xi}\varphi_i(\xi)\, m_{B(\xi_i^*,r_i)\cap\Omega}(u) - \sum_{j\in I_\eta}\varphi_j(\eta)\, m_{B(\xi_j^*,r_j)\cap\Omega}(u)\\
        &=\sum_{i\in I_\xi}\sum_{j\in I_\eta}\varphi_i(\xi)\,\varphi_j(\eta)\,[ m_{B(\xi_i^*,r_i)\cap\Omega}(u) -m_{B(\xi_j^*,r_j)\cap\Omega}(u)],
    \end{align*} on account of $\sum_{i\in I_\xi}\varphi_i(\xi)=\sum_{j\in I_\eta}\,\varphi_j(\eta)=1$. 
    we also have, by \cref{cond1} and \cref{cond2} that
    \begin{align*}
        |m_{B(\xi_i^*,r_i)\cap\Omega}(u) -m_{B(\xi_j^*,r_j)\cap\Omega}(u)|&\lesssim \fint_{B(\xi_i^*,r_i)\cap\Omega}\fint_{B(\xi_j^*,r_j)\cap\Omega}|u(\zeta)-u(\chi)|\,d\zeta\,d\chi\\
        &\lesssim \fint_{B(\eta, 6\,\text{dist}(\xi,\Omega))\cap\Omega}\fint_{B(\xi, 6\,\text{dist}(\eta,\Omega))\cap\Omega}|u(\zeta)-u(\chi)|\,d\zeta\,d\chi.
    \end{align*}
    Now, for $\zeta\in B(\xi, 6\,\text{dist}(\eta,\Omega))$ and $\chi\in B(\eta, 6\,\text{dist}(\xi,\Omega))$, we have
    \begin{align*}
        \dhn{\zeta}{\chi} \leq \dhn{\zeta}{\xi} + \dhn{\xi}{\eta} + \dhn{\eta}{\chi}\leq 6\,\text{dist}(\eta,\Omega) + \dhn{\xi}{\eta} + 6\,\text{dist}(\xi,\Omega) \lesssim \dhn{\xi}{\eta}
    \end{align*} since $\dhn{\xi}{\eta}\geq \frac12\max\{\text{dist}(\xi,\Omega),\text{dist}(\eta,\Omega)\}$.
    Combining the previous three displays, we obtain
    \begin{align*}
        \left(\gaglifractwo{\tilde{E}u(\xi)}{\tilde{E}u(\eta)}{\eta}{\xi}\right)^{\frac1p} & \lesssim \fint_{B(\eta, 6\,\text{dist}(\xi,\Omega))\cap\Omega}\fint_{B(\xi, 6\,\text{dist}(\eta,\Omega))\cap\Omega}\frac{|u(\zeta)-u(\chi)|}{\dhn{\zeta}{\chi}^{Q/p+s}}\,d\zeta\,d\chi\\
        &\lesssim (\mathcal{M}\times\mathcal{M})(F)(\xi,\eta)
    \end{align*} where
    \begin{align*}
        F(\zeta,\chi) = \frac{|u(\zeta)-u(\chi)|}{\dhn{\zeta}{\chi}^{Q/p+s}}\mathbbm{1}_{\Omega}(\zeta)\mathbbm{1}_\Omega(\chi)
    \end{align*} and $\mathcal{M}\times\mathcal{M}$ represents an iterated application of the Hardy-Littlewood maximal function, first for $F(\zeta,\chi)$ with respect to $\zeta$ and second for $\mathcal{M}(F(\cdot,\chi))(\xi)$ with respect to $\chi$.
    Therefore, by the $L^p$ boundedness of maximal function, we have
    \begin{align*}
        II_{3,1} & \lesssim \int_{\hn}\int_{\hn} (\mathcal{M}\times\mathcal{M})^p(F)(\xi,\eta)\,d\xi\,d\eta \\
        &\lesssim \int_{\hn}\int_{\hn} \gaglifrac \mathbbm{1}_{\Omega}(\zeta)\mathbbm{1}_\Omega(\chi)\,d\xi\,d\eta\\
        &=\intgaglifrac[u]{\Omega} \leq  \|u\|^p_{W^{s,p}(\Omega)}.
    \end{align*}
    \paragraph{\textbf{Estimate for $II_{3,2}$:}} If $\xi\in V\setminus\Omega$ and $\eta\in D_2(\xi)$, we have by \cref{pouS} that
    \begin{align*}
        |\tilde{E}u(\xi)-\tilde{E}u(\eta)| &= \left|\sum_{i\in I_\xi\cup I_\eta}(\varphi_i(\xi)-\varphi_i(\eta))m_{B({\xi^*_i},r_i)\cap\Omega(u)}\right|\\
        &=\left|\sum_{i\in I_\xi\cup I_\eta}(\varphi_i(\xi)-\varphi_i(\eta))(m_{B({\xi^*_i},r_i)\cap\Omega(u)}-m_{B(\xi,6\,\text{dist}(\xi,\Omega))\cap\Omega}(u))\right|\\
        &\lesssim \sum_{i\in I_\xi\cup I_\eta} \frac{\dhn{\xi}{\eta}}{r_i}\left|m_{B({\xi^*_i},r_i)\cap\Omega(u)}-m_{B(\xi,6\,\text{dist}(\xi,\Omega))\cap\Omega}(u)\right|\\
        &\lesssim \sum_{i\in I_\xi\cup I_\eta} \frac{\dhn{\xi}{\eta}}{r_i}\left|\fint_{B({\xi^*_i},r_i)\cap\Omega(u)}\fint_{B(\xi,6\,\text{dist}(\xi,\Omega))\cap\Omega} |u(\zeta)-u(\chi)| \,d\zeta\,d\chi\right|,
    \end{align*} where we have used $\sum_{i\in I_\xi\cup I_\eta} (\varphi(\xi)-\varphi(\eta))=0$ in the second step, and \cref{partuni}~(3) in the third step.
    
    For all $i\in I_\xi \cup I_\eta$, we will now prove that $r_i\sim \text{dist}(\xi,\Omega)$ and
    \begin{align*}
        B(\xi_i^*,r_i)\subset B(\xi,20\,\text{dist}(\xi,\Omega))\mbox{ and } |B(\xi_i^*,r_i)|\sim |B(\xi, 20\,\text{dist}(\xi,\Omega))\cap\Omega|.
    \end{align*} For $i\in I_\xi$, this follows from \cref{cond1} and \cref{cond2}. Now, let $i\in I_\eta$, then it holds that
    \begin{align}
        \label{cond3}
        \frac13\text{dist}(\eta,\Omega) \leq \text{dist}(\xi,\Omega)\leq 3\,\text{dist}(\eta,\Omega).
    \end{align}
    In order to see this, take $\overline{\eta}\in\overline{\Omega}$ so that $\dhn{\eta}{\overline{\eta}}=\text{dist}(\eta,\Omega)$ and by $\dhn{\xi}{\eta}\leq\frac12\max\{\text{dist}(\xi,\Omega),\text{dist}(\eta,\Omega)\}$ we have
    \begin{align*}
        \text{dist}(\xi,\Omega) &\leq \dhn{\xi}{\overline{\eta}} \leq \dhn{\xi}{\eta} + \dhn{\eta}{\overline{\eta}}\\
        & \leq \frac12\text{dist}(\xi,\Omega)+\frac{1}{2}\text{dist}(\eta,\Omega)+\text{dist}(\eta,\Omega)\\
        & = \frac12 \text{dist}(\xi,\Omega) + \frac32 \text{dist}(\eta,\Omega),
    \end{align*} so that it follows that $\text{dist}(\xi,\Omega) \leq 3\,\text{dist}(\eta,\Omega)$. Similarly, we have $\text{dist}(\eta, \Omega)\leq 3\,\text{dist}(\xi,\Omega)$.
    Thus certainly $B(\eta,6\,\text{dist}(\eta,\Omega))\subset B(\xi,20\,\text{dist}(\xi,\Omega))$ and then by \cref{whitney} and \cref{measuredense} the desired claim
    \begin{align*}
        &B(\xi_i^*,r_i)\subset B(\eta,6\,\text{dist}(\eta,\Omega))\subset B(\xi,20\,\text{dist}(\xi,\Omega))\mbox{ and }\\
        &|B(\xi_i^*,r_i)|\sim|B(\eta,6\,\text{dist}(\eta,\Omega))|\sim |B(\xi, 20\,\text{dist}(\xi,\Omega))\cap\Omega|.
    \end{align*}
    is true. Going back to the estimate, using the fact that $\sharp I_{\xi}\cup I_\eta\lesssim 1$ and from the above calculations, we get
    \begin{align*}
        |\tilde{E}u(\xi)-\tilde{E}u(\eta)| &\lesssim \frac{\dhn{\xi}{\eta}}{\text{dist}(\xi,\Omega)}\fint_{B(\xi,20\,\text{dist}(\xi,\Omega))\cap\Omega}\fint_{B(\xi,6\,\text{dist}(\xi,\Omega))\cap\Omega} |u(\zeta)-u(\chi)| \,d\chi\,d\zeta\\
        &\lesssim \frac{\dhn{\xi}{\eta}}{\text{dist}(\xi,\Omega)}\fint_{B(\xi,20\,\text{dist}(\xi,\Omega))\cap\Omega}\left(\fint_{B(\xi,6\,\text{dist}(\xi,\Omega))\cap\Omega} |u(\zeta)-u(\chi)|^p \,d\chi\right)^{\frac1p}\,d\zeta\\
        &\lesssim \frac{\dhn{\xi}{\eta}}{\text{dist}(\xi,\Omega)^{1-s}}\fint_{B(\xi,20\,\text{dist}(\xi,\Omega))\cap\Omega}\left(\int_{B(\xi,6\,\text{dist}(\xi,\Omega))\cap\Omega} \frac{|u(\zeta)-u(\chi)|^p}{\dhn{\zeta}{\chi}^{Q+sp}} \,d\chi\right)^{\frac1p}\,d\zeta\\
        &\lesssim \frac{\dhn{\xi}{\eta}}{\text{dist}(\xi,\Omega)^{1-s}}\mathcal{M}\left(\left(\int_{\Omega} \frac{|u(\cdot)-u(\chi)|^p}{\dhn{\cdot}{\chi}^{Q+sp}} \,d\chi\mathbbm{1}_\Omega(\cdot)\right)^{\frac1p}\right)(\xi),
    \end{align*} where in the second step, we use monotonicity of averages with respect to exponents; in the third step, we use $\dhn{\zeta}{\chi}\lesssim\text{dist}(\xi,\Omega)$ due to $\zeta,\chi\in B(\xi,20\,\text{dist}(\xi,\Omega))\cap\Omega$. As a result, we have
    \begin{align*}
        II_{3,2}\lesssim \int_{V\setminus\Omega}\left(\int_{D_2(\xi)}\frac{\dhn{\xi}{\eta}^{p-Q-sp}}{\text{dist}(\xi,\Omega)^{p-sp}}\,d\eta\right)\times\left[\mathcal{M}\left(\left(\int_{\Omega} \frac{|u(\cdot)-u(\chi)|^p}{\dhn{\cdot}{\chi}^{Q+sp}} \,d\chi\mathbbm{1}_\Omega(\cdot)\right)^{\frac1p}\right)(\xi)\right]^p\,d\xi
    \end{align*}
    For the first factor in the integrand, notice that by \cref{cond3}, we have $D_2(\xi)\subset B(\xi,15\,\text{dist}(\xi,\Omega))$ so that 
    \begin{align*}
        \int_{D_2(\xi)}\frac{\dhn{\xi}{\eta}^{p-Q-sp}}{\text{dist}(\xi,\Omega)^{p-sp}}\,d\eta\lesssim \int_{B(\xi,15\,\text{dist}(\xi,\Omega))}\frac{\dhn{\xi}{\eta}^{p-Q-sp}}{\text{dist}(\xi,\Omega)^{p-sp}}\,d\eta \lesssim 1
    \end{align*} where we make use of polar coordinates as in \cite[Proposition 5.4.4]{Bonfig2007}. Finally by the $L^p$ boundedness of the maximal function,
    \begin{align*}
        II_{3,2} & \lesssim \int_{\hn} \left[\mathcal{M}\left(\left(\int_{\Omega} \frac{|u(\cdot)-u(\chi)|^p}{\dhn{\cdot}{\chi}^{Q+sp}} \,d\chi\right)^{\frac1p}\mathbbm{1}_\Omega(\cdot)\right)(\xi)\right]^p\,d\xi\\
        & \lesssim \int_{\hn} \left[\left(\int_{\Omega} \frac{|u(\xi)-u(\chi)|^p}{\dhn{\xi}{\chi}^{Q+sp}} \,d\chi \right)^p \mathbbm{1}_\Omega(\xi)\right]^{\frac1p}\,d\xi\\
        & = \int_{\Omega} \int_{\Omega} \frac{|u(\xi)-u(\chi)|^p}{\dhn{\xi}{\chi}^{Q+sp}} \,d\chi\,d\xi \leq \|u\|^p_{W^{s,p}(\Omega)}.
    \end{align*}
    Combining all the steps, this finishes the proof of the existence of the extension operator.
     \end{proof}	

\section{A Gagliardo-Nirenberg Interpolation Inequality in Heisenberg Group}\label{sec:gagint} In this section, we prove the Gagliardo-Nirenberg interpolation inequality for Sobolev spaces on the Heisenberg Group. Firstly, we prove the extension domain version of Sobolev inequality using the extension operator whose existence we proved in the previous subsection.
	
	    \begin{lemma}\label{sobolheisen2}
		Let $s\in (0,1)$ and $sp<Q$. Let $B$ be a ball of radius $1$ in $\hn$. Then there exists a constant $C=C(Q,p,s)$ such that
		\begin{align}
			\|u\|_{L^{p^*}(B)}\leq C\|u\|_{W^{s,p}(B)} 
		\end{align} for all $u\in W^{s,p}(B)$ and $p^*=\frac{Qp}{Q-sp}$.
	\end{lemma} 
	
	\begin{proof}
		It was proved in \cite{Vodop1995} that a ball in the Heisenberg group is a so-called $(\varepsilon-\delta)$-domain. As observed in \cite[Introduction]{Koskela2008}, if $\Omega$ is a $\varepsilon-\delta$-domains then the measure density condition \cref{measuredense} is satisfied for all $x\in\overline{\Omega}$ and all $0<r<\delta$. As a result, a ball in the Heisenberg group $\hn$ is an extension domain for fractional Sobolev spaces. Hence for any ball $B\subset \hn$ and a function $u\in W^{s,p}(B)$, there exists  $\tilde{u}\in W^{s,p}(\hn)$ such that $\|\tilde{u}\|_{W^{s,p}(\hn)}\lesssim \|u\|_{W^{s,p}(B)}$.
		Hence
		\begin{align*}
			\|u\|_{L^{p^*}(B)}\leq \|\tilde{u}\|_{L^{p^*}(\hn)} \overset{\eqref{sobolheisen}}{\lesssim} [\tilde{u}]_{W^{s,p}(\hn)}\overset{\cref{extendthm}}{\lesssim} \|{u}\|_{W^{s,p}(B)}.
		\end{align*}
	\end{proof}

	Now, for the proof of the Gagliardo-Nirenberg interpolation inequality, we use the outline of the proof in \cite[Theorem 12.83]{Leoni17} which takes into consideration three separate cases, viz., $sp<Q$, $sp=Q$ and $sp>Q$.

    \begin{theorem}
        Let $1\leq p\leq\infty$, and $r> q\geq1$ satisfy 
        \begin{align}
            s-\frac{Q}p > -\frac{Q}{r},
        \end{align}
        let $\theta\in (0,1)$ be such that
        \begin{align}
            \theta\left(\frac1p-\frac{s}{Q}\right)+\frac{1-\theta}{q}=\frac1r.
        \end{align}
        Then there exists a constant $c=c(Q,s,p,q,\theta)>0$ such that 
        \begin{align}\label{gagliinterpol1}
            \|u\|_{L^r(\hn)}\leq c\|u\|^{1-\theta}_{L^q(\hn)}[u]^\theta_{W^{s,p}(\hn)}
        \end{align} for every $u\in L^q(\hn)\cap W^{s,p}(\hn)$.
    \end{theorem}

\begin{proof} we divide the proof into three cases:
\item
\paragraph{\textbf{Case 1: $sp<Q$}} In this case, we make use of \cref{sobolheisen}. By the H\"older's inequality, we have 
\begin{align}
    \int_{\hn} |u|^r\,d\xi = \int_{\hn} |u|^{\theta r}|u|^{(1-\theta)r}\,d\xi \leq \left(\int_{\hn}|u|^{p^*}\,d\xi\right)^{\frac{\theta r}{p^*}}\left(\int_{\hn}|u|^q\right)^{\frac{(1-\theta)r}{q}} 
\end{align} for any $\theta\in [0,1]$ such that
\begin{align}
    \frac{\theta r}{p^*}+\frac{(1-\theta)r}{q}=1.
\end{align}Then by \cref{sobolheisen} we obtain
\begin{align}
    \|u\|_{L^r(\hn)}\lesssim [u]_{W^{s,p}(\hn)}^\theta\|u\|_{L^q(\hn)}^{1-\theta}
\end{align} where
\begin{align*}
    \frac{1}{p^*}=\frac{1}{p}-\frac{s}{Q}. 
\end{align*}we need to check that the conditions of the theorem imply that $p<p^*$ and $\theta\in (0,1)$. However, $\frac1{p^*}=\frac1{p}-\frac{s}{Q}>0$. Therefore, $\frac{1}{p^*}<\frac1p$, as required. Moreover $\theta(-\frac1p+\frac{s}Q+\frac1q)=\frac1q-\frac1r> 0.$ Also, the condition $            s-\frac{Q}p > -\frac{Q}{r},
$ makes sure that $-\frac1p+\frac{s}Q+\frac1q>\frac1q-\frac1r$ so that $\theta\in (0,1)$ is well-defined.
\item 
\paragraph{\textbf{Case 2: $sp=Q$}} This case is an immediate consequence of \cref{bmoembedsobol} applied to \cref{bmointerpol}.
\item 
\paragraph{\textbf{Case 3: $sp>Q$}} By \cref{interlope1}
        \begin{align*}
            \|u\|_{L^r(\hn)}\lesssim \|u\|_{L^q(\hn)}^{\theta_1} [u]_{C^{0,\alpha}(\hn)}^{1-\theta_1},
        \end{align*} where $\theta_1=\frac{\alpha+Q/r}{\alpha+Q/q}$. Applying \cref{morreyheisen} to the previous display, we obtain
        \begin{align*}
            \|u\|_{L^r(\hn)}\lesssim \|u\|_{L^q(\hn)}^{1-\gamma} [u]_{W^{s,p}(\hn)}^{\gamma},
        \end{align*} where $\gamma=\frac{\frac1q-\frac1r}{\frac{s}{Q}-\frac1p+\frac1q}$.
\end{proof}

Due to the extension theorem proved in the previous section, we have the following estimate

    \begin{theorem}\label{gaginter}
        Let $B\subset\hn$ be a ball of radius $1$. Let $1\leq p\leq\infty$, and $r> q\geq1$ satisfy 
        \begin{align}
            s-\frac{Q}p > -\frac{Q}{r},
        \end{align}
        let $\theta\in (0,1)$ be such that
        \begin{align}\label{indexsetup1}
            \theta\left(\frac1p-\frac{s}{Q}\right)+\frac{1-\theta}{q}=\frac1r.
        \end{align}
        Then there exists a constant $c=c(Q,s,p,q,\theta)>0$ such that 
        \begin{align}\label{gagliinterpol2}
            \|u\|_{L^r(B)}\leq c\|u\|^{1-\theta}_{L^q(B)}\|u\|^\theta_{W^{s,p}(B)}
        \end{align} for every $u\in L^q(B)\cap W^{s,p}(B)$.
    \end{theorem}

    \begin{proof}
        As noted earlier, a ball in the Heisenberg group $\hn$ is an extension domain for fractional Sobolev spaces. Hence for any ball $B\subset \hn$ and a function $u\in W^{s,p}(B)$, there exists  $\tilde{u}\in W^{s,p}(\hn)$ such that $\|\tilde{u}\|_{W^{s,p}(\hn)}\lesssim \|u\|_{W^{s,p}(B)}$.
        Hence
        \begin{align*}
            \|u\|_{L^r(B)}&\leq \|\tilde{u}\|_{L^r(\hn)}\overset{\eqref{gagliinterpol1}}{\lesssim} \|\tilde{u}\|^{1-\theta}_{L^q(\hn)}[\tilde{u}]^\theta_{W^{s,p}(\hn)}\overset{\cref{extendthm}}{\lesssim} \|{u}\|^{1-\theta}_{L^q(B)}[{u}]^\theta_{W^{s,p}(B)}.
        \end{align*}
    \end{proof}

\section*{Acknowledgements} The first author is financially supported by the fellowship granted by CSIR (File No.: 09/0106(13571)/2022-EMR-I).


\begin{thebibliography}{MPPP23b}
	
	\bibitem[AB15]{AzBed2015}
	Jonas Azzam and Jacob Bedrossian.
	\newblock Bounded mean oscillation and the uniqueness of active scalar
	equations.
	\newblock {\em Trans. Amer. Math. Soc.}, 367(5):3095--3118, 2015.
	\newblock \href {https://doi.org/10.1090/S0002-9947-2014-06040-6}
	{\path{doi:10.1090/S0002-9947-2014-06040-6}}.
	
	\bibitem[ABKY11]{Aaltoetal11}
	Daniel Aalto, Lauri Berkovits, Outi~Elina Kansanen, and Hong Yue.
	\newblock John-{N}irenberg lemmas for a doubling measure.
	\newblock {\em Studia Math.}, 204(1):21--37, 2011.
	\newblock \href {https://doi.org/10.4064/sm204-1-2}
	{\path{doi:10.4064/sm204-1-2}}.
	
	\bibitem[AM18]{AdiMali2018}
	Adimurthi and Arka Mallick.
	\newblock A {H}ardy type inequality on fractional order {S}obolev spaces on the
	{H}eisenberg group.
	\newblock {\em Ann. Sc. Norm. Super. Pisa Cl. Sci. (5)}, 18(3):917--949, 2018.
	
	\bibitem[APT22]{APT8}
	Karthik Adimurthi, Harsh Prasad, and Vivek Tewary.
	\newblock Local {H}\" older regularity for nonlocal parabolic $ p $-{L}aplace
	equations.
	\newblock \href {https://arxiv.org/abs/2205.09695} {\path{arXiv:2205.09695}}.
	\newblock Accepted for publication in {\em  Ann. Sc. Norm. Super. Pisa Cl. Sci. (5)}
	
	\bibitem[AYY22]{DachunYang2022}
	Ryan Alvarado, Dachun Yang, and Wen Yuan.
	\newblock A measure characterization of embedding and extension domains for
	{S}obolev, {T}riebel-{L}izorkin, and {B}esov spaces on spaces of homogeneous
	type.
	\newblock {\em J. Funct. Anal.}, 283(12):Paper No. 109687, 71, 2022.
	\newblock \href {https://doi.org/10.1016/j.jfa.2022.109687}
	{\path{doi:10.1016/j.jfa.2022.109687}}.
	
	\bibitem[BGK23]{agnid2023}
	Agnid Banerjee, Prashanta Garain, and Juha Kinnunen.
	\newblock Lower semicontinuity and pointwise behavior of supersolutions for
	some doubly nonlinear nonlocal parabolic {$p$}-{L}aplace equations.
	\newblock {\em Commun. Contemp. Math.}, 25(8):Paper No. 2250032, 23, 2023.
	\newblock \href {https://doi.org/10.1142/S0219199722500328}
	{\path{doi:10.1142/S0219199722500328}}.
	
	\bibitem[BK24]{BK24}
	Sun-Sig Byun and Kyeongbae Kim.
	\newblock A {H}{\"o}lder estimate with an optimal tail for nonlocal parabolic
	p-{L}aplace equations.
	\newblock {\em Annali di Matematica Pura ed Applicata (1923-)},
	203(1):109--147, 2024.
	
	\bibitem[BKL24]{byundeepak2024}
	Sun-Sig Byun, Deepak Kumar, and Ho-Sik Lee.
	\newblock Global gradient estimates for the mixed local and nonlocal problems
	with measurable nonlinearities.
	\newblock {\em Calc. Var. Partial Differential Equations}, 63(2):Paper No. 27,
	48, 2024.
	\newblock \href {https://doi.org/10.1007/s00526-023-02631-2}
	{\path{doi:10.1007/s00526-023-02631-2}}.
	
	\bibitem[BLS21]{brasco2021continuity}
	Lorenzo Brasco, Erik Lindgren, and Martin Str{\"o}mqvist.
	\newblock Continuity of solutions to a nonlinear fractional diffusion equation.
	\newblock {\em Journal of Evolution Equations}, 21(4):4319--4381, 2021.
	\newblock URL:
	\url{https://link.springer.com/article/10.1007/s00028-021-00721-2}.
	
	\bibitem[BLU07]{Bonfig2007}
	A.~Bonfiglioli, E.~Lanconelli, and F.~Uguzzoni.
	\newblock {\em Stratified {L}ie groups and potential theory for their
		sub-{L}aplacians}.
	\newblock Springer Monographs in Mathematics. Springer, Berlin, 2007.
	
	\bibitem[CDG17]{CDG17}
	Jocemar~Q Chagas, Nicolau~ML Diehl, and Patr{\'\i}cia~L Guidolin.
	\newblock Some properties for the {S}teklov averages.
	\newblock {\em arXiv preprint arXiv:1707.06368}, 2017.
	
	\bibitem[Coz17]{C17}
	Matteo Cozzi.
	\newblock Regularity results and {H}arnack inequalities for minimizers and solutions of nonlocal problems: a unified approach via fractional {D}e {G}iorgi classes.
	\newblock {\em Journal of Functional Analysis}, 272(11):4762--4837, 2017.
	
	\bibitem[CW71]{Cw71}
	Ronald~R. Coifman and Guido Weiss.
	\newblock {\em Analyse harmonique non-commutative sur certains espaces
		homog\`enes}, volume Vol. 242 of {\em Lecture Notes in Mathematics}.
	\newblock Springer-Verlag, Berlin-New York, 1971.
	\newblock \'Etude de certaines int\'egrales singuli\`eres.
	
	\bibitem[DCKP16]{DCKP16}
	Agnese Di~Castro, Tuomo Kuusi, and Giampiero Palatucci.
	\newblock Local behavior of fractional p-minimizers.
	\newblock In {\em Annales de l'Institut Henri Poincar{\'e} C, Analyse non
		lin{\'e}aire}, volume~33, pages 1279--1299. Elsevier, 2016.
	
	\bibitem[DiB12]{DiB12}
	Emmanuele DiBenedetto.
	\newblock {\em Degenerate parabolic equations}.
	\newblock Springer Science \& Business Media, 2012.
	
	\bibitem[DKLN25]{diening2025}
	Lars Diening, Kyeongbae Kim, Ho-Sik Lee, and Simon Nowak.
	\newblock Gradient estimates for parabolic nonlinear nonlocal equations.
	\newblock {\em Calc. Var. Partial Differential Equations}, 64(3):Paper No. 98,
	2025.
	\newblock \href {https://doi.org/10.1007/s00526-025-02957-z}
	{\path{doi:10.1007/s00526-025-02957-z}}.
	
	\bibitem[DZZ21]{Ding2021}
	Mengyao Ding, Chao Zhang, and Shulin Zhou.
	\newblock Local boundedness and {H}\"older continuity for the parabolic
	fractional {$p$}-{L}aplace equations.
	\newblock {\em Calc. Var. Partial Differential Equations}, 60(1):Paper No. 38,
	45, 2021.
	\newblock \href {https://doi.org/10.1007/s00526-020-01870-x}
	{\path{doi:10.1007/s00526-020-01870-x}}.
	
	\bibitem[Fef69]{Feff1969}
	Charles~Louis Fefferman.
	\newblock {\em Inequalities for strongly singular convolution operators}.
	\newblock ProQuest LLC, Ann Arbor, MI, 1969.
	\newblock Thesis (Ph.D.)--Princeton University.
	\newblock URL:
	\url{http://gateway.proquest.com/openurl?url_ver=Z39.88-2004&rft_val_fmt=info:ofi/fmt:kev:mtx:dissertation&res_dat=xri:pqdiss&rft_dat=xri:pqdiss:7008362}.
	
	\bibitem[FZ24]{FangZhang2024}
	Yuzhou Fang and Chao Zhang.
	\newblock Regularity theory for nonlocal equations with general growth in the
	{H}eisenberg group.
	\newblock {\em Int. Math. Res. Not. IMRN}, (12):9962--9990, 2024.
	\newblock \href {https://doi.org/10.1093/imrn/rnae072}
	{\path{doi:10.1093/imrn/rnae072}}.
	
	\bibitem[FZZ24]{Fang_Zhang_Zhang_2024}
	Yuzhou Fang, Chao Zhang, and Junli Zhang.
	\newblock Local regularity for nonlocal double phase equations in the
	{H}eisenberg group.
	\newblock {\em Proceedings of the Royal Society of Edinburgh: Section A
		Mathematics}, page 1–37, 2024.
	\newblock \href {https://doi.org/10.1017/prm.2024.89}
	{\path{doi:10.1017/prm.2024.89}}.
	
	\bibitem[GLT25]{garain2025}
	Prashanta Garain, Erik Lindgren, and Alireza Tavakoli.
	\newblock Higher {H}\"older regularity for a subquadratic nonlocal parabolic
	equation.
	\newblock {\em J. Differential Equations}, 419:253--290, 2025.
	\newblock \href {https://doi.org/10.1016/j.jde.2024.11.024}
	{\path{doi:10.1016/j.jde.2024.11.024}}.
	
	\bibitem[HaKT08]{Koskela2008}
	Piotr Haj\l~asz, Pekka Koskela, and Heli Tuominen.
	\newblock Measure density and extendability of {S}obolev functions.
	\newblock {\em Rev. Mat. Iberoam.}, 24(2):645--669, 2008.
	\newblock \href {https://doi.org/10.4171/RMI/551} {\path{doi:10.4171/RMI/551}}.
	
	\bibitem[HKST15]{Shanmugalingam2015}
	Juha Heinonen, Pekka Koskela, Nageswari Shanmugalingam, and Jeremy~T. Tyson.
	\newblock {\em Sobolev spaces on metric measure spaces}, volume~27 of {\em New
		Mathematical Monographs}.
	\newblock Cambridge University Press, Cambridge, 2015.
	\newblock An approach based on upper gradients.
	\newblock \href {https://doi.org/10.1017/CBO9781316135914}
	{\path{doi:10.1017/CBO9781316135914}}.
	
	\bibitem[KW24a]{Kass2024AnalPDE}
	Moritz Kassmann and Marvin Weidner.
	\newblock Nonlocal operators related to nonsymmetric forms, {II}: {H}arnack
	inequalities.
	\newblock {\em Anal. PDE}, 17(9):3189--3249, 2024.
	\newblock \href {https://doi.org/10.2140/apde.2024.17.3189}
	{\path{doi:10.2140/apde.2024.17.3189}}.
	
	\bibitem[KW24b]{Kassweid2024}
	Moritz Kassmann and Marvin Weidner.
	\newblock The parabolic {H}arnack inequality for nonlocal equations.
	\newblock {\em Duke Math. J.}, 173(17):3413--3451, 2024.
	\newblock \href {https://doi.org/10.1215/00127094-2024-0008}
	{\path{doi:10.1215/00127094-2024-0008}}.
	
	\bibitem[KWgZ24]{kumagai2024local}
	Takashi Kumagai, Jian Wang, and Meng {G}e~Zhang.
	\newblock Local boundedness of solutions to parabolic equations associated with
	fractional $p$-{L}aplacian type operators, 2024.
	\newblock \href {https://arxiv.org/abs/2412.03770} {\path{arXiv:2412.03770}}.
	
	\bibitem[Leo17]{Leoni17}
	Giovanni Leoni.
	\newblock {\em A first course in {S}obolev spaces}, volume 181 of {\em Graduate
		Studies in Mathematics}.
	\newblock American Mathematical Society, Providence, RI, second edition, 2017.
	\newblock \href {https://doi.org/10.1090/gsm/181} {\path{doi:10.1090/gsm/181}}.
	
	\bibitem[Lia24a]{Liao2024}
	Naian Liao.
	\newblock H\"older regularity for parabolic fractional {$p$}-{L}aplacian.
	\newblock {\em Calc. Var. Partial Differential Equations}, 63(1):Paper No. 22,
	34, 2024.
	\newblock \href {https://doi.org/10.1007/s00526-023-02627-y}
	{\path{doi:10.1007/s00526-023-02627-y}}.
	
	\bibitem[Lia24b]{Liaomodulus2024}
	Naian Liao.
	\newblock On the modulus of continuity of solutions to nonlocal parabolic
	equations.
	\newblock {\em J. Lond. Math. Soc. (2)}, 110(3):Paper No. e12985, 30, 2024.
	\newblock \href {https://doi.org/10.1112/jlms.12985}
	{\path{doi:10.1112/jlms.12985}}.
	
	\bibitem[LW24]{liao2024timeinsensitivenonlocalparabolicharnack}
	Naian Liao and Marvin Weidner.
	\newblock Time-insensitive nonlocal parabolic {H}arnack estimates, 2024.
	\newblock URL: \url{https://arxiv.org/abs/2409.20097}, \href
	{https://arxiv.org/abs/2409.20097} {\path{arXiv:2409.20097}}.
	
	\bibitem[Min03]{Mingione2003}
	Giuseppe Mingione.
	\newblock The singular set of solutions to non-differentiable elliptic systems.
	\newblock {\em Arch. Ration. Mech. Anal.}, 166(4):287--301, 2003.
	\newblock \href {https://doi.org/10.1007/s00205-002-0231-8}
	{\path{doi:10.1007/s00205-002-0231-8}}.
	
	\bibitem[MPPP23b]{MPPP23}
	Maria Manfredini, Giampiero Palatucci, Mirco Piccinini, and Sergio Polidoro.
	\newblock H{\"o}lder continuity and boundedness estimates for nonlinear
	fractional equations in the {H}eisenberg group.
	\newblock {\em The Journal of Geometric Analysis}, 33(3):77, 2023.
	
	\bibitem[MS79]{Macias1979}
	Roberto~A. Mac\'ias and Carlos Segovia.
	\newblock A decomposition into atoms of distributions on spaces of homogeneous
	type.
	\newblock {\em Adv. in Math.}, 33(3):271--309, 1979.
	\newblock \href {https://doi.org/10.1016/0001-8708(79)90013-6}
	{\path{doi:10.1016/0001-8708(79)90013-6}}.
	
	\bibitem[Nhi96]{Nhieu1996}
	Duy-Minh Nhieu.
	\newblock {\em The extension problem for {S}obolev spaces on the {H}eisenberg
		group}.
	\newblock ProQuest LLC, Ann Arbor, MI, 1996.
	\newblock Thesis (Ph.D.)--Purdue University.
	\newblock URL:
	\url{http://gateway.proquest.com/openurl?url_ver=Z39.88-2004&rft_val_fmt=info:ofi/fmt:kev:mtx:dissertation&res_dat=xri:pqdiss&rft_dat=xri:pqdiss:9713571}.
	
	\bibitem[PP22]{PicCalVar2022}
	Giampiero Palatucci and Mirco Piccinini.
	\newblock Nonlocal {H}arnack inequalities in the {H}eisenberg group.
	\newblock {\em Calc. Var. Partial Differential Equations}, 61(5):Paper No. 185,
	30, 2022.
	\newblock \href {https://doi.org/10.1007/s00526-022-02301-9}
	{\path{doi:10.1007/s00526-022-02301-9}}.

    \bibitem[PT23]{PT23}
	Harsh Prasad and Vivek Tewary.
	\newblock Local boundedness of variational solutions to nonlocal double
              phase parabolic equations.
	\newblock {\em J. Differential Equations}, 351, 2023.
	\newblock \href {https://doi.org/10.1016/j.jde.2022.12.029}
	{\path{doi:10.1016/j.jde.2022.12.029}}.

	
	\bibitem[Str19]{Strom2019}
	Martin Str\"omqvist.
	\newblock Local boundedness of solutions to non-local parabolic equations
	modeled on the fractional {$p$}-{L}aplacian.
	\newblock {\em J. Differential Equations}, 266(12):7948--7979, 2019.
	\newblock \href {https://doi.org/10.1016/j.jde.2018.12.021}
	{\path{doi:10.1016/j.jde.2018.12.021}}.
	
	\bibitem[VG95]{Vodop1995}
	S.~K. Vodop\'yanov and A.~V. Greshnov.
	\newblock On the continuation of functions of bounded mean oscillation on
	spaces of homogeneous type with intrinsic metric.
	\newblock {\em Sibirsk. Mat. Zh.}, 36(5):1015--1048, i, 1995.
	\newblock \href {https://doi.org/10.1007/BF02112531}
	{\path{doi:10.1007/BF02112531}}.
	
	\bibitem[Zho15]{Zhou2015}
	Yuan Zhou.
	\newblock Fractional {S}obolev extension and imbedding.
	\newblock {\em Trans. Amer. Math. Soc.}, 367(2):959--979, 2015.
	\newblock \href {https://doi.org/10.1090/S0002-9947-2014-06088-1}
	{\path{doi:10.1090/S0002-9947-2014-06088-1}}.
	
\end{thebibliography}
\end{document}